\documentclass{amsart}

\usepackage[latin1]{inputenc}
\usepackage[T1]{fontenc}
\usepackage{verbatim}
\usepackage{multirow}

\usepackage{geometry}
\usepackage{amssymb}
\usepackage{amsmath}
\usepackage{graphicx}
\usepackage{amsthm}
\usepackage{caption}

\usepackage{color}

\usepackage{enumerate}

\newcommand{\dist}{\operatorname{dist}}

\newcommand{\ddiv}{\operatorname{div}}
\newcommand{\Hess}{\operatorname{Hess}}
\newcommand{\Ric}{\operatorname{Ric}}
\newcommand{\grad}{\operatorname{grad}}
\newcommand{\D}{\Delta}

\def\AS{A_{S_1(x)}}

\newcommand{\NN}{{\mathbb{N}}}
\newcommand{\RR}{{\mathbb{R}}}

\newcommand{\Q}{{\mathcal{Q}}}

\newcommand{\K}{{\mathcal{K}}}
\newcommand{\M}{{\mathcal{M}}}
\newcommand{\FK}{{\mathcal{FK}}}
\newcommand{\N}{{\mathcal{N}}}

\newcommand{\eChar}{\begin{enumerate}[(i)]}
\newcommand{\eCharR}{\begin{enumerate}[(a)]}
\newcommand{\eBr}{\begin{enumerate}[(1)]}

\def\P{\mathcal{P}}

\newcommand{\Abstract}

\newcommand\at[2]{\left.#1\right|_{#2}}


\title
{
Bakry-\'Emery curvature functions of graphs
}

\author{David Cushing}
\address{
Department of Mathematical Sciences\\
Durham University\\
South Road\\
Durham, DH1 3LE\\
United Kingdom}
\email{david.cushing@durham.ac.uk}

\author{Shiping Liu}
\address{
School of Mathematical Sciences\\
University of Science and Technology of China\\
96 Jinzhai Road\\
Hefei 230026\\
Anhui Province\\
China}
\email{spliu@ustc.edu.cn}

\author{Norbert Peyerimhoff}
\address{
Department of Mathematical Sciences\\
Durham University\\
South Road\\
Durham, DH1 3LE\\
United Kingdom}
\email{norbert.peyerimhoff@durham.ac.uk}

\date{\today}

\theoremstyle{plain}
\newtheorem{lemma}{Lemma}[section]
\newtheorem{theorem}[lemma]{Theorem}
\newtheorem{proposition}[lemma]{Proposition}
\newtheorem{corollary}[lemma]{Corollary}

\theoremstyle{definition}

\newtheorem{conjecture}[lemma]{Conjecture}
\newtheorem{definition}[lemma]{Definition}
\newtheorem{remark}[lemma]{Remark}
\newtheorem{remarks}[lemma]{Remarks}

\newtheorem{example}[lemma]{Example}
\newtheorem{examples}[lemma]{Examples}

\newtheorem{problem}[lemma]{Problem}

\numberwithin{equation}{section}

\begin{document}

\maketitle

\pagestyle{plain}

\begin{abstract}
  We study local properties of the Bakry-\'Emery curvature function
  $\K_{G,x}:(0,\infty]\to \mathbb{R}$ at a vertex $x$ of a graph $G$
  systematically. Here $\K_{G,x}(\N)$ is defined as the optimal
  curvature lower bound $\K$ in the Bakry-\'Emery curvature-dimension
  inequality $CD(\K,\N)$ that $x$ satisfies. We provide upper and
  lower bounds for the curvature functions, introduce fundamental
  concepts like curvature sharpness and $S^1$-out regularity, and
  relate the curvature functions of $G$ with various spectral
  properties of (weighted) graphs constructed from local structures of
  $G$. We prove that the curvature functions of the Cartesian product
  of two graphs $G_1,G_2$ are equal to an abstract product of
  curvature functions of $G_1,G_2$. We explore the curvature functions
  of Cayley graphs, strongly regular graphs, and many particular
  (families of) examples. We present various conjectures and construct
  an infinite increasing family of $6$-regular graphs which satisfy
  $CD(0,\infty)$ but are not Cayley graphs.
\end{abstract}

\section{Introduction}

In this section we introduce Bakry-\'Emery curvature and survey the main results of the paper.\\

A fundamental notion in the smooth setting of Riemannian manifolds is
Ricci curvature. This notion has been generalized in various ways to
the more general setting of metric spaces. In this article, we
consider the discrete setting of graphs and study the optimal Ricci
curvature lower bound $\K$ in Bakry-\'Emery's curvature-dimension
inequality $CD(\K,\N)$ at a vertex $x$ of a graph $G$ as a function
of the variable $\N\in (0,\infty]$. Let us start to introduce this
curvature notion which is based on the choice of a Laplace operator
and which has been studied extensively in recent years (see, e.g.,
\cite{Schmuckenschlager98,LY10,JL14,KKRT16,CLY14,LP14,HL15,HL16}).

Let $G=(V,E)$ be a locally finite simple graph (that is, no loops and
no multiple edges) with vertex set $V$ and edge set $E$. For any
$x,y\in V$, we write $x\sim y$ if $\{x,y\}\in E$. Let
$d_x:=\sum_{y:y\sim x}1$ be the \emph{degree of $x$}. We say a graph
$G$ is \emph{$d$-regular} if $d_x=d$ for any $x\in V$. Let
$\dist: V \times V \to \N \cup \{0\}$ denote the \emph{combinatorial
  distance function}. For any function $f: V\to \mathbb{R}$ and any
vertex $x\in V$, the \emph{(non-normalized) Laplacian} $\Delta$ is
defined via
\begin{equation}\label{eq:nonnormalised_Laplacian}
\Delta f(x):=\sum_{y,y\sim x}(f(y)-f(x)).
\end{equation}
The notion of a Laplacian can be generalised by intoducing a vertex
measure and edge weights. In this article we will only consider
curvature associated to the non-normalized Laplacian, except
for the final Section \ref{genmeassec}, where we will briefly provide
some additional information about this curvature notion for general
Laplacians.

\begin{definition}[$\Gamma$ and $\Gamma_{2}$ operators]\label{defn:GammaGamma2}
Let $G=(V,E)$ be a locally finite simple graph.
For any two functions $f,g: V\to \mathbb{R}$, we define
\begin{align*}
2\Gamma(f,g)&:=\D(fg)-f\D g-g\D f;\\
2\Gamma_2(f,g)&:=\D\Gamma(f,g)-\Gamma(f,\D g)-\Gamma(\D f,g).
\end{align*}
\end{definition}
We will write $\Gamma(f):=\Gamma(f,f)$ and $\Gamma_2(f,f):=\Gamma_2(f)$, for short.

\begin{definition}[Bakry-\'Emery curvature]\label{defn:BEcurvature} Let $G=(V,E)$ be a locally finite simple graph. Let $\K\in \mathbb{R}$ and $\N\in (0,\infty]$. We say that a vertex $x\in V$ satisfies the \emph{curvature-dimension inequality} $CD(\mathcal{K},\mathcal{N})$, if for any $f:V\to \mathbb{R}$, we have
\begin{equation}\label{eq:CDineq}
\Gamma_2(f)(x)\geq \frac{1}{\N}(\Delta f(x))^2+\K\Gamma(f)(x).
\end{equation}
We call $\K$ a lower Ricci curvature bound of $x$, and $\N$ a dimension parameter. The graph $G=(V,E)$ satisfies $CD(\K,\N)$ (globally), if all its vertices satisfy $CD(\K,\N)$. At a vertex $x\in V$, let $\K(G,x;\N)$ be the largest $\K$ such that (\ref{eq:CDineq}) holds for all functions $f$ at $x$ for a given $\N$. We call $\K_{G,x}(\N):=\K(G,x;\N)$ the \emph{Bakry-\'Emery curvature function} of $x$.
\end{definition}

The reader can find various modifications of this curvature notion in,
e.g., \cite{BHLLMY13, Horn14,FM14,FM15,LMP16,LM16}. It is natural to
ask about the motivation for this curvature. The notion is rooted on
\emph{Bochner's formula}, a fundamental identity in Riemannian
  Geometry. The following remark explains this connection to the
  smooth setting in more detail.

\begin{remark}
Let $(M,\langle \cdot, \cdot \rangle)$ be a  Riemannian manifold of dimension $n$ with
the Laplacian defined via $\Delta = \ddiv \circ \grad \le 0$.

Bochner's formula states for all smooth functions $f \in C^\infty(M)$ that
$$ \frac{1}{2} \Delta \vert \grad\: f \vert^2(x) = \vert \Hess\: f \vert^2(x)
+ \langle \grad\: \Delta f(x), \grad\: f(x) \rangle + \Ric(\grad\: f(x)), $$
where $\Hess$ denotes the Hessian and $\Ric$ denotes the Ricci tensor.
If $\Ric(v) \geq K_x \vert v \vert^2$ for all $v \in T_xM$ and, using the inequality $\vert \Hess\: f \vert^2(x) \geq \frac{1}{n} ( \Delta f(x) )^2$, we obtain
$$ \frac{1}{2} \Delta \vert \grad\: f \vert^2(x) - \langle \grad\: \Delta f(x), \grad\: f(x) \rangle \geq \frac{1}{n} ( \Delta f(x) )^2 + K_x \vert 
\grad\: f(x) \vert^2.$$
The $\Gamma$ and $\Gamma_2$ of Bakry-\'Emery \cite{BE85} for two functions $f,g \in C^\infty(M)$ are defined as
\begin{eqnarray*}
2 \Gamma(f,g) &:=& \Delta(fg) - f \Delta g - g \Delta f = \langle \grad\: f, \grad\: g \rangle, \\
2 \Gamma_2(f,g) &:=& \Delta \Gamma(f,g) - \Gamma(f,\Delta g) - \Gamma(g,\Delta f).
\end{eqnarray*}
Noting that
$$ \Gamma_2(f,f) = \frac{1}{2} \Delta \vert \grad\: f \vert^2 - \langle \grad\: \Delta f, \grad\: f \rangle $$
and by using Bochner's formula, we obtain the inequality
$$\Gamma_2(f,f)(x) \geq  \frac{1}{n} ( \Delta f(x) )^2 + K_x \Gamma(f,f)(x).$$
In conclusion, an $n$-dimensional Riemannian manifold $(M,\langle \cdot,\cdot \rangle)$ with Ricci curvature bounded below by $K_x$ at $x \in M$ satisfies an inequality of the form given in \eqref{eq:CDineq}.  This suggests to use this inequality to define, indirectly, a Ricci curvature notion for a metric space via the help of the Laplacian.
\end{remark}

Before we give a more detailed discussion of the results in this
article, we like to first provide a rough overview with references
to the sections:

\begin{itemize}
\item Section \ref{sec:gammagamma2props}: Properties of $\Gamma$ and $\Gamma_2$,
  by formulating curvature via semidefinite programming
\item Section \ref{props}: General properties of curvature functions
\item Sections \ref{section:ub} and \ref{section:lb}: Upper and lower curvature bounds
\item Section \ref{section:connComp}: Negative curvature at ``bottlenecks''
\item Section \ref{section:Cartesian}: Curvature of Cartesian products
\item Section \ref{gc}: Global $CD(0,\infty)$ conjectures 
\item Section \ref{sec:S1outreg}: Vertex curvature and spectral gaps in the $1$-sphere.
\item Section \ref{section:Cayley}: Curvature of Cayley graphs
\item Section \ref{section:srg}: Curvature of strongly regular graphs
\item Section \ref{genmeassec}: Curvature for graphs with general measures
\end{itemize}

\subsection{Properties of the Bakry-\'Emery curvature function}
In this article, we are particularly interested in the full Bakry-\'Emery curvature functions $\K_{G,x}: (0,\infty]\to \mathbb{R}$ at all vertices $x \in V$ which carry substantially more information than just the global $CD(\K,\N)$ condition. It follows directly from the definition
that $\K_{G,x}$ is monotone non-decreasing. In fact, we study further properties of
this curvature function in Section \ref{props} and show that the function $\K_{G,x}$ is concave, continuous, and $\lim_{\N \to 0} K_{G,x}(\N) = -\infty$  (Proposition \ref{prop:cur_function}). Moreover, there exist constants $c_1(G,x), c_2(G,x)$, depending on the local structure at $x$, such that
\begin{equation} \label{eq:lowuppestK} 
c_1(G,x) - \frac{2d_x}{\N} \le \K_{G,x}(\N) \le c_2(G,x) - \frac{2d_x}{\N}. 
\end{equation}
The curvature function $\K_{G,x}$ is fully determined by the topology
of the $2$-ball $B_2(x) = \{ y \in V: \dist(x,y) \le 2 \}$ centered at
$x \in V$ and Section \ref{section:ub} is concerned with the upper
bound $c_2(G,x)$ in \eqref{eq:lowuppestK} in terms of this local
structure. Introducing, for a vertex $y \in V$, the \emph{out degree}
(with respect to the center $x$)
$$ d_y^{x,+} := | \{ z: z \sim y, \dist(x,z) > \dist(x,y) \} |, $$
the \emph{average out degree} $av_1^+(x)$ of the $1$-sphere $S_1(x) := \{ y \in V: \dist(x,y) = 1 \}$ is defined by
$$ av_1^+(x) = \frac{1}{d_x} \sum_{y \in S_1(x)} d_y^+, $$
and  the constant $c_2(G,)$ in \eqref{eq:lowuppestK} is given by (Theorem \ref{thm:ub} and Definition \ref{defn:Kinfty0})
\begin{equation}\label{eq:uppbdK} 
c_2(G,x) = \K_\infty^0(x) := \frac{3+d_x-av_1^+(x)}{2}. 
\end{equation}
Since, in many cases, the curvature function agrees with this upper bound, we introduce the following terminology:

\begin{definition}[Curvature sharpness] Let $G=(V,E)$ be a locally finite simple graph. Let $\N \in (0,\infty]$. We call a vertex $x \in V$ to be \emph{$\N$-curvature sharp} if $\K_{G,x}(\N)$
agrees with the upper bound given in \eqref{eq:lowuppestK} and \eqref{eq:uppbdK}, that
is 
$$ \K_{G,x}(\N) = \K_\infty^0(x) - \frac{2d_x}{\N}. $$
We call the graph $G$ to be $\N$-curvature sharp, if every vertex $x \in V$ is $\N$-curvature sharp.
\end{definition}  

We will show that if $x$ is $\N$-curvature sharp, then it is also $\N'$-curvature for all smaller values $\N'$ (Proposition \ref{prop:cur_sharp_before}). Moreover, monotonicity and concavity of $\K_{G,x}$ imply that if $\K_{G,x}(\N_1) = \K_{G,x}(\N_2)$ for $\N_1 < \N_2$, then $\K_{G,x}(\N)$ is constant for all values $\N \ge \N_1$ (Proposition \ref{prop:cur_function_strictly_monotone}).  

A natural question is what local information can be extracted from the curvature function $\K_{G,x}$. We show that the degree $d_x$ can be read off via (Corollary \ref{cor:read_cur_dx})
$$ d_x = -\frac{1}{2} \lim_{\N \to 0} \N \K_{G,x}(\N), $$
and the average out-degree $av_1^+(x)$ can be read off via (Corollary \ref{cor:read_cur_av1})
$$ av_1^+ = 3 + d_x - 2 \lim_{\N \to 0} \left( \K_{G,x}(\N) + \frac{2d_x}{\N} \right). $$ 

Proposition \ref{prop:cur_function} also provides the lower curvature
bound $c_1(G,x) = \K_{G,x}(\infty)$ in \eqref{eq:lowuppestK}. Section
\ref{section:lb} provides another lower curvature bound in terms of
the upper bound \eqref{eq:uppbdK} and a correction term given by the
(non-positive) smallest eigenvalue of a specific matrix
$\widehat{\P}_{\N}$ (Theorem \ref{thm:lb}).

\subsection{Curvature of $S_1$-out regular vertices}

Now we introduce a certain homogeneity property within the $2$-ball
of a vertex, called \emph{$S_1$-out regularity}. It turns out that
this notion is closely linked to the curvature sharpness introduced
above.

\begin{definition}[$S_1$-out regularity]
We say a locally finite simple graph $G$
is \emph{$S_1$-out regular} at $x$, if all vertices in $S_1(x)$ have the same out degree.
\end{definition}

We have the following surprising characterization: A graph $G$ is
$S_1$-out regular at a vertex $x$ if and only if there exists
$\N \in (0,\infty]$ such that $x$ is $\N$-curvature sharp (Corollary
\ref{cor:Outregular_characterization}). That is, the curvature
function $\K_{G,x}$ assumes the upper bound $\K_\infty^0(x) - 2d_x/\N$
for some $\N \in (0,\infty]$ if and only if the local structure around
$x$ is homogeneous in the sense of $S_1$-out regularity. Moreover,
when $G$ is $S_1$-out regular at $x$, there exists a threshold
$\N_0(x)$ such that $x$ is $\N$-curvature sharp for any
$\N \in(0,\N_0(x)]$, and $\K_{G,x}(\N) \equiv \K_{G,x}(\N_0(x))$ for
any $\N \in [\N_0(x),\infty]$ (Theorem \ref{thm:OutRegularFormula}).

\subsection{Curvature and local connectedness}

\begin{figure}[h]
    \centering
    \includegraphics[height=2cm]{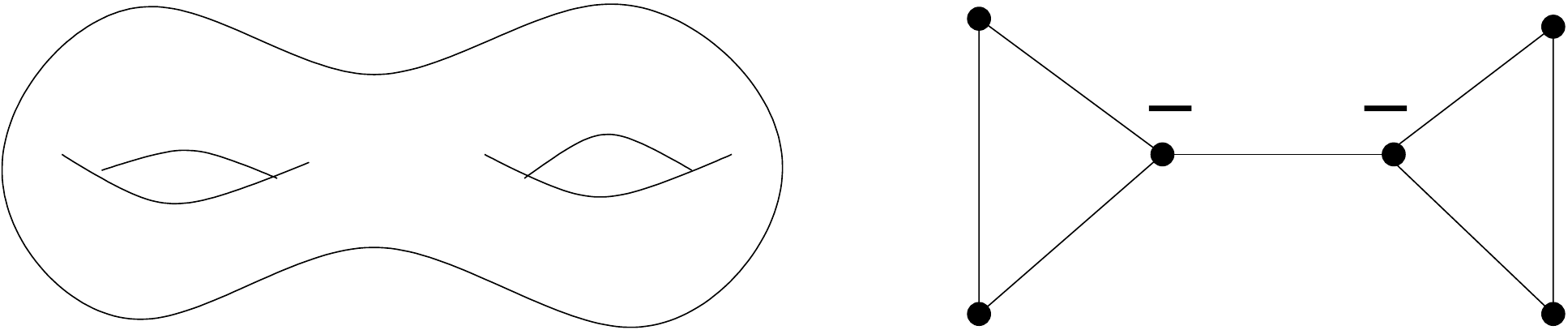}
    \caption{Negative curvature at a ``bottleneck''\label{bottleneck}}
  \end{figure}

  A well know phenomenon in the setting of Riemannian surfaces is that
  ``bottlenecks'' generate regions of negative curvature. A similar
  phenomenon occurs in the graph setting (the two vertices at the
  bottleneck in Figure \ref{bottleneck} have strictly negative
  curvature functions). More generally, the curvature
  $\K_{G,x}(\infty)$ is -- with very few exceptions -- always
  negative, if the \emph{punctured $2$-ball}
  $\mathring{B}_2(x) = B_2(x) - \{ x \}$ has more than one connected
  component (Theorem \ref{thm:conncomp}). Here $\mathring{B}_2(x)$
  denotes the subgraph containing all spherical edges of $S_1(x)$ and
  all radial edges between $S_1(x)$ and the $2$-sphere $S_2(x)$ (but
  not the radial edges of $S_2(x)$, since they have no influence on
  the curvature function at $x$). Further relations between the
  curvature and the local structure of $\mathring{B}_2(x)$ will be
  presented in Section \ref{section:connComp}. In the specific case of
  an $S_1$-out regular vertex $x \in V$ with $\mathring{B}_2(x)$ having more
  than one connected component, we derive in Section
  \ref{sec:S1outreg} the explicit expression
  $\K_{G,x}(\infty) = (3-d_x-av_1^+)/2$, which is negative as soon as
  $d_x+av_1^+ > 3$ (Corollary \ref{cor:curvS1out}). This follows from
  a precise formula for the curvature function $\K_{G,x}$ in terms of
  the spectral gap of a weighted graph $S_1''(x)$ of size $d_x$
  constructed from $\mathring{B}_2(x)$ (Theorem \ref{thm:curvs1out})
  in the specific case of an $S_1$-regular vertex $x$.

\subsection{Curvature of Cartesian products}

In Section \ref{section:Cartesian}, we discuss curvature functions of Cartesian
products. Let $G_i = (V_i,E_i)$ be two locally finite simple graphs
and let $x \in V_1, y \in V_2$. Then the curvature function $\K_{G_1 \times G_2,(x,y)}$,
$(x,y) \in V_1 \times V_2$ is gvien by
$$ \K_{G_1 \times G_2,(x,y)}(\N) = \K_{G_1,x}(\N_1) = \K_{G_2,y}(\N_2), $$
where $\N_1, \N_2 > 0$ are chosen such that $\N = \N_1 + \N_2$ and
$\K_{G_1,x}(\N_1) = \K_{G_2,y}(\N_2)$. We say that
$\K_{G_1 \times G_2,(x,y)}$ is the $*$-product of $\K_{G_1,x}$ and
$\K_{G_2,y}$. This $*$-product (Definition \ref{defn:star_operation}) of the curvature
functions is a well-defined abstract product and interesting in its
own right. In Figure \ref{figure:Cur_fct}, we illustrate the curvature
functions of the complete graphs $K_2$ and $K_3$ and their Cartesian
product $K_2 \times K_3$ (Example \ref{example:LKmn}). Observe that
$\K_{K_2}(4) = \K_{K_3}(4) = 3/2$. As illustrated in Figure
\ref{figure:Cur_fct_2d}, this tells us that
$\K_{K_2 \times K_3}(8) = 3/2$.

\begin{figure}[h]
\begin{minipage}[t]{0.45\linewidth}
\centering
\includegraphics[width=\textwidth]{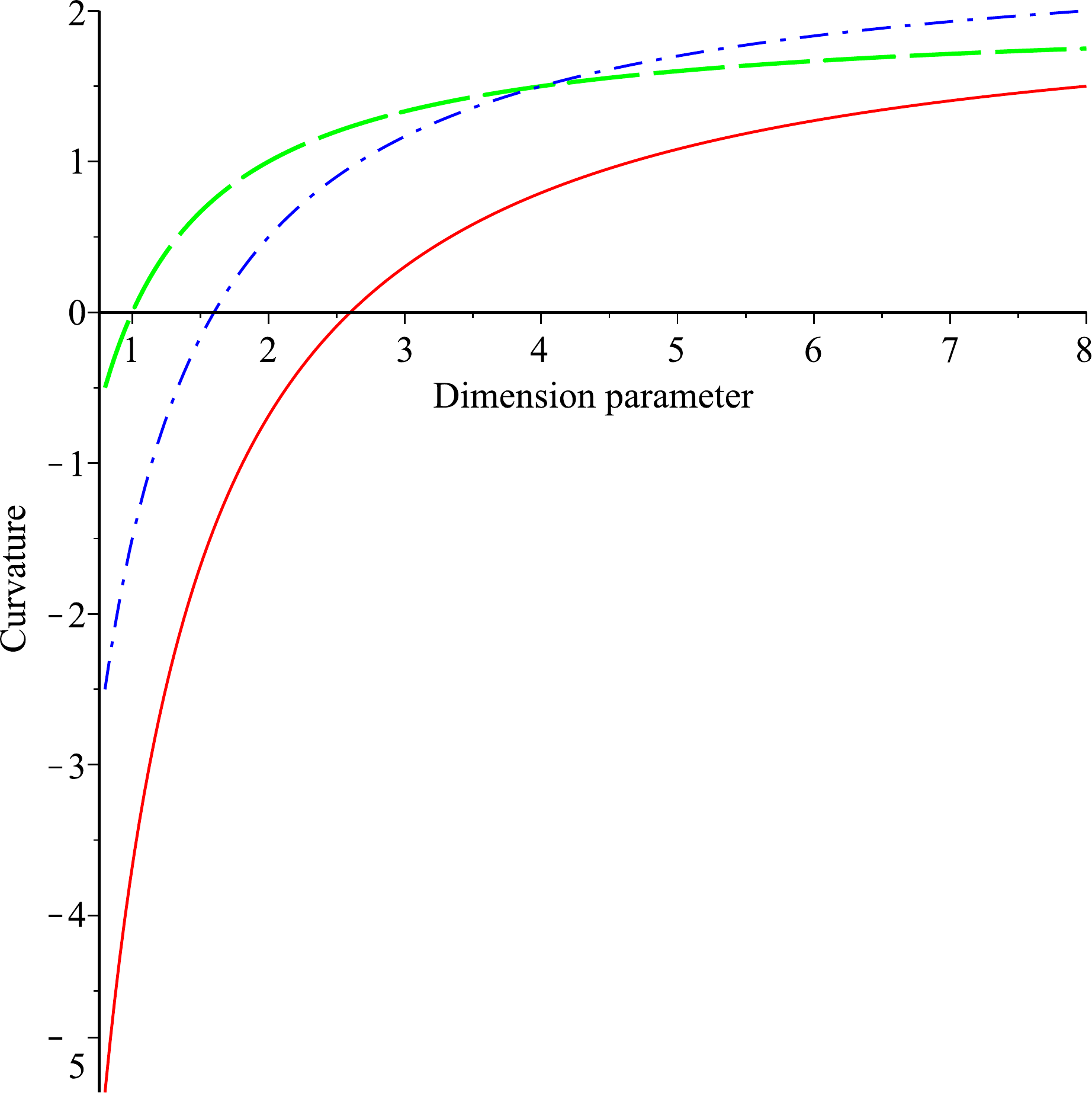}
\captionsetup{width=\linewidth}
\caption{Curvature functions of $\K_{K_2}(\N)$ (dashed), $\K_{K_3}(\N)$ (dashdotted), and $\K_{{K_2\times K_3}}(\N)$ (solid) when $\N\in [0.8,8]$.\label{figure:Cur_fct}}
\end{minipage}
\hfill
\begin{minipage}[t]{0.45\linewidth}
\centering
\includegraphics[width=\textwidth]{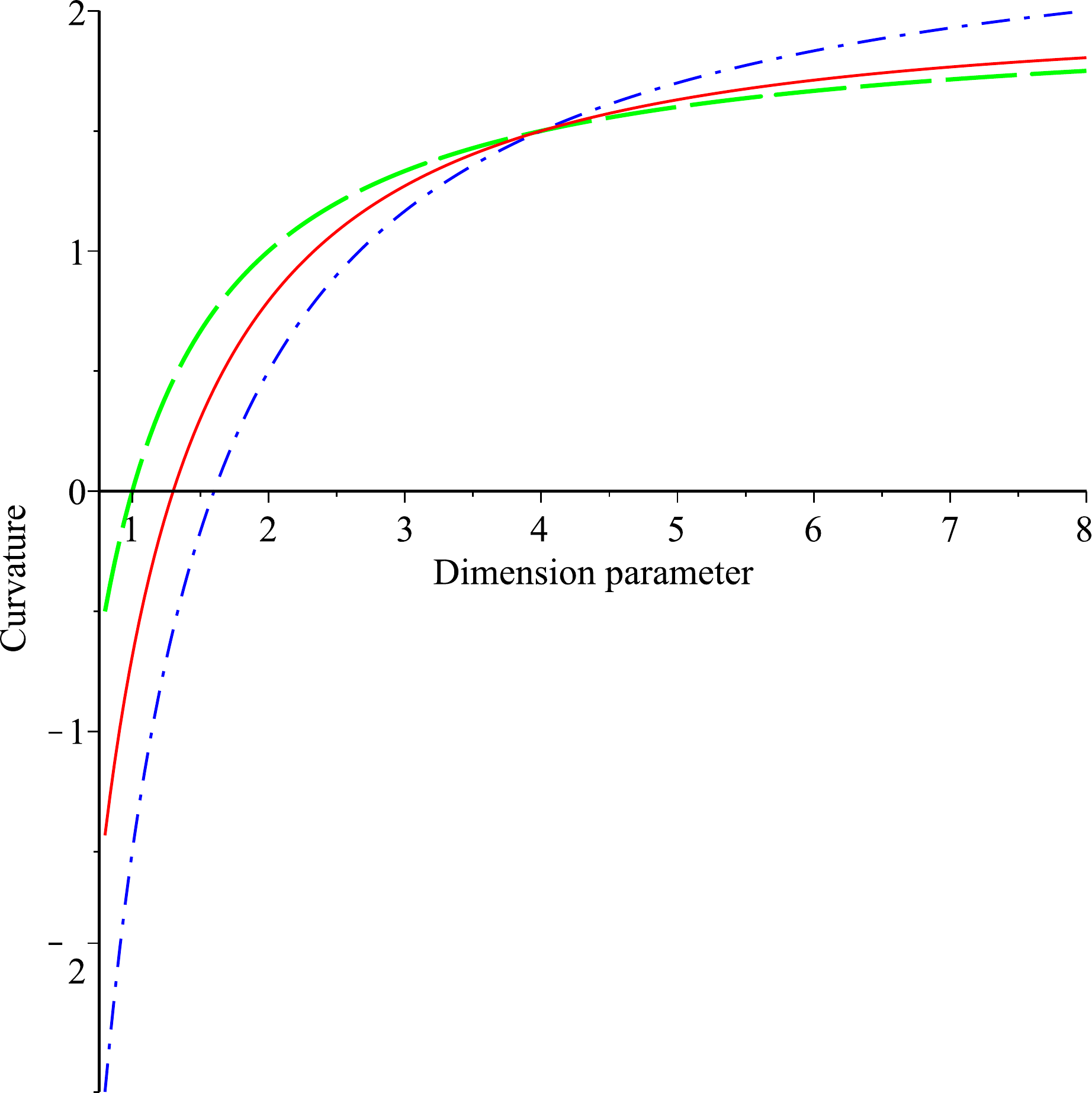}
\captionsetup{width=\linewidth}
\caption{Curvature functions of $\K_{K_2}(\N)$ (dashed), $\K_{K_3}(\N)$ (dashdotted), and $\K_{{K_2\times K_3}}(2\N)$ (solid) when $\N\in [0.8,8]$.\label{figure:Cur_fct_2d}}
\end{minipage}
\end{figure}

\subsection{Curvature of specific graph families}

We calculate curvature functions of many particular graph families, e.g.,
\begin{itemize}
\item paths and cycles, star graphs and regular trees, complete
  graphs, complete bipartite graphs, and crown graphs in Section
  \ref{section:families_examples},
\item hypercubes and line graphs of bipartite graphs in Section \ref{subsec:examples},
\item and Johnson graphs in Example \ref{example:Johnson}.
\end{itemize}

In Section \ref{section:Cayley}, we study curvature functions of \emph{Cayley graphs}. It
is well know that all abelian Cayley graphs satisfy the $CD(0,\infty)$ condition. In
Theorem \ref{thm:coxcurv}, we give a direct relation between the curvature
function of the Cayley graph of a Coxeter group with standard generators (which can be huge;
for example, the Cayley graph of $E_8$ is of size
$192 \cdot 10!$) and the maximal eigenvalue of the Laplacian of the
corresponding Coxeter diagram (which is usually very small
in comparison; in the example $E_8$ of size $8$). 

It is not known and a very interesting question whether there exist
infinite expander families in the class $CD(0,\infty)$ or not (Conjecture
\ref{conj:exp}). Such an expander family cannot consist of \emph{abelian} Cayley
graphs. In Example \ref{example:family_non_Cay} we construct an
infinite family of increasing $6$-regular non-Cayley graphs satisfying
$CD(0,\infty)$, but it is easy to see that this class is not a family
of expanders.

Another interesting class of graphs, studied in Section \ref{section:srg}, are
\emph{strongly regular graphs}. A strongly regular graph $G$ comes with a
family of parameters $(N,d,\alpha,\beta)$, determining its girth
(which can only by $3,4$ or $5$). The curvature functions of strongly
regular graphs with girth $4$ or $5$ are very simple and only depend
on the parameter $d$ (Corollary \ref{cor:srglb}). In the girth $3$ case, the
parameters $(d,\alpha,\beta)$, together with the spectrum of the
adjacency matrix of $1$-sphere $S_1(x)$, determine the curvature
function $\K_{G,x}$ (Theorem \ref{thm:srgCur}), but there are examples of
strongly regular graphs (Shrikhande and $4 \times 4$ rook's graph)
with the same parameters $(N,d,\alpha,\beta)$ and different curvature
functions (Example \ref{example:shrikhande_rook}). The curvature results suggest that all
strongly regular graphs with girth $3$ should satisfy $CD(2,\infty)$
(Conjecture \ref{conj:srg}), but at present we are not even able to prove that
they satisfy $CD(0,\infty)$.

\subsection{Return to Riemannian manifolds}

Readers not interesting in the Riemannian manifold case can safely
skip this subsection, which draws a comparison between curvature
functions of graphs with that of weighted Riemannian manifolds.

A \emph{weighted Riemannian manifold} is a triple
$(M^n,g,e^{-f}d\mathrm{vol}_g)$, where $(M^n,g)$ is an $n$-dimensional
complete Riemannian manifold, $d\mathrm{vol}_g$ is the Riemannian
volume element, and $f$ is a smooth real valued function on $M^n$. The
$\N$-dimensional weighted Ricci tensor of
$(M^n,g,e^{-f}d\mathrm{vol}_g)$ is
\begin{equation}\label{eq:RicciTensor}
\mathrm{Ric}_f(\N):=\mathrm{Ric}+\mathrm{Hess}f-\frac{df\otimes df}{\N-n},
\end{equation}
where $\mathrm{Ric}$ is the Ricci curvature tensor of $(M^n,g)$, $\mathrm{Hess}f$ is the Hessian of $f$ (\cite{BE85,Qian97}). Using the $f$-Laplacian $\D_f=\D_g-\nabla f\cdot\nabla$, where $\D_g$ is the Laplace-Beltrami operator on $(M^n,g)$, one can define the Bakry-\'Emery curvature-dimension inequality $CD(\K,\N)$ as in Definitions \ref{defn:GammaGamma2} and \ref{defn:BEcurvature}. Then $CD(\K,\N),\N\in (n,\infty]$ holds if and only if $\mathrm{Ric}_f(\N)\geq \K$ (see, e.g., \cite[Section 3]{Bakry}).
Recall $n$ in (\ref{eq:RicciTensor}) is the dimension of the underlying Riemannian manifold. When $f$ is not constant, the lower bound of $\mathrm{Ric}_f(\N)$ tends to $-\infty$ as $\N$ tends to $n$. So, in comparison, a graph can be considered as $0$-dimensional. This is natural in the sense that we are using a difference operator to define the curvature functions of a graph.

Recently, the conditions $\mathrm{Ric}_f(\N)\geq \K$ on $(M^n,g,e^{-f}d\mathrm{vol}_g)$, where $\N<n$, have also been studied in \cite{KM13,Ohta13}. In particular, for $\N\in (-\infty,n)$, $CD(\K,\N)$ holds if and only if $\mathrm{Ric}_f(\N)\geq \K$ (\cite[Remark 2.4]{KM13}, \cite[Theorem 4.10]{Ohta13}). In principle, the curvature functions of a graph studied in this article can also be defined on $(-\infty,0)\cup (0,\infty]$. However, we will restrict ourselves to curvature functions on the interval $(0,\infty]$.

\section{Bakry-\'Emery curvature and local $\Gamma$ and $\Gamma_2$ matrices}
\label{sec:gammagamma2props}

In this section, we view curvature as solution of a \emph{semidefinite
programming problem} and derive upper curvature bounds via higher
multiplicities of the zero eigenvalue of certain matrices. We also
derive some properties of vertices satisfying the $CD(0,\N)$ condition.

\subsection{Fundamental notions}

Henceforth, we use the standard notation $[k] := \{1,2,\dots,k\}$.
Given a vertex $x\in V$, the curvature function $\K_{G,x}$ only
depends on the local structure of the graph around $x$.  We now
prepare the notations describing this local structure.  We denote by
$\mathrm{dist}$ the discrete graph distance. For any $r\in\mathbb{N}$,
the $r$-ball centered at $x$ is defined as
$$B_r(x):=\{y\in V: \mathrm{dist}(x,y)\leq r\},$$
and the $r$-sphere centered at $x$ is
$$S_r(x):=\{y\in V: \mathrm{dist}(x,y)=r\}.$$
Then we have the following decomposition of the $2$-ball $B_2(x)$:
\begin{equation*}
B_2(x)=\{x\}\sqcup S_1(x)\sqcup S_2(x).
\end{equation*}
We call an edge $\{y,z\}\in E$ a \emph{spherical edge} (w.r.t. $x$) if $\mathrm{dist}(x,y)=\mathrm{dist}(x,z)$, and a \emph{radial edge} if otherwise. For a vertex $y\in V$, we define
\begin{align*}
d_y^{x,+}&:=|\{z: z\sim y, \mathrm{dist}(x,z)>\mathrm{dist}(x,y)\}|, \\
d_y^{x,0}&:=|\{z: z\sim y, \mathrm{dist}(x,z)=\mathrm{dist}(x,y)\}|, \\
d_y^{x,-}&:=|\{z: z\sim y, \mathrm{dist}(x,z)<\mathrm{dist}(x,y)\}|.
\end{align*}
In the above, the notation $|\cdot|$ stands for the cardinality of the set. We call $d_y^{x,+}$, $d_y^{x,0}$, and $d_y^{x,-}$ the \emph{out degree}, \emph{spherical degree}, and \emph{in degree} of $y$ w.r.t. $x$. We sometimes write $d_y^+, d_y^0, d_y^-$ for short when the reference vertex $x$ is clear from the context.

By abuse of notion, we use $S_1(x)$ in this article also for the
induced subgraph of the vertices in $S_1(x)$ in a graph $G$; we use
$B_2(x)$ also for the subgraph of $G$ with vertex set $B_2(x)$ and
edge set given by the radial edges connecting $\{x\}$ and $S_1(x)$,
the radial edges between $S_1(x)$ and $S_2(x)$, and the
spherical edges in $S_1(x)$. Note that this graph is not the induced
subgraph of $B_2(x)$ in $G$ (since the spherical edges in $S_2(x)$ are
not included), but this local information is all that is needed for
the calculation of the Bakry-\'Emery curvature function $\K_{G,x}$.
We denote by $\mathring{B}_2(x)$ the subgraph of $B_2(x)$ obtained by
deleting $\{x\}$ and all radial edges connecting $\{x\}$ and $S_1(x)$.

Since at a vertex $x$, both $\Gamma(f,g)(x)$ and $\Gamma_2(f,g)(x)$ are quadratic forms, we can talk about their local matrices. The curvature-dimension inequalities can be reformulated as linear matrix inequalities. Recall the following proposition (see \cite[Proposition 3.10]{LMP16}).
\begin{proposition}[\cite{LMP16}]\label{prop:LMP}
Let $G=(V,E)$ be a locally finite simple graph and let $x\in V$. The Bakry-\'Emery curvature function $\K_{G,x}(\N)$ valued at $\N\in (0,\infty]$
 is the solution of the following semidefinite programming,
\begin{align*}
 &\text{maximize}\,\,\, K\\
&\text{subject to}\,\,\,\Gamma_2(x)-\frac{1}{\N}\Delta(x)^\top\Delta(x)\geq K\Gamma(x),
\end{align*}
\end{proposition}

In the above, the local matrices $\Gamma_2(x)$, $\D(x)$, and $\Gamma(x)$ are matrices of sizes $|V|$ by $|V|$, $1$ by $|V|$, and $|V|$ by $|V|$, respectively. But their non-trivial blocks are of relatively small sizes. For example, the non-trivial block of $\D(x)$ is the one corresponding to vertices $B_1(x)$ given by
\begin{equation}\label{eq:Laplacian}
\Delta(x)=\begin{pmatrix} -d_x & 1 & \cdots & 1 \end{pmatrix},
\end{equation}
It is of size $1$ by $|B_1(x)|=d_x+1$. The non-trivial blocks of $\Gamma(x)$ and $\Gamma_2(x)$ are of sizes $|B_2(x)|$ by $|B_2(x)|$ and $|B_1(x)|$ by $|B_1(x)|$, respectively.
In the remaining part of this paper, we will reserve the notations $\Gamma_2(x)$, $\Delta(x)$, and $\Gamma(x)$ for their non-trivial block. Whenever we write linear combinations of them, we pad matrices of smaller sizes with $0$ entries.

We will discuss the local matrices $\Gamma(x)$ and $\Gamma_2(x)$ in more details in the following subsections.

\subsection{Local $\Gamma$ matrix and its basic properties}
We check by definition that $\Gamma(x)$ is a $|B_1(x)|$ by $|B_1(x)|$ matrix corresponding to vertices in $B_1(x)$ given by
\begin{equation}\label{eq:Gamma}
2\Gamma(x)=\begin{pmatrix}
d_x & -1 & \cdots & -1 \\
-1  & 1 & \cdots & 0 \\
\vdots & \vdots & \ddots &\vdots \\
-1 & 0 & \cdots & 1
\end{pmatrix}.
\end{equation}
The following property is a direct observation. Let us denote by $\mathbf{1}_{d_x+1}:=(1,1,\ldots,1)^\top$ and by $\mathbf{1}_{d_x+1}^{\perp}$ the orthogonal complement of the subspace spanned by $\mathbf{1}_{d_x+1}$ in $\mathbb{R}^{d_x+1}$. For convenience, we will often drop the subindex when no confusion is possible.

\begin{proposition}\label{prop:Gamma}
Let $v\in \mathbb{R}^{d_x+1}$. Then $\Gamma(x)v=0$ if and only if $v=a\mathbf{1}$ for some constant $a\in \mathbb{R}$. Moreover, the smallest eigenvalue $\lambda_{\min}(\at{2\Gamma(x)}{\mathbf{1}^\perp})$ of $2\Gamma(x)$ restricted to the subspace $\mathbf{1}^\perp$
satisfies
\begin{equation}\label{eq:GammaMini}
\lambda_{\min}(\at{2\Gamma(x)}{\mathbf{1}^\perp})\geq 1.
\end{equation}
The equality holds in (\ref{eq:GammaMini}) when $d_x>1$.
\end{proposition}
\begin{proof}
By (\ref{eq:Gamma}), $2\Gamma(x)$ is diagonal dominant and $2\Gamma(x) \mathbf{1}=0$.

For any $v:=(v_0,v_1,\ldots, v_{d_x})^\top\in \mathbf{1}^\perp$, we have $v_0=-\sum_{i=1}^{d_x}v_i$, and therefore,
\begin{equation}
v^\top(2\Gamma(x))v=(d_x+1)v_0^2+\sum_{i=0}^{d_x}v_i^2\geq |v|^2,
\end{equation}
where $|v|$ is the norm of $v$. This shows $\lambda_{\min}(\at{2\Gamma(x)}{\mathbf{1}^\perp})\geq 1$.

When $d_x>1$, there exists $v:=(v_0,v_1,\ldots, v_{d_x})^\top\in \mathbf{1}^\perp$ with $v_0=0$. We can check
$$2\Gamma(x)v=v.$$
Hence in this case $\lambda_{\min}(\at{2\Gamma(x)}{\mathbf{1}^\perp})= 1$.
\end{proof}

\subsection{Local $\Gamma_2$ matrix and its basic properties}

The matrix $\Gamma_2(x)$ is of size $|B_2(x)|\times |B_2(x)|$ with the following structure (\cite[Proposition 3.12]{LMP16})
\begin{equation}\label{eq:Gamma2}
4\Gamma_2(x)=\begin{pmatrix}
(4\Gamma_2(x))_{x,x} & (4\Gamma_2(x))_{x,S_1(x)} & (4\Gamma_2(x))_{x,S_2(x)} \\
(4\Gamma_2(x))_{S_1(x),x} & (4\Gamma_2(x))_{S_1(x),S_1(x)} & (4\Gamma_2(x))_{S_1(x), S_2(x)}\\
(4\Gamma_2(x))_{S_2(x),x} & (4\Gamma_2(x))_{S_2(x), S_1(x)} & (4\Gamma_2(x))_{S_2(x),S_2(x)}
\end{pmatrix}.
\end{equation}
The sub-indices indicate the vertices that each submatrix is corresponding to.
We will omit the dependence on $x$ in the above expressions for simplicity. When we exchange the order of the sub-indices, we mean the transpose of the original submatrix. For example, we have $(4\Gamma_2)_{S_1,x}:=((4\Gamma_2)_{x,S_1})^\top$.

Denote the vertices in $S_1(x)$ by $\{y_1,\ldots, y_{d_{x}}\}$. Then we have
\begin{equation}\label{eq:Gamma2xS1}
(4\Gamma_2)_{x,x}=3d_x+d_x^2,\,\,\,(4\Gamma_2)_{x,S_1}=\begin{pmatrix}
-3-d_x-d_{y_1}^+ & \cdots & -3-d_x-d_{y_{d_x}}^+
\end{pmatrix},
\end{equation}
and
\begin{align}
&(4\Gamma_2)_{S_1,S_1}\notag\\
=&\begin{pmatrix}
 5-d_x+3d_{y_1}^++4d^0_{y_1}& 2-4w_{y_1y_2} & \cdots & 2-4w_{y_1y_{d_x}}\\
 2-4w_{y_1y_2} & 5-d_x+3d_{y_2}^++4d^0_{y_2} & \cdots & 2-4w_{y_2y_{d_x}}\\
 \vdots & \vdots & \ddots & \vdots \\
 2-4w_{y_1y_{d_x}} & 2-4w_{y_2y_{d_x}} & \cdots & 5-d_x+3d_{y_{d_x}}^++4d^0_{y_{d_x}}
\end{pmatrix},\label{eq:Gamma2S1S1}
\end{align}
where we use the notation that for any two vertices $x,y\in V$,
\begin{equation}\label{eq:0_1_edge_weight}
w_{xy}=\begin{cases}
1, & \text{ if $x\sim y$ }\\
0, & \text{otherwise}.
\end{cases}
\end{equation}


Denote the vertices in $S_2(x)$ by $\{z_1,\ldots, z_{|S_2(x)|}\}$. Then we have
\begin{equation}\label{eq:Gamma2xS2}
(4\Gamma_2)_{x,S_2}=\begin{pmatrix}
d_{z_1}^- & d_{z_2}^- & \cdots & d_{z_{|S_2(x)|}}^-
\end{pmatrix}
,
\end{equation}
\begin{equation}\label{eq:Gamma2S1S2}
(4\Gamma_2)_{S_1,S_2}=\begin{pmatrix}
-2w_{y_1z_1} & -2w_{y_1z_2} & \cdots & -2w_{y_1z_{|S_2(x)|}}\\
\vdots & \vdots & \ddots & \vdots \\
-2w_{y_{d_x}z_1} & -2w_{y_{d_x}z_2} & \cdots & -2w_{y_{d_x}z_{|S_2(x)|}}
\end{pmatrix}.
\end{equation}
and
\begin{equation}\label{eq:Gamma2S2S2}
(4\Gamma_2)_{S_2,S_2}=\begin{pmatrix}
d_{z_1}^- & 0 & \cdots & 0 \\
0 & d_{z_2}^- & \cdots & 0 \\
\vdots & \vdots & \ddots & \vdots\\
0 & 0 & \cdots & d^-_{z_{|S_2(x)|}}
\end{pmatrix}.
\end{equation}
Note that each diagonal entry of $(4\Gamma_2)_{S_2,S_2}$ is positive.

Let $A(G)$ be the adjacency matrix of the graph $G$. Then we see
$$(4\Gamma_2)_{S_1,S_2}=-2\cdot A(G)_{S_1,S_2}.$$
In fact, we can decompose the matrix $4\Gamma_2$ as follows:
\begin{align}
&4\Gamma_2=\left(\begin{array}{c|c|c}
0 & 0 & 0\\\hline
0 & -4\at{\D}{S_1(x)} & 0\\\hline
0 & 0 & 0
\end{array}
\right)+
\left(\begin{array}{c|cc}
0 & 0 & 0\\\hline
0 & \multicolumn{2}{c}{\multirow{2}{*}{\raisebox{-1mm}{\scalebox{1}{$-2\at{\D}{\mathring{B}_2(x)}$}}}}  \\
0 &
\end{array}
\right)\notag\\
+&\left(\begin{array}{c|ccc|ccc}
3d_x+d_x^2 & -3-d_x-d_{y_1}^+ & \cdots & -3-d_x-d_{y_{d_x}}^+ & d_{z_1}^- & \cdots & d_{z_{|S_2|}}^-\\\hline
-3-d_x-d_{y_1}^+ & 5-d_x+d_{y_1}^+ & \cdots & 2 & 0 & \cdots & 0\\
\vdots & \vdots & \ddots & \vdots & \vdots &\ddots & \vdots\\
-3-d_x-d_{y_{d_x}}^+  & 2& \dots &5-d_x+d_{y_{d_x}}^+ & 0 & \cdots & 0\\\hline
d_{z_1}^- & 0 & \cdots & 0 & -d_{z_1}^- & \cdots & 0\\
\vdots & \vdots & \ddots & \vdots & \vdots & \ddots & \vdots\\
d_{z_{|S_2|}}^- & 0 & \cdots & 0 & 0 & \cdots & -d_{z_{|S_2|}}^-
\end{array}
\right).\notag
\end{align}
In the above, $\at{\D}{S_1(x)}$ and $\at{\D}{\mathring{B}_2(x)}$ are the non-normalized Laplacian of the subgraphs $S_1(x)$ and $\mathring{B}_2(x)$, respectively.

The following proposition can be checked directly.
\begin{proposition}\label{prop:Gamma2Constant}
For the constant vector $\mathbf{1}\in \mathbb{R}^{|B_2(x)|}$, we have $\Gamma_2(x)\mathbf{1}=0$.
\end{proposition}

\subsection{Multiplicity of zero eigenvalue of $\Gamma_2(x)$ and curvature}
By Proposition \ref{prop:Gamma2Constant}, the multiplicity of zero eigenvalue of matrix $\Gamma_2(x)$ is at least one.
In this subsection, we discuss an interesting relation between this multiplicity and the curvature at $x$.

\begin{theorem}\label{thm:spectrum_curvature}
Let $G=(V,E)$ be a locally finite simple graph and $x\in V$ be a vertex. If the multiplicity of the zero eigenvalue of $\Gamma_2(x)$ is at least $2$, then we have $\K_{G,x}(\infty)\leq 0$.
\end{theorem}


Let us first show the following lemma.
\begin{lemma}\label{lemma:gammagamma2}
Let $v=\begin{pmatrix}
v_1 \\ v_2
\end{pmatrix}$ with $v_1\in \mathbb{R}^{|B_1(x)|}$ and $v_2\in \mathbb{R}^{|S_2(x)|}$ be a non-constant vector such that $\Gamma_2(x)v=0$. Then we have $\Gamma(x)v_1\neq 0$.
\end{lemma}
\begin{proof}
We prove the lemma by contradiction. Assume that $\Gamma(x)v_1=0$. By Proposition \ref{prop:Gamma}, we have
\begin{equation}\label{eq:l281}
v_1=a\mathbf{1}_{|B_1(x)|},\,\,\text{ for some } \,\,a\in \mathbb{R}.
\end{equation}
Let us denote
$$w:=a\mathbf{1}_{|B_2(x)|}-v =\begin{pmatrix}
0 \\ a\mathbf{1}_{|B_1(x)|}-v_2
\end{pmatrix}.$$
Then by Proposition \ref{prop:Gamma2Constant}, we have $\Gamma_2(x)w=0$.

Since the submatrix $(\Gamma_2)_{S_2,S_2}$ is invertible (recall (\ref{eq:Gamma2S2S2})), we conclude that
\begin{equation}\label{eq:l282}
v_2=a\mathbf{1}_{|B_1(x)|}.
\end{equation}
(\ref{eq:l281}) and (\ref{eq:l282}) imply that $v$ is a constant vector, which is a contradiction.
%
\end{proof}

\begin{proof}[Proof of Theorem \ref{thm:spectrum_curvature}]
We argue by contradiction. Suppose that $K:=\K_{G,x}(\infty)>0$. Then we have
\begin{equation}\label{eq:Kassumption}
\Gamma_2(x)-K\Gamma(x)\geq 0.
\end{equation}
By assumption, there exists a non-constant vector
$v=\begin{pmatrix}
v_1 \\ v_2
\end{pmatrix}$ such that
\begin{equation}\label{eq:contradictions}
v^T\Gamma_2(x)v=0.
\end{equation}
Applying Lemma \ref{lemma:gammagamma2}, we obtain from (\ref{eq:Kassumption})
\begin{equation}
v^T\Gamma_2(x)v\geq Kv_1^T\Gamma(x)v_1>0,
\end{equation}
which is a contradiction to (\ref{eq:contradictions}).
\end{proof}

For any $\N\in (0,\infty]$ and $K\in \mathbb{R}$, we denote
\begin{equation}\label{eq:MKN}
M_{K,\N}(x):=\Gamma_2(x)-\frac{1}{\N}\Delta(x)^\top\Delta(x)-K\Gamma(x),
\end{equation}
Observe that $M_{K,\N}(x)\mathbf{1}=0$, and its $(S_2,S_2)$-block, which equals $(\Gamma_2)_{S_2,S_2}$, is invertible.
Therefore, from the above proofs, it is not hard to see analogous result of Theorem \ref{thm:spectrum_curvature} also holds for $M_{K, \N}$.

\begin{theorem}\label{thm:CurDimMultiplicity}
Let $G=(V,E)$ be a locally finite simple graph and $x\in V$ be a vertex. Let $\N\in (0,\infty]$ and $K\in \mathbb{R}$. If the multiplicity of the zero eigenvalue of the matrix $M_{K,\N}$ is at least $2$, then we have $\K_{G,x}(\N)\leq K$.
\end{theorem}

The following result is an immediate consequence of Theorem \ref{thm:CurDimMultiplicity}.

\begin{corollary}\label{cor:preciseCur}
Let $G=(V,E)$ be a locally finite simple graph and $x\in V$ be a vertex. Let $\N\in (0,\infty]$ and $K\in \mathbb{R}$. Then the following are equivalent:
\begin{itemize}
\item[(i)]
$\K_{G,x}(\N)=K$;
\item[(ii)]
The matrix $M_{K,\N}(x)\geq 0$ and the multiplicity of its zero eigenvalue is at least $2$.
\end{itemize}
\end{corollary}
\begin{proof}
(ii) $\Rightarrow$ (i): Since $M_{K,\N}(x)\geq 0$, we have $\K_{G,x}(\N)\geq K$. By Theorem \ref{thm:CurDimMultiplicity}, (ii) implies $\K_\N(G,x)\leq K$. Therefore, we obtain $\K_\N(G,x)=K$.

(i) $\Rightarrow$ (ii): (i) implies $M_{K,\N}\geq 0$ immediately. Assume that the zero eigenvalue of $M_{K,\N}(x)$ has multiplicity $1$. Then the smallest eigenvalue $\lambda_{\min}(\at{M_{K,\N}(x)}{\mathbf{1}^\perp})$ of $M_{K,\N}$ restricted to the space $\mathbf{1}^\perp$ is positive. Let $\lambda_{\max}(\Gamma(x))$ be the maximal eigenvalue of $\Gamma(x)$. We observe that
\begin{equation}\label{eq:contradiction}
M_{K,\N}(x)-\epsilon \Gamma(x)\geq 0, \,\,\text{for any }\,\,0<\epsilon<\frac{\lambda_{\max}(\Gamma(x))}{\lambda_{\min}(\at{M_{K,\N}(x)}{\mathbf{1}^\perp})},
\end{equation}
which is a contradiction to (i).
\end{proof}

\subsection{Vertices satisfying $CD(0,\N)$}\label{sec:graphsCD0N}
In this subsection, we discuss immediate properties of a vertex $x$ satisfying $CD(0,\N)$, by considering two particular principal minors of the matrix $M_{0,\N}(x)$: the determinants of $(x,x)$- and $(B_1(x),B_1(x))$- blocks.
\begin{proposition}\label{prop:ub2}
Let $G=(V,E)$ be a locally finite simple graph and $x\in V$. If $x$ satisfies $CD(0,\N)$, then we have
\begin{equation}\label{eq:ub2}
\frac{4d_x}{d_x+3}\leq \N.
\end{equation}
\end{proposition}

\begin{proof}
By Corollary \ref{cor:preciseCur}, we have $4M_{0,\N}(x)\geq 0$.
From Sylvester's criterion we have
$$(4M_{0,\N}(x))_{x,x}=3d_x+d_x^2-\frac{4d_x^2}{\N}\geq 0.$$
Rearranging we thus obtain (\ref{eq:ub2}).
\end{proof}

Proposition \ref{prop:ub2} has interesting consequences.
\begin{corollary}\label{cor:Nsmaller1}
Let $G=(V,E)$ be a locally finite simple graph and let $x\in V$. Then $\K_{G,x}(\N)<0$ when $\N<1$.
\end{corollary}
\begin{proof}
Let $\N<1$. Assume that $\K_{G,x}(\N)\geq 0$. Then by Proposition \ref{prop:ub2}, we have
$4d_x<3d_x+1$. Hence $d_x<1$, which is impossible.
\end{proof}
\begin{remark}
Corollary \ref{cor:Nsmaller1}  can also be shown by Lichnerowicz type estimate, see \cite[Corollary 6.2]{LMP16}.
\end{remark}
\begin{corollary}
Let $G=(V,E)$ be a locally finite simple graph satisfying $CD(0,\N)$. Let $d_{\max}$ denote the maximum degree taken over all vertices. Then
\begin{itemize}
\item
If $\N\in (0,4)$ then $d_{\max}\leq \frac{3\N}{4-\N}$.
\item
If $\N=4$ then $d_{\max}$ may be arbitrarily large.
\end{itemize}
\end{corollary}
\begin{proof}
The case $\N\in (0,3)$ is obtained directly from Proposition \ref{prop:ub2}.
For $\N=4$, consider the complete graphs $K_{n}$ on $n$ vertices which satisfy $CD(0,4)$ (see, e.g., \cite[Proposition 3]{JL14} or Example \ref{example:Kn}).
\end{proof}

\begin{proposition}\label{prop:aminusboverN}
Let $G=(V,E)$ be a locally finite simple graph. For a vertex $x\in V$, let us denote
$$c_1(x):=\det\left(\Gamma_2(x)_{B_1,B_1}\right)\,\,\text{and}\,\,c_2(x)=\D(x) \mathrm{adj}\left(\Gamma_2(x)_{B_1,B_1}\right)\D(x)^\top,$$
where $\mathrm{adj}(\cdot)$ stands for the adjugate matrix. If $x$ satisfies $CD(0,\N)$, then we have
$$c_1(x)\geq \frac{c_2(x)}{\N}.$$
\end{proposition}
\begin{proof}
By assumption, we have $M_{0,\N}(x)\geq 0$. From Sylvester's criterion we have
$$\det\left(\Gamma_2(x)_{B_1,B_1}-\frac{1}{\N}\D(x)^\top\D(x)\right)\geq 0.$$
Applying the Matrix Determinant Lemma, we obtain
$$\det\left(\Gamma_2(x)_{B_1,B_1}\right)-\frac{1}{\N}\D(x) \mathrm{adj}\left(\Gamma_2(x)_{B_1,B_1}\right)\D(x)^\top\geq 0.$$This completes the proof.
\end{proof}

\begin{example}[$K_{2,6}$]\label{example:K26} Consider the complete bipartite graph $K_{2,6}$. Let $x$ be a vertex with degree $2$. Then we can check
$$\Gamma_2(x)_{B_1,B_1}=\frac{1}{2}\begin{pmatrix}
5 & -5 & -5 \\
-5 & 9 & 1\\
-5 & 1 & 9
\end{pmatrix},\,\,\text{and}\,\,\D(x)=\begin{pmatrix}
-2 & 1 & 1
\end{pmatrix}.
$$
Hence, we have $c_1(x)=0$ and $c_2(x)=20$. By Proposition \ref{prop:aminusboverN}, we know, for any finite $\N > 0$, that $x$ does not satisfy $CD(0, \N)$. However, we will show that the graph $K_{2,6}$ satisfies $CD(0,\infty)$.
\end{example}
In Section \ref{section:ub}, we will derive interesting upper curvature bound by considering the $(\{x\}\sqcup S_2(x), \{x\}\sqcup S_2(x))$-minor of $M_{K,\N}(x)$.

\section{An upper curvature bound}\label{section:ub}
In this section, we derive an upper bound on Bakry-\'Emery curvature
function $\K_{G,x}$ in terms of the local topological structure around
$x$ (Theorem \ref{thm:ub}). Let us denote the average degree and
average out-degree of $S_1(x)$ by
\begin{equation}
av_1(x):=\frac{1}{d_x}\sum_{y\in S_1(x)}d_y,\,\,\text{ and }\,\,av_1^+(x):=\frac{1}{d_x}\sum_{y\in S_1(x)}d_y^+
\end{equation}
We write $d_y^0=\#_{\Delta}(x,y)$ alternatively to emphasis its geometric meaning: It is the number of triangles ($3$-cycles) including the edge $\{x,y\}$. We have
\begin{equation}
\sum_{y\in S_1(x)}d_y^0= \sum_{y\in S_1(x)}\#_{\Delta}(x,y)=:2\#_{\Delta}(x),
\end{equation}
where $\#_{\Delta}(x)$ is the number of triangles including the vertex $x$.

\begin{theorem}\label{thm:ub}
Let $G=(V,E)$ be a locally finite simple graph and let $x\in V$. For $\N\in (0,\infty]$, we have
\begin{equation}\label{eq:ub}
K_{G,x}(\N)\leq 2+\frac{d_{x}-av_1(x)}{2}+\frac{\#_{\Delta}(x)}{d_{x}}-\frac{2d_{x}}{\N}.
\end{equation}
\end{theorem}
\begin{proof}
Let $\K:=\K_{G,x}(\N)$. Then the matrix
$$4M_{\K,\N}(x) = 4\left(\Gamma_2(x)-\frac{1}{\N}\Delta(x)^\top\Delta(x)-\K\Gamma(x)\right)\geq 0.$$
Let $M_{0}$ be the submatrix corresponding to the vertices $\{x\}\sqcup S_2(x)=\{x,z_1, \ldots, z_{|S_2(x)|}\}$.
The matrix $M_0$ has the form
\begin{equation}
M_0=\begin{pmatrix}
d_{x}^{2}+3d_{x}-2\K d_{x}-\frac{4d^{2}_{x}}{\N}  & d_{z_1}^- & d_{z_2}^- & \cdots & d^-_{z_{|S_2(x)|}} \\
d_{z_1}^- & d_{z_1}^- & 0 & \cdots & 0 \\
d_{z_2}^- & 0 & d_{z_2}^- & \cdots & 0 \\
\vdots & \vdots & \vdots & \ddots & \vdots\\
d^-_{z_{|S_2(x)|}} & 0 & 0 & \cdots & d^-_{z_{|S_2(x)|}}
\end{pmatrix}.
\end{equation}
We have, by Sylvester's criterion, $det(M_0)\geq 0.$ Thus
\begin{align}
d_{x}^{2}+3d_{x}-2\K d_{x}-\frac{4d^{2}_{x}}{\N} & \geq \sum_{z\in S_2(x)} d_{z}^-=\sum_{y\in S_1(x)}d_y^+\notag
\\
& = \sum_{y\in S_{1}(x)}d_{y} - d_{x} - 2\#_{\Delta}(x).\label{eq:CalK0}
\end{align}
Rearranging gives (\ref{eq:ub})
as required.
\end{proof}

\begin{definition}[The constant $\K_\infty^0(x)$]\label{defn:Kinfty0}
Let $G=(V,E)$ be a locally finite simple graph. For any vertex $x\in V$, we define
\begin{equation}\label{eq:K0}
\K_\infty^0(x):=2+\frac{d_x-av_1(x)}{2}+\frac{1}{2d_x}\sum_{y\in S_1(x)}\#_{\Delta}(x,y).
\end{equation}
By the calculations in (\ref{eq:CalK0}), we can reformulate $\K_\infty^0(x)$ as follows:
\begin{equation}\label{eq:K0short}
\K_\infty^0(x)=\frac{3+d_x-av_1^+(x)}{2}.
\end{equation}
\end{definition}
In terms of the above definition, we can rewrite (\ref{eq:ub}) as
$$\K_{G,x}(\N)\leq \K_\infty^0(x)-\frac{2d_x}{\N}.$$

\begin{corollary}
Let $G=(V,E)$ be a $d$-regular graph and let $x\in V$. Then
\begin{equation}\label{eq:ubRegularInfinity}
K_{G,x}(\infty)\leq 2+\frac{\#_{\Delta}(x)}{d}=2+\frac{1}{2d}\sum_{y\in S_1(x)}\#_{\Delta}(x,y).
\end{equation}
\end{corollary}

\begin{proof}
Follows immediately from Theorem \ref{thm:ub} from noting that in a regular graph we have $d_x=av_1(x)$.
\end{proof}
\begin{remark}\label{remark:KKRT}
We remark that on a $d$-regular graph,  Klartag, Kozma, Ralli, and Tetali  \cite[Theorem 1.2]{KKRT16} have derived an upper bound of $\K_{G,x}(\infty)$, via a different calculating method from here. Namely, they proved
\begin{equation}
\K_\infty(G,x)\leq 2+\frac{1}{2}\max_{y\in S_1(x)}\#_{\Delta}(x,y).
\end{equation}
We comment that their proof can also produce the stronger estimate (\ref{eq:ubRegularInfinity}).
\end{remark}


Theorem \ref{thm:ub} provides stronger estimates than Proposition \ref{prop:ub2}. For example, we have the following one for regular graphs.

\begin{corollary}
Let $G=(V,E)$ be a $d$-regular graph satisfying $CD(0,\N)$. Then
$$d\leq\frac{\N+\sqrt{\N^{2}+2\N\#_{\Delta}(x)}}{2},\,\,\text{ for every}\,\,x\in V.$$
In particular, if $G$ is triangle free, then
$d\leq\N$.
\end{corollary}

Theorem \ref{thm:ub} also tells cases that the curvature has to be negative. For example, we have the following straightforward consequence.
\begin{corollary}
Let $G=(V,E)$ be a triangle free graph and let $x\in V$. Suppose that
$$av_1(x)>4+d_{x}.$$
Then we have $\K_{G,x}(\infty)<0$.
\end{corollary}
The local structures with negative curvature will be explored further in Section \ref{section:connComp}.

\section{Fundamental properties of curvature functions}\label{props}
In this section, we discuss fundamental properties of the curvature function $\K_{G,x}$.

\begin{proposition}\label{prop:cur_function} Let $G=(V,E)$ be a locally finite simple graph and $x\in V$. Then the curvature function $\K_{G,x}:(0,\infty]\to \mathbb{R}$ has the following properties:
\begin{itemize}
\item[(i)] $\K_{G,x}$ is monotone non-decreasing.
\item[(ii)] $\K_{G,x}$ is continuous.
\item[(iii)] For any $\N\in(0,\infty]$, we have
\begin{equation}\label{eq:shape}
\K_{G,x}(\infty)-\frac{2d_x}{\N}\leq \K_{G,x}(\N)\leq \K_\infty^0(x)-\frac{2d_x}{\N}.
\end{equation}
In particular, $\lim_{\N\to 0}\K_{G,x}(\N)=-\infty$.
\item[(iv)] $\K_{G,x}$ is a concave function.
\end{itemize}

\begin{remarks}
\begin{itemize}
\item[(i)] Note that for a given vertex, $\K_{G,x}(\infty)$ and
  $\K_\infty^0(x)$ are both fixed constant. Hence, (\ref{eq:shape})
  describes a rough shape of the graph of the curvature function
  $\K_{G,x}$.
\item[(ii)] We are grateful to Bobo Hua and also to 
Jim Portegies, who
  independently raised the concavity question for this curvature function.
\end{itemize}
\end{remarks}

\end{proposition}
We first prove the following lemma.
\begin{lemma}\label{lemma:continuity}
For any $0<\N_1\leq \N_2\leq \infty$, we have
\begin{equation}\label{eq:continuity}
\K_{G,x}(\N_2)\leq \K_{G,x}(\N_1)+2d_x\left(\frac{1}{\N_1}-\frac{1}{\N_2}\right)
\end{equation}
\end{lemma}
\begin{proof}
By definition, we have
\begin{align*}
\Gamma_2(x)\geq & \frac{1}{\N_2}\D(x)^\top\D(x)+\K_{G,x}(\N_2)\Gamma(x)\\
=& \frac{1}{\N_1}\D(x)^\top\D(x)-\left(\frac{1}{\N_1}-\frac{1}{\N_2}\right)\D(x)^\top\D(x)+\K_{G,x}(\N_2)\Gamma(x).
\end{align*}
Observe that
\begin{equation}\label{eq:GammaDeltaIdJd}
d_x\cdot 2\Gamma-\D^\top\D=\left(\begin{array}{cccc}
0 & 0 & \cdots & 0\\
0 & \multicolumn{3}{c}{\multirow{3}{*}{\raisebox{1mm}{\scalebox{1}{$d_xI_{d_x}-J_{d_x}$}}}}  \\
\vdots & \\
0 &
\end{array}
\right)
\end{equation}
where $I_{d_{x}}$ is the $d_x$ by $d_x$ identity matrix and $J_{d_{x}}$ is the $d_x$ by $d_x$ matrix whose entries all equal $1$.
Since the matrix in (\ref{eq:GammaDeltaIdJd}) is diagonal dominant, we have
\begin{equation}\label{eq:discrete_differ_continuous}
\D(x)^\top\D(x)\leq 2d_x\Gamma(x).
\end{equation}
Inserting (\ref{eq:discrete_differ_continuous}), we continue the calculation to obtain
\begin{equation*}
\Gamma_2(x)\geq \frac{1}{\N_1}\D(x)^\top\D(x)+\left(\K_{G,x}(\N_2)-2d_x\left(\frac{1}{\N_1}-\frac{1}{\N_2}\right)\right)\Gamma(x).
\end{equation*}
This implies (\ref{eq:continuity}).
\end{proof}
Now we are ready to show Proposition \ref{prop:cur_function}.

\begin{proof}[Proof of Proposition \ref{prop:cur_function}]
(i). The monotonicity is clear from definition.

(ii). By (i) and Lemma \ref{lemma:continuity}, we have, for any $0<\N_1\leq \N\leq\infty$,
$$\K_{G,x}(\N_1)\leq \K_{G,x}(\N)\leq \K_{G,x}(\N_1)+2d_x\left(\frac{1}{\N_1}-\frac{1}{\N}\right).$$
This shows $\K_{G,x}: (0,\infty]\to \mathbb{R}$ is a continuous function.

(iii). The upper bound in (\ref{eq:shape}) is from Theorem \ref{thm:ub} and the lower bound is from Lemma \ref{lemma:continuity} by taking $\N_1=\N$ and $\N_2=\infty$.

(iv). Let $\N_1 < \N_2$ and $\N = \alpha \N_1 + (1-\alpha) \N_2$ for some
$\alpha \in [0,1]$. Let $\K_j = \K_{G,x}(\N_j)$ for $j \in \{1,2\}$. 
We need to show that
\begin{equation} \label{eq:convbeg}
\K_{G,x}(\N) \ge \alpha \K_1 + (1-\alpha) \K_2. 
\end{equation}
It follows from \eqref{eq:CDineq} that we have for any $f:V\to \mathbb{R}$,
$j \in \{1,2\}$,
$$ \N_j \Gamma_2(f)(x) - \N_j \K_j \Gamma(f)(x) \ge (\Delta f(x))^2. $$
This implies that
\begin{equation}\label{eq:conv1} 
\N \Gamma_2(f)(x) - \left( \alpha \N_1 \K_1 + (1-\alpha) \N_2 \K_2 \right)
\Gamma(f)(x)  \ge (\Delta f(x))^2. 
\end{equation}
Recall from the monotonicity of $\K_{G,x}$ that $\K_1 \le \K_2$ and, therefore,
$$ \alpha(1-\alpha)(\N_2-\N_1)(\K_2-\K_1) \ge 0. $$
This transforms straighforwardly into
$$ \alpha \N_1 \K_1 + (1-\alpha) \N_2 \K_2 \ge \left( \alpha \N_1 + (1-\alpha) \N_2 \right) \left( \alpha \K_1 + (1-\alpha) \K_2\right) = \N 
\left( \alpha \K_1 + (1-\alpha) \K_2\right). $$
Plugging this into \eqref{eq:conv1}, using 
$$ \Gamma (f)(x) = \frac{1}{2} \sum_{y,y\sim x}(f(y)-f(x))^2 \ge 0, $$ 
and reversing the original calculations, we end up with
$$ \Gamma_2(f)(x) \ \ge \frac{1}{\N} (\Delta f(x))^2
+ \left( \alpha \K_1 + (1-\alpha) \K_2\right) \Gamma(f)(x). $$
This shows \eqref{eq:convbeg}, finishing the proof.
\end{proof}

Proposition \ref{prop:cur_function} (iii) implies that we can read the degree of $x$ from its curvature funciton.
\begin{corollary}\label{cor:read_cur_dx}
Let $G=(V,E)$ be a locally finite simple graph and $\K_{G,x}:(0,\infty]\to \mathbb{R}$ be the curvature function of $x\in V$. Then
$$d_x=-\frac{1}{2}\lim_{\N\to 0}\N\K_{G,x}(\N).$$
\end{corollary}

The following property tells that, for any curvature function $\K_{G,x}$,  there always exists a threshold $\N_0(x)\in (0,\infty]$ such that $\K_{G,x}$ is strictly monotone increasing on $(0,\N_0(x)]$, and is constant on $[\N_0(x),\infty]$.
\begin{proposition}\label{prop:cur_function_strictly_monotone}
Let $G=(V,E)$ be a locally finite simple graph and $x\in V$. If there exist $\N_1<\N_2$ such that $\K_{G,x}(\N_1)=\K_{G,x}(\N_2)$, then we have
\begin{equation}\label{eq:cur_function_strictly_monotone}
\K_{G,x}(\N)=\K_{G,x}(\N_1)\,\,\,\,\forall\,\,\N\in [\N_1,\infty].
\end{equation}
\end{proposition}

\begin{proof}
First, by monotonicity, we know $\K_{G,x}(\N)$ is constant on $[\N_1,\N_2]$. Let us denote this constant by $K:=\K_{G,x}(\N_1)$. Again, by monotonicity, we have
$\K_{G,x}(\N) \ge K$ for all $\N \ge \N_2$. But the existence of $\N > \N_2$
with $\K_{G,x}(\N) > K$ would contradict to the concavity of $\K_{G,x}$
(w.r.t. the three points $\N_1 < \N_2 < \N$). This finishes the proof.
\end{proof}

Lemma \ref{lemma:continuity} implies the following property on curvature sharpness.
\begin{proposition}\label{prop:cur_sharp_before}
Let $G=(V,E)$ be a locally finite graph and $x\in V$. If $x$ is $\N$-curvature sharp, then $x$ is $\N'$-curvature sharp for any $\N'\in(0,\N]$.
\end{proposition}
\begin{proof}
If $x$ is $\N$-curvature sharp, then $\K_{G,x}(\N)=\K_\infty^0(x)-2d_x/\N$. By Lemma \ref{lemma:continuity}, we obtain, for any $\N'\in(0,\N]$,
$$\K_{G,x}(\N')\geq \K_\infty^0(x)-\frac{2d_x}{\N'}.$$
Recalling the upper bound in Theorem \ref{thm:ub}, we see the above equality holds.
\end{proof}
In particular, an $\infty$-curvature sharp vertex is $\N$-curvature sharp for any dimension $\N\in (0,\infty]$.

\section{Reformulation of semidefinite programming problem and lower curvature bound}\label{section:lb}

In this section, we derive a reformulation (see Theorem \ref{thm:Sharpness}) of the semidefinite programming problem in Proposition \ref{prop:LMP}. This leads to a lower bound of the curvature function $\K_{G,x}$ in terms of the upper bound $\K_\infty^0(x)-\frac{2d_x}{\N}$ and the minimal eigenvalue of a local matrix $\widehat{\P}_{\N}(x)$, which reflects the topological structure of the neighbourhood around $x$. When $G$ is $S_1$-out regular at $x$, our lower bound estimate provides a precise formula for $\K_{G,x}$.

\subsection{Main results without proofs}

We refer the readers for the proofs of the main results presented here
to the next subsection.

\begin{definition}[Matrices $\P_\infty$ and $\widehat{\P}_\infty$]\label{defn:Phat}
Let $G=(V,E)$ be a locally finite simple graph and let $x\in V$.
$\widehat{\P}_{\infty}(x)$ is a $(d_x+1)$ by $(d_x+1)$ matrix defined as
\begin{equation}
\widehat{\P}_\infty(x):=\left(\begin{array}{cccc}
0 & d_{y_1}^+-av_1^+(x) & \cdots & d_{y_{d_x}}^+-av_1^+(x)\\
d_{y_1}^+-av_1^+(x) & \multicolumn{3}{c}{\multirow{3}{*}{\raisebox{-1mm}{\scalebox{1}{$\P_\infty(x)$}}}}  \\
\vdots & \\
d_{y_{d_x}}^+-av_1^+(x) &
\end{array}
\right),
\end{equation}
where $\P_\infty(x)$ is a $d_x$ by $d_x$ matrix corresponding to the vertices in $S_1(x)$ given as follows.
For any $i,j\in [d_x], i\neq j$, we have
\begin{equation}\label{eq:Pij}
(\P_\infty(x))_{ij}:=2-4w_{y_iy_j}-4\sum_{z\in S_2(x)}\frac{w_{y_iz}w_{zy_j}}{d_z^-},
\end{equation}
with $w_{uv}$ as defined in \eqref{eq:0_1_edge_weight}
and, for any $i\in [d_x]$,
\begin{equation}\label{eq:Pii}
(\P_\infty(x))_{ii}:=-\sum_{j\in [d_x],j\neq i}(\P_\infty)_{ij}-(d_{y_i}^+-av_1^+(x)).
\end{equation}
\end{definition}

\begin{remark}
Note that the entry $(\P_\infty(x))_{ij}$ in (\ref{eq:Pij}) is determined by the number of $1$-paths between $y_i$ and $y_j$ (i.e., $w_{y_iy_j}$, which is either $0$ or $1$), and a weighted counting of the $2$-paths between $y_i$ and $y_j$ via vertices in $S_2(x)$. The entry $(\P_\infty(x))_{ii}$ in (\ref{eq:Pii}) is defined such that $$\widehat{\P}_\infty(x)\mathbf{1}=0.$$
By a direct calculation, one can reformulate (\ref{eq:Pii}) as
\begin{equation}
(\P_\infty(x))_{ii}=-2(d_x-1)+3d_{y_i}^++av_1^+(x)+4d_{y_i}^0-4\sum_{z\in S_2(x)}\frac{w_{y_iz}^2}{d_z^-}.
\end{equation}
\end{remark}

\begin{definition}[Matrices $\widehat{\P}_{\N}$]\label{defn:PN}
Let $G=(V,E)$ be a locally finite simple graph and let $x\in V$.
For $\N\in (0,\infty]$, we define
\begin{equation} \label{eq:PN}
\widehat{\P}_{\N}(x):=\widehat{\P}_{\infty}(x)+\frac{4}{\N}\left(\begin{array}{cccc}
0 & 0 & \cdots & 0\\
0 & \multicolumn{3}{c}{\multirow{3}{*}{\raisebox{1mm}{\scalebox{1}{$d_xI_{d_x}-J_{d_x}$}}}}  \\
\vdots & \\
0 &
\end{array}
\right).
\end{equation}
\end{definition}
By definition, we have $\widehat{\P}_\N \mathbf{1}=0$.

Now, we are ready to state our main results.

\begin{theorem}\label{thm:Sharpness}
Let $G=(V,E)$ be a locally finite simple graph and let $x\in V$. Then for any $\N\in~(0,\infty]$, $\K_{G,x}(\N)$ is the solution of the following semidefintie programming,
\begin{align*}
 &\text{maximize}\,\,\, \K_\infty^0(x)-\frac{2d_x}{\N}-\frac{\lambda}{2}\\
&\text{subject to}\,\,\,\widehat{\P}_\N(x)\geq -\lambda\cdot 2\Gamma(x).
\end{align*}

Moreover, the following are equivalent:
\begin{itemize}
\item[(i)]
$\K_{G,x}(\N)=\K_\infty^0(x)-\frac{2d_x}{\N}-\frac{\lambda}{2}$;
\item[(ii)]
The matrix $\widehat{\P}_\N(x)+\lambda\cdot 2\Gamma(x)$ is positive semidefinite and has zero eigenvalue of multiplicity at least $2$.
\end{itemize}
\end{theorem}

Theorem \ref{thm:Sharpness} can be considered as a new version of Proposition \ref{prop:LMP} and Corollary \ref{cor:preciseCur}, in terms of the matrix $\widehat{\P}_\N(x)+\lambda\cdot 2\Gamma(x)$ instead of $4M_{K,\N}$. Note the latter matrix has a larger size.

As a consequence of Theorem \ref{thm:Sharpness}, we have the following lower bound of the curvature function.
\begin{theorem}\label{thm:lb}
Let $G=(V,E)$ be a locally finite simple graph and let $x\in V$, Then for any $\N\in (0,\infty]$, we have
\begin{equation}\label{eq:lb}
\K_{G,x}(\N)\geq \K_\infty^0(x)-\frac{2d_x}{\N}+\frac{1}{2}\lambda_{\min}(\widehat{\P}_\N(x)).
\end{equation}
The above estimate is sharp if and only if the zero eigenvalue of the matrix $$\widehat{\P}_\N(x)-\lambda_{\min}(\widehat{\P}_\N(x))\cdot 2\Gamma(x)$$ has multiplicity at least $2$.
\end{theorem}

Recall in Corollary \ref{cor:read_cur_dx} that we can read the degree $d_x$ from the curvature function $\K_{G,x}$. With the help of Theorem \ref{thm:lb}, we can further read the average out degree $av_1^+(x)$.
\begin{corollary}\label{cor:read_cur_av1}
Let $G=(V,E)$ be a locally finite simple graph and $\K_{G,x}:(0,\infty]\to \mathbb{R}$ be the curvature function of $x\in V$.
Then
$$av_1^+(x)=3+d_x-2\lim_{\N\to 0}\left(\K_{G,x}(\N)+\frac{2d_x}{\N}\right).$$
\end{corollary}
\begin{proof}
Combining the upper bound in Theorem \ref{thm:ub} and the lower bound in Theorem \ref{thm:lb}, we obtian
$$2\K_\infty^0(x)+\lambda_{\min}(\widehat{\P}_\N(x))\leq 2\left(\K_{G,x}(\N)+\frac{2d_x}{\N}\right)\leq 2\K_\infty^0(x)=3+d_x-av_1^+(x).$$
Therefore, it remains to prove $\lim_{\N\to 0}\lambda_{\min}(\widehat{\P}_\N(x))=0$. Since $\widehat{\P}_\N(x)\mathbf{1}=0$, we only need to show for any $v=(v_0,v_1,\ldots,v_{d_x})^\top\in \mathbf{1}^\perp$, $\lim_{\N\to 0}v^\top\widehat{\P}_\N(x)v\geq 0$.

By (\ref{eq:PN}), we calculate
\begin{align*}
v^\top\widehat{\P}_\N(x)v=&2v_0\sum_{i=1}^d(d_{y_i}^+-av_1^+(x))v_i+\begin{pmatrix}
v_1 & \cdots & v_{d_x}
\end{pmatrix}\P_\infty \begin{pmatrix}
v_1 \\ \vdots \\ v_{d_x}
\end{pmatrix}\\
&+\frac{4}{\N}\left[d_x\sum_{i=1}^{d_x}v_i^2-\left(\sum_{i=1}^{d_x}v_i\right)^2\right].
\end{align*}
If $v=(-d,1,\ldots,1)^\top$, we can check from above $v^\top\widehat{\P}_\N(x)v=0$. Otherwise, we have $d_x\sum_{i=1}^{d_x}v_i^2-\left(\sum_{i=1}^{d_x}v_i\right)^2>0$, and therefore $\lim_{\N\to 0}v^\top\widehat{\P}_\N(x)v=\infty$.
\end{proof}



When $G$ is $S_1$-out regular at $x$, we obtain, from Theorem \ref{thm:lb}, a precise formula for calculating the curvature function $\K_{G,x}$.
Note in this case, we have $\P_\infty(x)\mathbf{1}=0$, and, therefore, $\lambda_{\min}(\P_\infty(x))\leq 0$.
\begin{theorem}\label{thm:OutRegularFormula}
Let $G=(V,E)$ be a locally finite simple graph and let $x\in V$.
Assume that $G$ is $S_1$-out regular at $x$, i.e.,
\begin{equation}\label{eq:outRegularity}
d_{y_i}^+=av_1^+(x), \,\,\text{for any}\,\, y_i\in S_1(x).
\end{equation} Then, we have for any $\N\in (0,\infty]$
\begin{equation}\label{eq:OutRegularFormula}
\K_{G,x}(\N)=\K_\infty^0(x)-\frac{2d_x}{\N}+\frac{1}{2}\lambda_{\min}\left(\P_\infty(x)+\frac{4}{\N}(d_xI_{d_x}-J_{d_x})\right).
\end{equation}
More explicitly, we have
\begin{equation}\label{eq:step_fucntion}
\K_{G,x}(\N)=\left\{
              \begin{aligned}
                &\K_{\infty}^0(x)-\frac{2d_x}{\N}, &&\hbox{if $0<\N\leq \N_0(x)$;} \\
                &\K_{\infty}^0(x)-\frac{2d_x}{\N_0(x)}, &&\hbox{if $\N>\N_0(x)$,}
              \end{aligned}
            \right.
\end{equation}
where
\begin{equation}\label{eq:N0}
\N_0(x):=\frac{4d_x}{-\lambda_{\min}\left(\P_\infty(x)\right)}.
\end{equation}
When $\lambda_{\min}\left(\P_\infty(x)\right)=0$, (\ref{eq:N0}) reads as $\N_0(x)=\infty$.
\end{theorem}
\begin{remark}
When $d_x>1$, $\mathbf{1}_{d_x}^\perp\neq \emptyset$, then $\P_\infty(x)\mathbf{1}=0$ implies
\begin{align*}
\frac{1}{2}\lambda_{\min}&\left(\P_\infty(x)+\frac{4}{\N}(d_xI_{d_x}-J_{d_x})\right)
=\min\left\{0, \frac{1}{2}\lambda_{\min}\left(\at{\P_\infty(x)}{\mathbf{1}^\perp}\right)+\frac{2}{\N}d_x\right\}\\
&=\left\{\begin{aligned}
                &0, &&\hbox{if $0<\N\leq \overline{\N}_0(x)$;} \\
                &-\frac{2d_x}{\N_0(x)}+\frac{2d_x}{\N}, &&\hbox{if $\N>\overline{\N}_0(x)$,}
              \end{aligned}
            \right.
\end{align*}
where
$$
\overline{\N}_0(x):=\left\{\begin{aligned}
                &\frac{4d_x}{-\lambda_{\min}\left(\at{\P_\infty(x)}{\mathbf{1}^\perp}\right)}, &&\hbox{if $\lambda_{\min}\left(\at{\P_\infty(x)}{\mathbf{1}^\perp}\right)<0$;} \\
                &\infty, &&\hbox{if $\lambda_{\min}\left(\at{\P_\infty(x)}{\mathbf{1}^\perp}\right)\geq 0$.}
              \end{aligned}
            \right.$$
Since $\lambda_{\min}(\P_\infty(x))=\min\{0,\lambda_{\min}\left(\at{\P_\infty(x)}{\mathbf{1}^\perp}\right)\}$, we have $\N_0(x)=\overline{\N}_0(x)$. Therefore (\ref{eq:step_fucntion}) is a reformulation of (\ref{eq:OutRegularFormula}). When $d_x=1$, we observe $\P_\infty(x)$ and $d_xI_{d_x}-J_{d_x}$ are both one by one zero matrices, and, therefore, (\ref{eq:step_fucntion}) coincides with (\ref{eq:OutRegularFormula}).
\end{remark}

\begin{example}[Leaves]\label{example:leaf}Let $G=(V,E)$ be a locally finite simple graph. Let $x\in V$ be a leaf of $G$, i.e., $d_x=1$. Then, since both $\P_\infty(x)$ and $d_xI_{d_x}-J_{d_x}$ are one by one zero matrices, we have for any $\N\in(0,\infty]$,
\begin{equation}\label{eq:leaf}
\K_{G,x}(\N)=2-\frac{av_1^+(x)}{2}-\frac{2}{\N}.
\end{equation}
\end{example}

Theorem \ref{thm:OutRegularFormula} tells that, when $G$ is $S_1$-out regular at $x$, there always exists $\N_0(x)\in (0,\infty]$, such that $x$ is $\N$-curvature sharp for any $\N\in (0,\N_0(x)]$, and $\K_{G,x}(\N)=\K_{G,x}(\N_0(x))$ is constant for $\N\in [\N_0(x),\infty]$. In fact, this property is a characterization of the $S_1$-out regularity of $x$.

\begin{corollary}\label{cor:Outregular_characterization}
Let $G=(V,E)$ be a locally finite simple graph. Then $G$ is $S_1$-out regular at $x$ if and only if there exists $\N\in(0,\infty]$ such that $x$ is $\N$-curvature sharp.
\end{corollary}
\begin{proof}
Assume that $x$ is $\N$-curvature sharp for $\N\in(0,\infty]$. Then we obtain, by Theorem \ref{thm:Sharpness}, $\widehat{\P}_\N(x)\geq 0$. Therefore, by Sylvester's criterion, we have for any $y_i\in S_1(x)$,
$$\det\begin{pmatrix}
0 & d_{y_i}^+-av_1^+(x) \\
d_{y_i}^+-av_1^+(x) & \P_\infty(x)_{ii}+\frac{4d_x}{\N}-1
\end{pmatrix}=-(d_{y_i}^+-av_1^+(x))^2\geq 0.
$$
This implies $d_{y_i}^+=av_1^+(x)$ for any $y_i\in S_1(x)$.

The other direction is a straightforward consequence of Theorem \ref{thm:OutRegularFormula}.
\end{proof}
Corollary \ref{cor:Outregular_characterization} and Theorem \ref{thm:OutRegularFormula} implies the following characterization.
\begin{corollary}\label{cor:full_sharpness_characterization}
Let $G=(V,E)$ be a locally finite simple graph and $x\in V$. Then $x$ is $\infty$-curvature sharp if and only if $G$ is $S_1$-out regular at $x$ and the matrix $\P_\infty(x)\geq 0$.
\end{corollary}


In the following example, which will play again a role in Example \ref{example:inffam6reg}, we illustrate how the above results can be applied to calculate explicit curvature functions.
\begin{example} \label{ex:norbexample}
  Figure \ref{Fnorbexample} shows a $4$-regular graph $G=(V,E)$ with
  two types of vertices, denoted by $x_1,x_2,x_3$ and
  $y_1,y_2,y_3,y_4$. We will now calculate the curvature functions
  $\K_{G,x_i}$ and $\K_{G,y_j}$ explicitly. The symmetries of the
  graph imply that these functions do not depend on $i$ or $j$ and it suffices
  to calculate $\K_{G,x_1}$ and $\K_{G,y_1}$.

  \begin{figure}[h]
    \centering
    \includegraphics[width=0.4\textwidth]{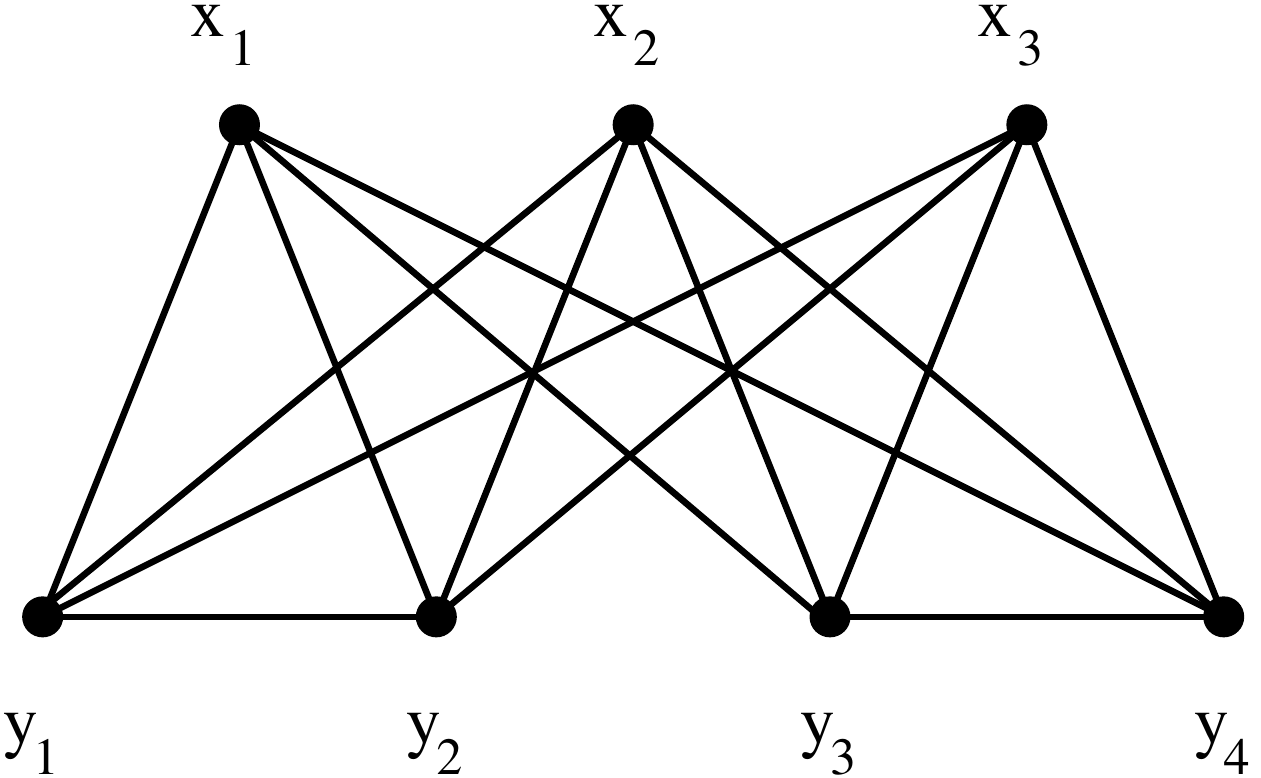}
    \caption{A $4$-regular graph with two types of vertices $x_i$ and
      $y_j$\label{Fnorbexample}}
  \end{figure}

  For the calculation of $\K_{G,x_1}$, we observe that $x_1$ is
  $S_1$-out regular and apply Theorem \ref{thm:OutRegularFormula}.
  We have
  $av_1^+(x_1) = 2$ and \eqref{eq:K0short} yields
  $\K_\infty^0(x_1) = \frac{5}{2}$. The matrix $\P_\infty(x_1)$ takes
  the simple form
  $$ \P_\infty(x_1) = \begin{pmatrix} 4 & -4 & 0 & 0 \\ -4 & 4 & 0 & 0 \\
    0 & 0 & 4 & -4 \\ 0 & 0 & -4 & 4 \end{pmatrix} $$
  and has $\lambda_{\min}(\P_\infty(x_1))=0$ since it is diagonal dominant.
  Therefore
  $\N_0(x_1)$, defined in \eqref{eq:N0}, is equal to infinity and we
  have
  $$ \K_{G,x_1}(\N) = \frac{5}{2} - \frac{8}{\N}. $$

  \begin{figure}[h]
    \centering
    \includegraphics[width=0.4\textwidth]{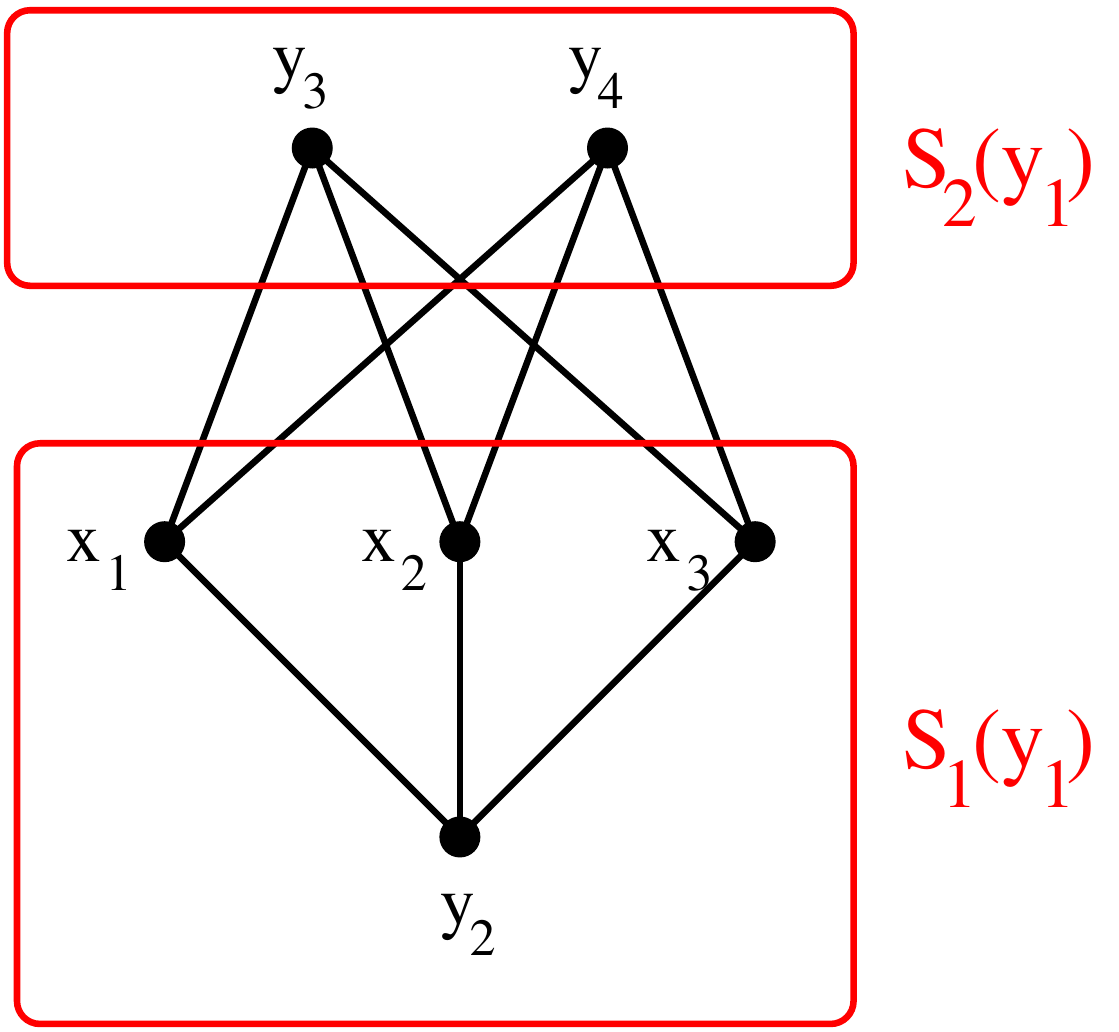}
    \caption{Illustration of $\mathring{B}_2(y_1)$\label{Fnorbex2}}
  \end{figure}

  The calculation of $\K_{G,y_1}$ is much more involved and, since
  $y_1$ is not $S_1$-out regular, we need to employ Theorem
  \ref{thm:Sharpness}. Figure \ref{Fnorbex2} illustrates the punctured
  $2$-ball $\mathring{B}_2(y_1)$. From this we can read off
  $av_1^+(y_1) = 3/2$ and $\K_\infty^0(y_1) = 11/4$. Therefore, we
  have
  \begin{equation} \label{eq:Ky1-norbex}
  \K_{G,y_1}(\N) = \frac{11}{4} - \frac{8}{\N} - \frac{\lambda}{2},
  \end{equation}
  with $\lambda \ge 0$ chosen such that
  $\widehat \P_\N(y_1) + \lambda \cdot 2\Gamma(y_1)$ is positive
  semidefinite with zero eigenvalue of multiplicity at least
  $2$. Recall that $\widehat \P_\N(y_1)$ was defined in \eqref{eq:PN}
  via $\widehat \P_\infty(y_1)$. $\widehat \P_\infty(y_1)$ and
  $2\Gamma(y_1)$, as matrices with entries corresponding to
  $y_1,y_2,x_1,x_2,x_3$, are given by
  $$ \widehat \P_\infty(y_1) = \begin{pmatrix}
  0 & -\frac{3}{2} & \frac{1}{2} & \frac{1}{2} & \frac{1}{2} \\
  -\frac{3}{2} & \frac{15}{2} & -2 & -2 & -2 \\
  \frac{1}{2} & -2 & \frac{17}{6} & -\frac{2}{3} & -\frac{2}{3} \\
  \frac{1}{2} & -2 & -\frac{2}{3} & \frac{17}{6} & \frac{-2}{3} \\
  \frac{1}{2} & -2 & -\frac{2}{3} & -\frac{2}{3} & \frac{17}{6}
  \end{pmatrix}, \quad \text{and} \,\,\, 2\Gamma(y_1) =
  \begin{pmatrix} 4 & -1 & -1 & -1 & -1 \\ -1 & 1 & 0 & 0 & 0 \\
  -1 & 0 & 1 & 0 & 0 \\ -1 & 0 & 0 & 1 & 0 \\ -1 & 0 & 0 & 0 & 1
  \end{pmatrix}. $$
  The eigenvalues of $\widehat \P_\N(y_1) + \lambda \cdot 2\Gamma(y_1)$
  are $0$, $(32 + 2\lambda \N +7\N)/(2\N) > 0$ with multiplicity $2$, and
  $$ \frac{f(\lambda,\N) \pm \sqrt{f^2(\lambda,\N) + g(\lambda,\N)}}{2\N} $$
  with
  $$ f(\lambda,\N) = 6\N \lambda + 9\N + 16 > 0,\ \quad \text{and} \,\,\,
  g(\lambda,\N) = 5\N (-4\N \lambda^2 - (36\N + 64)\lambda + 3\N). $$
  For all eigenvalues to be non-negative and two of them to be $0$,
  we need to have $g(\lambda,\N) = 0$, leading to
  $$ \lambda = \frac{-(9\N+16)+\sqrt{3\N^2 + (9\N+16)^2}}{2\N}. $$
  Plugging this into \eqref{eq:Ky1-norbex}, we obtain
  $$ \K_{G,y_1}(\N) = 5 - \frac{8+\sqrt{21\N^2 + 72\N + 64}}{2\N}. $$
  In particular, we have $\K_{G,x_i}(\infty) = 5/2$ for
  $i \in \{1,2,3\}$ and $\K_{G,y_j}(\infty) = 5 - \sqrt{21}/2 = 2.7087\dots$ for
  $j \in \{1,2,3,4\}$.

  Even though this can also be derived from the symmetries of the
  graph, the fact that the curvatures at $x_i$ and $y_j$ are different
  implies that there is no graph automorphism mapping a vertex $x_i$
  to a vertex $y_j$.
\end{example}

\subsection{Proofs of main results}

In the following, we prove Theorems \ref{thm:Sharpness}, \ref{thm:lb}, and \ref{thm:OutRegularFormula}.
Let us first show the following general result.
\begin{proposition}\label{prop:MatrixVersion1}
Let $M$ be a symmetric square matrix such that
\begin{equation*}
M=\begin{pmatrix}
M_{11} & M_{12}\\
M_{21} & M_{22}
\end{pmatrix},
\end{equation*}
where $M_{11}, M_{22}$ are two square submatrices of orders $m_1, m_2$, respectively, and $M_{22}>0$.
Denote by $\Q(M)$ the matrix
\begin{equation}\label{eq:QM}
\Q(M):=M_{11}-M_{12}M_{22}^{-1}M_{21}.
\end{equation}
Then,
\begin{itemize}
\item[(i)]
$M\geq 0$ if and only if $\Q(M)\geq 0$;
\item[(ii)]
$M$ has zero eigenvalue of multiplicity at least $2$ if and only if $\Q(M)$ has zero eigenvalue of multiplicity at least $2$.
\item[(iii)]
If $M\mathbf{1}_{m_1+m_2}=0$, then $\Q(M)\mathbf{1}_{m_1}=0$.
\end{itemize}

\end{proposition}

\begin{proof}
Since $M_{22}>0$, there exist a matrix $A_0>0$, such that $M_{22}=A_0A_0^\top$. Set $C_0^\top:=~A_0^{-1}M_{21}$. Then we have $A_0C_0^\top=M_{21}$, and
\begin{equation}\label{eq:CCT}
C_0C_0^\top=M_{12}(A_0^{-1})^\top A_0^{-1}M_{21}=M_{12}M_{22}^{-1}M_{21}.
\end{equation}
Therefore, we obtain
\begin{equation}\label{eq:MCAQ}
M=\begin{pmatrix}
C_0 \\ A_0
\end{pmatrix}\begin{pmatrix}
C_0^\top & A_0^\top
\end{pmatrix}+\begin{pmatrix}
\Q(M) & 0 \\
0 & 0
\end{pmatrix}.
\end{equation}

Observing the first matrix on the RHS of (\ref{eq:MCAQ}) is positive semidefinite, we conclude that $\Q(M)\geq 0$ implies $M\geq 0$.

Conversely, if $\Q(M)\not\geq 0$,
then there exists a vector $v$, such that $v^\top \Q(M)v<0$. Set
\begin{equation}\label{eq:w}
w:=-(A_0^\top)^{-1}C_0^\top v.
\end{equation} We calculate
\begin{equation}\label{eq:CAvw=0}
\begin{pmatrix}
C_0^\top & A_0^\top
\end{pmatrix}\begin{pmatrix}
v \\ w
\end{pmatrix}=C_0^\top v+A_0^\top w=0.
\end{equation}
By (\ref{eq:MCAQ}), this implies
\begin{equation*}
\begin{pmatrix}
v^\top & w^\top
\end{pmatrix} M \begin{pmatrix}
v \\ w
\end{pmatrix}=v^\top \Q(M)v<0.
\end{equation*}
Hence, $M\not\geq 0$. This finishes the proof of (i).

Now we prove (ii). Let $v_1,v_2$ be two linearly independent eigenvectors of $\Q(M)$ corresponding to zero. We define the corresponding vectors $w_i,i=1,2$ via (\ref{eq:w}). Due to (\ref{eq:CAvw=0}), we have
$$M\begin{pmatrix}
v_i \\ w_i
\end{pmatrix}=\Q(M)v_i=0,\,\,i=1,2.
$$
That is, we find two independent eigenvectors of $M$ corresponding to eigenvalue zero.

Conversely, let $\begin{pmatrix}
v_i\\w_i
\end{pmatrix}, i=1,2$ be two linearly independent eigenvectors of $M$ corresponding to eigenvalue zero. We have
\begin{equation}\label{eq:gerneral3}
M\begin{pmatrix}
v_i \\ w_i
\end{pmatrix}=0, \,\,\text{ implies }\,\, \Q(M)v_i=0,
\end{equation}
by (\ref{eq:MCAQ}) and
\begin{equation*}
\begin{pmatrix}
C_0^\top & A_0^\top
\end{pmatrix}\begin{pmatrix}
v_i \\ w_i
\end{pmatrix}=A_0^{-1}\begin{pmatrix}
M_{21} & M_{22}
\end{pmatrix}\begin{pmatrix}
v_i \\ w_i
\end{pmatrix}=0.
\end{equation*}
It remains to prove $v_1$ and $v_2$ are linearly independent. Suppose $v_1$ and $v_2$ are linearly dependent. W.o.l.g., we can assume $v_1=v_2$. Then we get
$$M\begin{pmatrix}
0 \\ w_1-w_2
\end{pmatrix}=0, \,\,\text{and, in particular, }\,\, M_{22}(w_1-w_2)=0.
$$
Recall $M_{22}>0$. We have $w_1=w_2$. This contradicts to the fact that $\begin{pmatrix}
v_i\\w_i
\end{pmatrix}, i=1,2$ are linearly independent. We now finish the proof of (ii).

(iii) is a particular case of (\ref{eq:gerneral3}).
\end{proof}



The following lemma is a key observation to apply Proposition \ref{prop:MatrixVersion1} to our situation.

\begin{lemma}\label{lemma:QM=P}
Let $G=(V,E)$ be a locally finite simple graph and let $x\in V$. Then we have
\begin{align}
\widehat{\P}_\infty(x)
=(4\Gamma_2)_{B_1, B_1}-(4\Gamma_2)_{B_1,S_2} (4\Gamma_2)_{S_2, S_2}^{-1}(4\Gamma_2)_{S_2, B_1}-2\K_\infty^0\cdot 2\Gamma.\label{eq:QMK0N}
\end{align}
Note we drop the dependence on $x$ in the RHS of (\ref{eq:QMK0N}) for convenience.
\end{lemma}
\begin{remark}
Let us write
\begin{equation}\label{eq:MK0N}
\Gamma_2-\K_\infty^0\Gamma=\begin{pmatrix}
(\Gamma_2)_{B_1,B_1}-\K_\infty^0 \Gamma & (\Gamma_2)_{B_1,S_2}\\
(\Gamma_2)_{S_2,B_1} & (\Gamma_2)_{S_2,S_2}
\end{pmatrix}
\end{equation}
Then by (\ref{eq:QM}), we can reformulate (\ref{eq:QMK0N}) as
\begin{equation}\label{eq:PinftyQ}
\widehat{\P}_\infty=4\Q(\Gamma_2-\K^0_\infty\Gamma).
\end{equation}
\end{remark}
\begin{proof}

Recall (\ref{eq:Gamma2xS2}), (\ref{eq:Gamma2S1S2}), and (\ref{eq:Gamma2S2S2}). We calculate
\begin{align}
&(4\Gamma_2)_{B_1,S_2} (4\Gamma_2)_{S_2, S_2}^{-1}(4\Gamma_2)_{S_2,B_1}\notag\\
=&\begin{pmatrix}
d_z^- & \cdots & d^-_{z_{|S_2|}}\\
-2w_{y_1z_1} & \cdots & -2w_{y_1z_{|S_2|}}\\
\vdots & \vdots & \vdots \\
-2w_{y_{d_x}z_1} & \cdots & -2w_{y_{d_x}z_{|S_2|}}
\end{pmatrix}
\begin{pmatrix}
\frac{1}{d^-_{z_1}} & \cdots & 0 \\
\vdots & \ddots & \vdots \\
0 & \cdots & \frac{1}{d^-_{z_{|S_2|}}}
\end{pmatrix}\begin{pmatrix}
d_z^- & \cdots & d^-_{z_{|S_2|}}\\
-2w_{y_1z_1} & \cdots & -2w_{y_1z_{|S_2|}}\\
\vdots & \vdots & \vdots \\
-2w_{y_{d_x}z_1} & \cdots & -2w_{y_{d_x}z_{|S_2|}}
\end{pmatrix}^\top\notag\\
=&\begin{pmatrix}
\sum_{z\in S_2}d_z^- & -2d_{y_1}^+ & \cdots & -2d_{y_{d_x}}^+\\
-2d_{y_1}^+ & 4\sum_{z\in S_2(x)}\frac{w_{y_1z}^2}{d^-_z} & \cdots & 4\sum_{z\in S_2(x)}\frac{w_{y_1z}w_{zy_{d_x}}}{d^-_z}\\
\vdots & \vdots & \ddots & \vdots \\
-2d_{y_{d_x}}^+ & 4\sum_{z\in S_2(x)}\frac{w_{y_{d_x}z}w_{zy_1}}{d^-_z} & \cdots & 4\sum_{z\in S_2(x)}\frac{w_{y_{d_x}z}^2}{d^-_z}
\end{pmatrix}.\label{eq:GammaCCT}
\end{align}

On the other hand, using (\ref{eq:Gamma}), (\ref{eq:Gamma2xS1}), (\ref{eq:Gamma2S1S1}), and (\ref{eq:K0short}), we calculate
\begin{align}\label{eq:GammaM11}
&(4\Gamma_2)_{B_1, B_1}-2\K_\infty^0\cdot 2\Gamma\notag\\
=&\begin{pmatrix}
d_x\cdot av_1^+(x) & -d_{y_1}^+-av_1^+(x) & -d_{y_2}^+-av_1^+(x)& \cdots & -d_{y_{d_x}}^+-av_1^+(x)\\
-d_{y_1}^+-av_1^+(x) & (4\Gamma_2)_{y_1,y_1}-2\K_\infty^0 & 2-4w_{y_1y_2}& \cdots & 2-4w_{y_{1}y_{d_x}}\\
-d_{y_2}^+-av_1^+(x) & 2-4w_{y_2y_1} & (4\Gamma_2)_{y_2,y_2}-2\K_\infty^0 & \cdots & 2-4w_{y_2y_{d_x}}\\
\vdots & \vdots &\vdots &\ddots & \vdots\\
-d_{y_{d_x}}^+-av_1^+(x) & 2-4w_{y_{d_x}y_{1}} & 2-4w_{y_{d_x}y_2} &\cdots & (4\Gamma_2)_{y_{d_x},y_{d_x}}-2\K_\infty^0
\end{pmatrix},
\end{align}
where for any $y_i\in S_1$,
\begin{align*}
(4\Gamma_2)_{y_i,y_i}-2\K_\infty^0=&(5-d_x+3d_{y_i}^++4d_{y_i}^0)-(3+d_x-av_1^+(x))\\
=&-2(d_x-1)+3d_{y_i}^++av_1^+(x)+4d_{y_i}^0.
\end{align*}
Observe that
$$d_x\cdot av_1^+(x)=\sum_{y\in S_1(x)}d_y^+=\sum_{z\in S_2(x)}d_z^-.$$
Therefore, subtracting (\ref{eq:GammaCCT}) from (\ref{eq:GammaM11}) produces the matrix $\widehat{\P}_{\infty}(x)$.
\end{proof}

\begin{proof}[Proof of Theorem \ref{thm:Sharpness}]
Recall Proposition \ref{prop:LMP}: $\K_{G,x}(\N)$ is the solution of the following semidefinite programming problem:
\begin{align*}
 &\text{maximize}\,\,\, K\\
&\text{subject to}\,\,\,\Gamma_2(x)-\frac{1}{\N}\Delta(x)^\top\Delta(x)\geq K\Gamma(x).
\end{align*}
We change the variable $K$ in the above problem to $\lambda$, which is given by
\begin{equation}\label{eq:lambdaK}
\lambda:=2\left(\K_\infty^0-\frac{2d_x}{\N}-K\right).
\end{equation}
Let us write
$$M_{K,\N}:=\Gamma_2-\frac{1}{\N}\D^\top\D-K\Gamma=\begin{pmatrix}
(\Gamma_2)_{B_1,B_1}-\frac{1}{\N}\D^\top\D-K\Gamma & (\Gamma_2)_{B_1,S_2}\\
(\Gamma_2)_{S_2,B_1} & (\Gamma_2)_{S_2,S_2}
\end{pmatrix}.
$$
Recalling (\ref{eq:QM}), we have
\begin{align*}
4\Q(M_{K,\N})=&4\Q(\Gamma_2-\K_\infty^0\Gamma)+4(\K_\infty^0-K)\Gamma-\frac{4}{\N}\D^\top\D\\
=&\widehat{\P}_\infty+\frac{4}{\N}\left(d_x\cdot 2\Gamma-\D^\top\D\right)+\lambda\cdot 2\Gamma.
\end{align*}
In the second equality above, we used Lemma \ref{lemma:QM=P}. Recalling (\ref{eq:GammaDeltaIdJd}),
we have, by (\ref{eq:PN}),
\begin{equation}\label{eq:QMPN}
4\Q(M_{K,\N})=\widehat{\P}_{\N}+\lambda\cdot 2\Gamma.
\end{equation}
Applying Proposition \ref{prop:MatrixVersion1} (i), we have $M_{K,\N}\geq 0$ if and only if $4\Q(M_{K,\N})\geq 0$. Hence, (\ref{eq:QMPN}) implies that the semidefinite programming problem is equivalent to the one in Theorem \ref{thm:Sharpness}.

Using (\ref{eq:lambdaK}) and (\ref{eq:QMPN}), the equivalence of (i) and (ii) in Theorem \ref{thm:Sharpness} is then a straightforward consequence of Corollary \ref{cor:preciseCur} and Proposition \ref{prop:MatrixVersion1}.
%
\end{proof}

We remark that, by (\ref{eq:QMPN}), the fact $\widehat{\P}_{\N}\mathbf{1}=0$ can also be derived from $M_{K,\N}\mathbf{1}=0$ and Proposition \ref{prop:MatrixVersion1}.

\begin{proof}[Proof of Theorem \ref{thm:lb}] Applying Theorem \ref{thm:Sharpness}, we only need to show
\begin{equation}\label{eq:goal}
\widehat{\P}_\N(x)-\lambda_{\min}(\widehat{\P}_\N(x))\cdot 2\Gamma(x)\geq 0.
\end{equation}
Let us denote the above matrix by $L$ for short. If $d_x=1$, we have $d_y^+=av_1^+(x)$ for the neighbor $y$ of $x$. Recall $\widehat{\P}_\N(x)\mathbf{1}=0$. We have $\widehat{\P}_\N(x)\mathbf{1}$ is a zero matrix. Therefore, (\ref{eq:goal}) is true.
If $d_x>1$,
(\ref{eq:goal}) is true because $\widehat{\P}_\N(x)\mathbf{1}=0$, $2\Gamma(x)\mathbf{1}=0$, and
\begin{align*}
\lambda_{\min}(\at{L}{\mathbf{1}^\perp})\geq &\lambda_{\min}\left(\at{\widehat{\P}_\N(x)}{\mathbf{1}^\perp}\right)-\lambda_{\min}(\widehat{\P}_\N(x))\cdot\lambda_{\min}(\at{2\Gamma(x)}{\mathbf{1}^\perp})\\
= & \lambda_{\min}\left(\at{\widehat{\P}_\N(x)}{\mathbf{1}^\perp}\right)-\lambda_{\min}(\widehat{\P}_\N(x))\geq 0.
\end{align*}
In the first inequality above, we used the fact $\lambda_{\min}(\widehat{\P}_\N)\leq 0$, which follows from $\widehat{\P}_\N(x)\mathbf{1}=0$, and in the subsequent equality, we used Proposition \ref{prop:Gamma}.
\end{proof}

\begin{proof}[Proof of Theorem \ref{thm:OutRegularFormula}]
Since $d_{y_i}^+=av_1^+(x)$ for any $y_i\in S_1(x)$, we have
\begin{equation}\label{eq:PNhatOutRegular}
\widehat{\P}_\N(x)=\left(\begin{array}{cccc}
0 & 0 & \cdots & 0\\
0 & \multicolumn{3}{c}{\multirow{3}{*}{\raisebox{1mm}{\scalebox{1}{$\P_\infty(x)+\frac{4}{\N}(d_xI_{d_x}-J_{d_x})$}}}}  \\
\vdots & \\
0 &
\end{array}
\right).
\end{equation}
Hence $\lambda_{\min}(\widehat{\P}_\N(x))=\lambda_{\min}\left(\P_\infty(x)+\frac{4}{\N}(d_xI_{d_x}-J_{d_x})\right)$. If $\lambda_{\min}(\widehat{\P}_\N(x))=0$, then the equality (\ref{eq:OutRegularFormula}) follows from Theorem \ref{thm:lb} and the upper bound in Theorem \ref{thm:ub}.

Otherwise, we have $\lambda_{\min}(\widehat{\P}_\N(x))<0$.
By Theorem \ref{thm:lb}, it remains to show the matrix
\begin{equation}
\widehat{\P}_\N(x)-\lambda_{\min}(\widehat{\P}_\N)(x)\cdot 2\Gamma(x)
\end{equation}
has at least two independent zero eigenvectors.
Recall the constant vector $\mathbf{1}_{d_x+1}$ is one zero eigenvector. Since $\lambda_{\min}(\widehat{\P}_\N(x))<0$, there exist $v=(v_0,v_1,\ldots, v_{d_x})\in \mathbf{1}_{d_x+1}^\perp$
which is the eigenvector of $\widehat{\P}_N$ corresponding to $\lambda_{\min}(\widehat{\P}_\N)$. By (\ref{eq:PNhatOutRegular}), we can assume $v_0=0$. Then we check $2\Gamma v=v$. (Recall Proposition \ref{prop:Gamma}).
Therefore, $v$ is another zero eigenvector of
$\widehat{\P}_\N-\lambda_{\min}(\widehat{\P}_\N)\cdot 2\Gamma$.
\end{proof}

\subsection{Families of examples}\label{section:families_examples}

We now employ our results to discuss several families of examples.

\begin{example}[Regular trees]\label{example:dtree} Let $T_d=(V,E)$ be a $d$-regular tree and $x\in V$. We have
\begin{equation}\label{eq:dtree}
\K_{T_d,x}(\N)=\left\{
              \begin{aligned}
                &2-\frac{2d}{\N}, &&\hbox{if $0<\N\leq 2$;} \\
                &2-d, &&\hbox{if $\N>2$.}
              \end{aligned}
            \right.
\end{equation}
\end{example}
\begin{proof}
For any $y\in S_1(x)$, we have $d_y^+=av_1^+(x)=d-1$. Hence $T_d$ is $S_1$-out regular at $x$ and we apply Theorem \ref{thm:OutRegularFormula}.
Note all the off-diagonal entries of $\P_\infty(x)$ equal $2$. Then by the property $\P_\infty(x)\mathbf{1}=0$, we obtain
\begin{equation}
\P_\infty(x)=-2(dI_d-J_d).
\end{equation}
Observe that the set of eigenvalues of the matrix $dI_d-J_d$ is
\begin{equation}\label{eq:dI-J}
\sigma(dI_d-J_d)=\{0,\underbrace{d,\ldots,d}_{d-1}\}.
\end{equation}
Therefore, $\lambda_{\min}(\P_\infty(x))=-2d$, and $\N_0(x)$, defined in (\ref{eq:N0}), is 2.
Noticing that
$$\K_\infty^0(x)=\frac{3+d_x-av_1^+(x)}{2}=2,$$
we obtain (\ref{eq:dtree}) from (\ref{eq:step_fucntion}).
\end{proof}

While regular trees are $\N$-curvature sharp only for $\N\in (0,2]$, we will see the complete graphs are curvature sharp for any $\N\in (0,\infty]$.

\begin{example}[Complete graphs]\label{example:Kn} Let $K_n=(V,E)$ be the complete graph on $n\geq 2$ vertices and $x\in V$. For any $\N\in (0,\infty]$, we have
\begin{equation}\label{eq:Kn}
\K_{K_n,x}(\N)=\frac{n+2}{2}-\frac{2(n-1)}{\N}.
\end{equation}
\end{example}
We remark that (\ref{eq:Kn}) has been obtained in \cite[Proposition 3]{JL14} via different calculating method. Below we show (\ref{eq:Kn}) follows immediately from Theorem \ref{thm:OutRegularFormula}.
\begin{proof}
We check $d_x=n-1$ and $d_y^+=av_1^+(x)=0$, for any $y\in S_1(x)$. Therefore,
\begin{equation}
\K_{\infty}^0(x)=\frac{3+d_x-av_1^+(x)}{2}=\frac{n+2}{2}.
\end{equation}
Since $w_{y_iy_j}=1$ for any pair of vertices in $S_1(x)$, all off-diagonal entries of $\P_\infty(x)$ equal $-2$. By the property $\P_\infty(x)\mathbf{1}=0$, we know
\begin{equation}
P_\infty(x)=2(n-1)I_{n-1}-2J_{n-1}.
\end{equation}
Recall (\ref{eq:dI-J}), we have $\lambda_{\min}(\P_\infty(x))=0$. Hence, we have $\N_0(x)$, defined in (\ref{eq:N0}), equal $\infty$. By (\ref{eq:step_fucntion}), we obtain (\ref{eq:Kn}).
\end{proof}

Next, we consider the family of complete bipartite graphs, which are possibly irregular, but are still $S_1$-out regular.

\begin{example}[Complete bipartite graphs]\label{example:Kmn}
Let $K_{m,n}=(V,E)$ be a complete bipartite graph. Let $x\in V$ be a vertex with degree $d_x=n$.
If $n=1$ or $n\leq 2m-2$, we have for any $\N\in (0,\infty]$
\begin{equation}\label{eq:KmnCase1}
\K_{K_{m,n},x}(\N)=\frac{4+n-m}{2}-\frac{2n}{\N}.
\end{equation}
If, otherwise, $n\neq 1$ and $n>2m-2$, we have
\begin{equation}\label{eq:KmnCase2}
\K_{K_{m,n},x}(\N)=\left\{
              \begin{aligned}
                &\frac{4+n-m}{2}-\frac{2n}{\N}, &&\hbox{if $0<\N\leq \frac{2n}{n-2m+2}$;} \\
                &\frac{3m-n}{2}, &&\hbox{if $\N>\frac{2n}{n-2m+2}$.}
              \end{aligned}
            \right.
\end{equation}
In particular, we have $\K_{K_{1,1}}(\infty)=2$ and, when $(n,m)\neq (1,1)$,
\begin{equation}
\K_{K_{m,n},x}(\infty)=\frac{m+2-|n-2m+2|}{2}.
\end{equation}
\end{example}

\begin{table}[h]
\centering 
\begin{tabular}{c| cccccccccc}
\hline
\hline\\[-2ex]
$m\diagdown n$ & 1 & 2 & 3 & 4 & 5 & 6 & 7 & 8 & 9 & 10 \\
\hline\\[-1.3ex]         
1 & 2 & 0.5 & 0 & -0.5 & -1 & -1.5 & -2 & -2.5 & -3 & -3.5\\
2 & 1.5 & 2 & 1.5 &  1 & 0.5 & 0 & -0.5 & -1 &-1.5 & -2 \\
3 & 1 & 1.5 & 2 &  2.5 & 2 & 1.5 & 1 & 0.5 & 0 & -0.5\\
4 & 0.5 & 1 & 1.5 &  2 & 2.5 & 3 & 2.5 & 2 & 1.5 & 1 \\
5 & 0 & 0.5 & 1 & 1.5 & 2 & 2.5 & 3 & 3.5 & 3 & 2.5\\
6 & -0.5 & 0 & 0.5 & 1 & 1.5 & 2 & 2.5 & 3 & 3.5 & 4 \\
7 & -1 & -0.5 & 0 & 0.5 & 1 & 1.5 & 2 & 2.5 & 3 & 3.5\\
8 & -1.5 & -1 & -0.5 & 0 & 0.5 & 1 & 1.5 & 2 & 2.5 & 3 \\
9 & -2 & -1.5 & -1 & -0.5 & 0 & 0.5 & 1 & 1.5 & 2 & 2.5\\
10 & -2.5 & -2 & -1.5 & -1 & -0.5 & 0 & 0.5 & 1 & 1.5 & 2
\\[1ex] 
\hline                          
\end{tabular}
\label{tab:CurvatureCompleteBipartite}
\caption{$\K_{K_{m,n},x}(\infty)$ at $x\in V$ with $d_x=n$} 
\end{table}

\begin{remark}\label{remark:K26}
We remark that $K_{2,6}$ has constant flat curvature, i.e., $\K_{K_{2,6},x}(\infty)=0 \,\,\forall x$. But at each vertex $x$ with degree $2$, $\K_{K_{2,6},x}(\N)<0$ for any finite dimension $\N$. This has already been observed in Example \ref{example:K26}. On the other hand, at each vertex $y$ with degree $6$, we see $\K_{K_{2,6},y}(\N)=0$, for any $\N\in [3,\infty]$.
\end{remark}

\begin{proof} Let $x\in V$ be a vertex such that $d_x=n$. Then, we have $d_y^+=av_1^+(x)=m-1,d_y^0=0$, for any $y\in S_1(x)$. Therefore, we obtain
\begin{equation}
\K_\infty^0(x)=\frac{3+d_x-av_1^+(x)}{2}=\frac{4+n-m}{2}.
\end{equation}
Note that $d_z^-=n$, for any $y\in S_1(x)$, and there are $(m-1)$ $2$-paths connecting any two vertices $y_i,y_j\in S_1(x)$ via a vertex in $S_2(x)$. Therefore, each off-diagonal entry of $\P_\infty(x)$ equals $2-\frac{4}{n}(m-1)$. By the property that $\P_\infty(x)\mathbf{1}=0$, we have
\begin{equation*}
\P_\infty(x)=-\frac{2}{n}(n-2m+2)(nI_n-J_n).
\end{equation*}
Recalling (\ref{eq:dI-J}), we have $\sigma(\P_\infty(x))=\{0,-2(n-2m+2),\ldots, -2(n-2m+2)\}$.

If $n=1$ or $n\leq 2m-2$, i.e., if $n=m=1$ or $n\leq 2m-2$, we have $\lambda_{\min}(\P_\infty(x))=0$. Hence, $\N_0(x)=\infty$. We obtain (\ref{eq:KmnCase1}) by (\ref{eq:step_fucntion}).
If, otherwise, $n\neq 1$ and $n>2m-2$, we have $\lambda_{\min}(\P_\infty(x))=-2(n-2m+2)$. Hence, $\N_0(x)=\frac{2n}{n-2m+2}$.
Noticing that
$$\K_\infty^0(x)-\frac{2d_x}{\N_0(x)}=\frac{4+n-m}{2}-(n-2m-2)=\frac{3m-n}{2},$$
we obtain (\ref{eq:KmnCase2}) by (\ref{eq:step_fucntion}).
\end{proof}

Particularly, we have the curvature function for star graphs $Star_n=K_{1,n}$. We can suppose $n\geq 2$. (Recall when $n=1$, $Star_1=K_{1,1}=K_2$.)
\begin{example}[Star graphs]\label{example:Starn} Let $Star_n=(V,E), n\geq 2$ be a star graph. For the vertex $x$ with $d_x=n$, we have
\begin{equation}
\K_{Star_n,x}(\N)=\left\{
              \begin{aligned}
                &\frac{3+n}{2}-\frac{2n}{\N}, &&\hbox{if $0<\N\leq 2$;} \\
                &\frac{3-n}{2}, &&\hbox{if $\N>2$.}
              \end{aligned}
            \right.
\end{equation}
For any leaf $y$, we have $av_1^+(y)=n-1$, and hence, by Example \ref{example:leaf},
\begin{equation}
\K_{Star_n,y}(\N)=\frac{5-n}{2}-\frac{2n}{\N}\,\,\,\,\,\,\,\forall\,\, \N\in(0,\infty].
\end{equation}
\end{example}

From the above examples, we can derive the curvature function for cycles.
\begin{example}[Cycles]\label{example:Cn} Let $C_n=(V,E)$ be a cycle graph with $n$ vertices and $x\in V$. Since $C_3=K_3$, we have, by Example \ref{example:Kn}, for any $\N\in (0,\infty]$,
\begin{equation}\label{eq:C3}
\K_{C_3,x}(\N)=\frac{5}{2}-\frac{4}{\N}.
\end{equation}
Since $C_4=K_{2,2}$, we have, by Example \ref{example:Kmn}, for any
$\N\in(0,\infty]$,
\begin{equation}\label{eq:C4}
\K_{C_4,x}(\N)=2-\frac{4}{\N}.
\end{equation}
When $n\geq 5$, the local subgraph $B_2(x)$ is isomorphic to that of a vertex in a $2$-regular tree (i.e., infinite path). Therefore, we have, by Example \ref{example:dtree},
\begin{equation}\label{eq:C5greater}
\K_{C_n,x}(\N)=\left\{
              \begin{aligned}
                &2-\frac{4}{\N}, &&\hbox{if $0<\N\leq 2$;} \\
                &0, &&\hbox{if $\N>2$,}
              \end{aligned}
            \right.\,\,\,\,\,\text{for }\,\,n\geq 5.
\end{equation}
\end{example}

From Example \ref{example:Kmn}, we see $K_{n,n}$ is $\infty$-curvature sharp. Next we see when $n\geq 4$, it is still $\infty$-curvature sharp after removing the edges of a perfect matching.

\begin{example}[Crown graphs]\label{example:Crown} The crown graph $\mathrm{Crown}(n,n)$ is the graph with vertex set $V:=\{x_1,\ldots, x_n,x_{n+1},\ldots, x_{2n}\}$ and edge set $E:=\{\{x_i,x_{n+j}\}:i,j\in [n], i\neq j\}$. $\mathrm{Crown}(1,1)$ is an empty graph, $\mathrm{Crown}(2,2)$ is two copies of $K_2$, and $\mathrm{Crown}(3,3)$ is the cycle $C_6$. Let $x\in V$. When $n\geq 4$, we have for any $\N\in(0,\infty]$,
\begin{equation}\label{eq:Crown}
\K_{\mathrm{Crown}(n,n),x}(\N)=2-\frac{2(n-1)}{\N}.
\end{equation}
\end{example}
\begin{proof}
Since $x$ is $S_1$-out regular with $d_x=n-1,\, av_1^+(x)=n-2$, we have $\K_\infty^0(x)=2$. The off-diagonal entry of $\P_\infty(x)$ equals $2-\frac{n-3}{n-2}$. By the property that $\P_\infty\mathbf{1}=0$, we have
$$\P_\infty(x)=\frac{2(n-4)}{n-2}((n-2)I_{n-2}-J_{n-2}).$$
Therefore, we have $\lambda_{\min}(\P_\infty(x))=0$, which implies $\N_0(x)=\infty$ and then (\ref{eq:Crown}).
\end{proof}


Next, we calculating the curvature functions of paths $P_n$ with $n$ vertices. We will assume $n\geq 4$. (Recall $P_2=K_2$, $P_3=K_{1,2}=Star_2$.)
\begin{example}[Paths]\label{example:Pn} Let $P_n=(V,E), n\geq 4$ be a path graph. For the two leaves $x_i,i=1,2$, we have $av_1^+(x_i)=1$, and hence, by Example \ref{example:leaf}, for any $\N\in (0,\infty]$,
\begin{equation}\label{eq:Pnleaf}
\K_{P_n, x_i}(\N)=\frac{3}{2}-\frac{2}{\N}.
\end{equation}
For the two inner vertices next to leaves, i.e., $y_i\in V$ with $\{y_i,x_i\}\in E$, $i=1,2$, we have, for any $\N\in (0,\infty]$,
\begin{equation}\label{eq:Pn_next_to_leaf}
\K_{P_n,y_i}(\N) = \frac{5}{4} - \frac{8+\sqrt{\N^2+(4\N-8)^2}}{4\N}.
\end{equation}
For the remaining vertices $z_i,i=1,\ldots, n-4$, the curvature function agrees with that of a $2$-regular tree. Hence, we have, by Example \ref{example:dtree},
\begin{equation}\label{eq:Pn_other_inner_vertices}
\K_{P_n,z_i}(\N)=\left\{
              \begin{aligned}
                &2-\frac{4}{\N}, &&\hbox{if $0<\N\leq 2$;} \\
                &0, &&\hbox{if $\N>2$,}
              \end{aligned}
            \right.
\end{equation}
\end{example}
\begin{proof}
We now prove (\ref{eq:Pn_next_to_leaf}).
For an inner vertex $y_i$ next to a leaf, we have $d_{y_i}=2$, $av_1^+(y_i)=1/2$, and therefore $\K_\infty^0(y_i)=9/4$. Applying Theorem \ref{thm:Sharpness}, $$ \K_{G,y_i}(\N) = \frac{9}{4} - \frac{\lambda}{2} - \frac{4}{\N}, $$
where $\lambda \ge 0$ is chosen such that the matrix
$$ \widehat \P_\N(y_i) + \lambda \cdot 2\Gamma(y_i) =
\begin{pmatrix} 0 & \frac{1}{2} & -\frac{1}{2} \\
\frac{1}{2} & -\frac{5}{2} + \frac{4}{\N} & 2 - \frac{4}{\N} \\
-\frac{1}{2} & 2 - \frac{4}{\N} & -\frac{3}{2} + \frac{4}{\N} \end{pmatrix} +
\lambda \begin{pmatrix} 2 & -1 & -1 \\ -1 & 1 & 0 \\ -1 & 0 & 1 \end{pmatrix}$$
is positive semidefinite and has zero eigenvalue of multiplicity at least $2$.
This condition leads to
$$ \lambda = \frac{4\N-8+\sqrt{\N^2+(4\N-8)^2}}{2\N}, $$
and the curvature functions is therefore (\ref{eq:Pn_next_to_leaf}).
\end{proof}

\section{Curvature and connectedness of $\mathring{B}_2(x)$}\label{section:connComp}
In this section, we prove relations between the curvature function
$\K_{G,x}$ at a vertex $x \in V$ and topological properties of the
punctured $2$-ball $\mathring{B}_2(x)$. More precisely, we show that
\begin{description}
\item[(a)] The curvature $\K_{G,x}(\infty)$ is -- with very few
  exceptions -- always negative if $\mathring{B}_2(x)$ consists of
  more than one connected component.
\item[(b)] The curvature function $\K_{G,x}$ does not decrease under
  adding edges in $S_1(x)$, or merging two vertices in $S_2(x)$ which
  do not have common neighbours. Obviously, these
  operation increases the connectedness of $\mathring{B}_2(x)$.
\end{description}

\subsection{Connected components and negative curvature}

Let $G=(V,E)$ be a locally finite simple graph, $x \in V$ be a vertex
and $d =d_x$ its degree. Henceforth we assume that we have chosen a
specific connected component of $\mathring{B}_2(x)$. We denote the
vertices of this connected component in $S_1(x)$ and $S_2(x)$ by
$y_1,\dots,y_r$ and $z_1,\dots,z_s$, respectively. Figure
\ref{Ftwoball} illustrates connected components of
$\mathring{B}_2(x)$. Note that the punctured $2$-ball
$\mathring{B}_2(x)$ has more than one component if and only if
$d > r$.

\begin{figure}[h]
    \centering
    \includegraphics[width=0.5\textwidth]{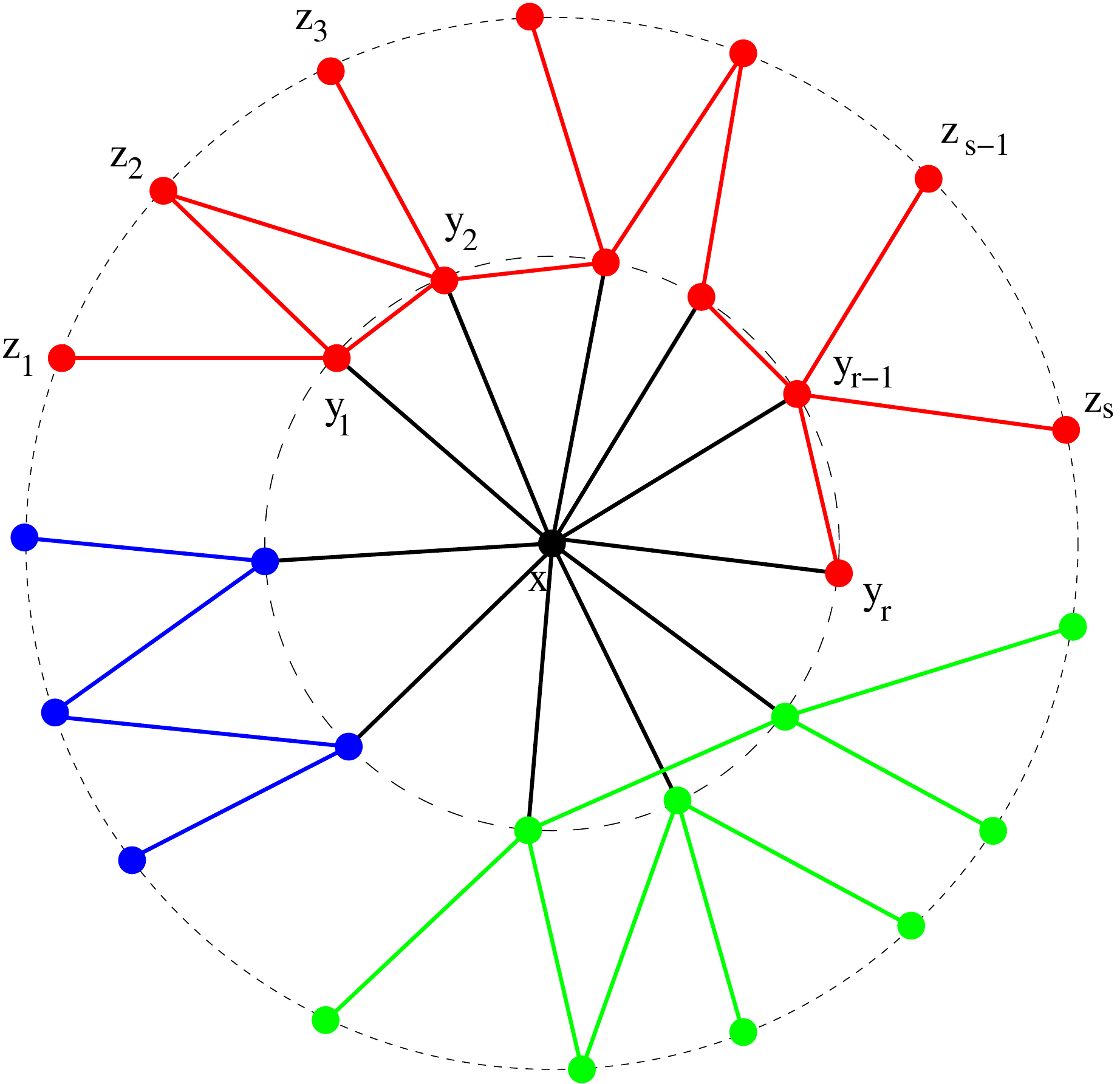}
    \caption{Connected components of $\mathring{B}_2(x)$
    in different colours; choosing the red connected component leads to
    $d = 11$ and $r=6$.\label{Ftwoball}}
\end{figure}

Our first two results assume that $\mathring{B}_2(x)$ has more than
one connected component and distinguish the cases $s > 0$ (our
connected component has vertices in $S_2(x)$) and $s = 0$ (our
connected component consists entirely of vertices in $S_1(x)$). Our
first result deals with the case that our connected component of
$\mathring{B}_2(x)$ has only vertices in $S_1(x)$.

\begin{lemma} \label{lem:snuld4}
  Assume that $\mathring{B}_2(x)$ has more than one connected
  component (i.e., $d > r$) and that $s = 0$ and
  $d \ge 4$. Then we have $\K_{G,x}(\infty) < 0$.
\end{lemma}

For the proofs, it is useful to introduce the notions ${\bf 0}_t$ and
${\bf 1}_t$ for the all-zero and the all-one column vector of size $t$.

\begin{proof}
  Let $A$ be the submatrix of $4\Gamma_2(x)$ corresponding to the
  vertices $x$, $y_1, \dots, y_r$. Let $C$ be the $(2 \times (r+1))$
  matrix
  $$ C = \begin{pmatrix} 1 & {\bf 0}_r^\top \\  0 & {\bf 1}_r^\top \end{pmatrix}. $$
  Then the $(2 \times 2)$ matrix $A_0 =CAC^\top$ has the form
  $$ A_0 = \begin{pmatrix} 3d+d^2 & -r(3+d) \\ -r(3+d) & r(5-d)+2r(r-1)
  \end{pmatrix}. $$
  Since
  \begin{equation} \label{eq:detA0} \det A_0 = - (3+d) r (d-3) (d-r),
  \end{equation}
  $d \ge 4$ implies $\det A_0 < 0$ and $A_0$ cannot be positive
  semidefinite. This implies $\K_{G,x}(\infty) < 0$.
\end{proof}

Our second result reads as follows.

\begin{lemma} \label{lem:sposd3}
  Assume that $\mathring{B}_2(x)$ has more than one connected
  component (i.e., $d > r$) and that $s > 0$ and
  $d \ge 3$. Then we have $\K_{G,x}(\infty) < 0$.
\end{lemma}

\begin{proof}
  Let $A$ be the submatrix of $4\Gamma_2(x)$ corresponding to the
  vertices $x$, $y_1, \dots, y_r$, $z_1, \dots, z_s$. Let $C$ be the
  $(3 \times (1+r+s))$ matrix
  $$ C = \begin{pmatrix} 1 & {\bf 0}_r^\top & {\bf 0}_s^\top \\
    0 & {\bf 1}_r^\top & {\bf 0}_s^\top \\
    0 & {\bf 0}_r^\top & {\bf 1}_s^\top \end{pmatrix}. $$
  Then the $(3 \times 3)$ matrix $A_0 = C A C^\top$ has the form
  \begin{equation} \label{eq:A0two}
  A_0 = \begin{pmatrix} 3d+d^2 & -(3+d)r - S & S \\
    -(3+d)r - S & r(5-d) + 3S + 2 r(r-1) & -2S \\
    S & -2S & S \end{pmatrix}
  \end{equation}
  with $S = \sum_{i=1}^r d_{y_i}^{out} = \sum_{j=1}^s d_{z_j}^{in} \ge
  s$. Choosing the row vector $v = (r,d,d+1/S)$, we obtain
  \begin{equation} \label{eq:vA0v} v^\top A_0 v = -r d (d-3) (d-r) -
    2(d-r) + \frac{1}{S} \le \frac{1}{S} -2(d-r) \le \frac{1}{S} - 4
    \le -3 < 0.
  \end{equation}
  This shows that $A$ is not positive semidefinite and, therefore,
  $\K_{G,x}(\infty) < 0$.
\end{proof}

In the situation described in Lemma \ref{lem:sposd3}, we can only have
$\K_{G,x}(\infty) > 0$ if $d=2$ and, consequently, $r=1$. In this
case, there is only one vertex of our connected component, denoted by
$y$, in $S_1(x)$, and the next result tells us that
$\K_{G,x}(\infty) < 0$ unless the out degree of $y$ satisfies
$d_y^+ = 1$.

\begin{lemma} \label{lem:sd2}
  Assume that $d=2$ and $\mathring{B}_2(x)$ has two connected
  components, i.e., $S_1(x) = \{y,y'\}$ and $y \not\sim y'$. Then we
  have $\K_{G,x}(\infty) < 0$ if $d_y^+ \ge 2$ or $d_{y'}^+ \ge 2$.
\end{lemma}

\begin{proof} Following the proof of Lemma \ref{lem:sposd3}, we are in
  the special case $(d,r)=(2,1)$ and $S = d_y^+$. The matrix $A_0$
  from \eqref{eq:A0two} then simplifies to
  $$ A_0 = \begin{pmatrix} 10 & -(5+S) & S \\ -(5+S) & 3+3S & -2S \\
  S & -2S & S \end{pmatrix}, $$
  and we have $\det A_0 = S(5-3S)$. If $d_y^+ = S \ge 2$, we obtain
  $\det A_0 < 0$, i.e., $4\Gamma_0$ cannot be positive semidefinite and
  we have $\K_{G,x}(\infty) < 0$. The same holds true when we replace $y$
  by $y'$, finishing the proof.
\end{proof}

\begin{theorem} \label{thm:conncomp} Let $G = (V,E)$ be a locally
  finite simple graph and $x \in V$. If $\mathring{B}_2(x)$ has more
  than one connected component then $\K_{G,x}(\infty) < 0$, except for
  one of the five cases (a)-(e) presented in Figure \ref{Fexceptions}.
\end{theorem}

\begin{figure}[h]
    \centering
    \includegraphics[width=\textwidth]{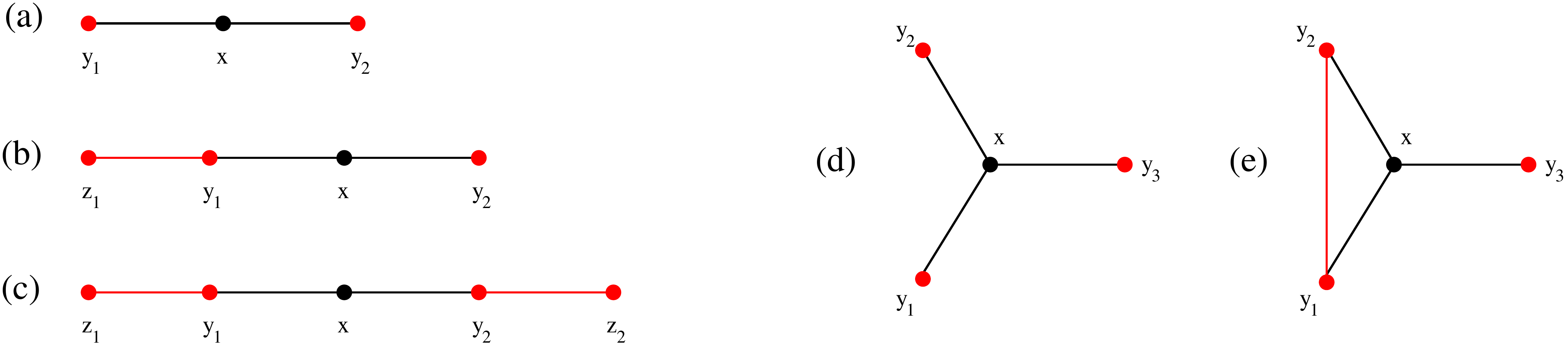}
    \caption{$2$-balls $B_2(x)$ with $\K_{G,x}(\infty) \ge 0$. The
      corresponding punctured $2$-balls $\mathring{B}_2(x)$ are
      red.\label{Fexceptions}}
\end{figure}

\begin{proof} We assume that $\mathring{B}_2(x)$ has at least two
  connected components and, therefore, $d \ge 2$. If $d \ge 4$ we have
  $\K_{G,x}(\infty) < 0$ by Lemmata \ref{lem:snuld4} and
  \ref{lem:sposd3}. So it only remains to investigate $2 \le d \le 3$.

  If $S_2(x) \neq \emptyset$, there exists a connected component of
  $\mathring{B}_2(x)$ with a vertex in $S_2(x)$ and Lemma
  \ref{lem:sposd3} implies $\K_{G,x}(\infty) < 0$ if $d=3$. So we are
  left with $d=2$ and, in view of Lemma \ref{lem:sd2}, the only
  remaining possibilities for $B_2(x)$ to achieve
  $\K_{G,x}(\infty) \ge 0$ in the case $S_2(x) \neq \emptyset$ are (b)
  or (c).

  In the case $S_2(x) = \emptyset$, $d=2$ and $\K_{G,x}(\infty) \ge 0$
  lead necessarily to the configuration (a) for
  $B_2(x)$. Similarly, $d=3$ and $\K_{G,x}(\infty) \ge 0$ lead
  necessarily to the configurations (d) anf (e).
\end{proof}

\begin{examples}\label{example:exceptions} Let us study the curvature
  functions $\K_{G,x}$ for the exceptional $2$-balls $B_2(x)$ in
  Figure \ref{Fexceptions} (a)-(e).
\begin{description}
\item[(a)] Here we have $G = K_{1,2}$ and
$$ \K_{G,x}(\N) = \begin{cases} \frac{5}{2} - \frac{4}{\N}, & \text{if $0 < \N \le 2$,} \\ \frac{1}{2}, & \text{if $\N > 2$.} \end{cases} $$
\item[(b)] It follows from Example \ref{example:Pn} that the curvature function
is given by
$$ \K_{G,x}(\N) = \frac{5}{4} - \frac{8+\sqrt{\N^2+(4\N-8)^2}}{4\N}. $$
In particular, we have $\K_{G,x}(\infty) = \frac{5-\sqrt{17}}{4} = 0.219\dots$.
\item[(c)] Here the curvature function $\K_{G,x}$ agrees with the
  curvature function of the tree $T_2$, i.e.,
$$ \K_{G,x}(\N) = \begin{cases} 2 - \frac{4}{\N}\ & \text{if $0 < \N \le 2$,} \\
0, & \text{if $\N > 2$.} \end{cases} $$
\item[(d)] Here we have $G = K_{1,3}$ and
$$ \K_{G,x}(\N) = \begin{cases} 3 - \frac{6}{\N}, & \text{if $0 < \N \le 2$,} \\ 0, & \text{if $\N > 2$.} \end{cases} $$
\item[(e)] Since $S_2(x) = \emptyset$, the vertex $x \in V$ is
$S_1$-out regular. Therefore, Theorem \ref{thm:OutRegularFormula}
yields
$$ \K_{G,x}(\N) = 3 - \frac{6}{\N} - \frac{1}{2} \lambda_{\min} \left[
\begin{pmatrix} 0 & -2 & 2 \\ -2 & 0 & 2 \\ 2 & 2 & -4 \end{pmatrix} +
\frac{4}{\N} \begin{pmatrix} 2 & -1 & -1 \\ -1 & 2 & -1 \\
-1 & -1 & 2 \end{pmatrix} \right]. $$
Since the spectrum of the involved $(3 \times 3)$ matrix is
$\{ 0, -6 + 12/N, 12 + 2/N\}$, we conclude
$$ \K_{G,x}(\N) =  \begin{cases} 3 - \frac{6}{\N}, & \text{if $0 < \N \le 2$,}
\\ 0, & \text{if $\N > 2$,} \end{cases} $$
which agrees with the curvature function in (d).
\end{description}
\end{examples}

Theorem \ref{thm:conncomp} and the curvature calculations for the
exceptional cases have the following immediate consequences.

\begin{definition} Let $G = (V,E)$ be be a locally finite simple
  graph. An edge $e \in E$ is called an \emph{$(r,s)$-bridge} if the
  graph $G$ decomposes after removal of the edge $e = \{x,y\}$ into
  two separate non-empty components and if the degrees of the vertices
  $x$ and $y$ in each of the components after removal of $e$ are $r$
  and $s$, respectively.
\end{definition}

\begin{corollary} Let $G = (V,E)$ be be a locally finite simple
  graph and $e =\{x,y\} \in E$ be an $(r,s)$-bridge.
  \begin{itemize}
  \item[(a)] If $r=0$, $x$ is a leaf and the curvature function
    $\K_{G,x}$ is given by
    $$ \K_{G,x}(\N) = 2 - \frac{s}{2} - \frac{2}{\N}. $$
    In particular, we have $\K_{G,x}(\infty) \ge 0$ iff $s \le 4$.
  \item[(b)] If $r=1$, we can only have $\K_{G,x}(\infty) \ge 0$ if
    $s \in \{0,1\}$.
  \item[(c)] If $r=2$, we have $\K_{G,x}(\infty) \le 0$, and we can
    only have equality if $s=0$.
  \item[(d)] If $r \ge 3$, we have always $\K_{G,x}(\infty) < 0$.
  \end{itemize}
\end{corollary}

\begin{proof} The case $r=0$ follows from Example \ref{example:leaf}.
  In all other cases, $\mathring{B}_2(x)$ has at least two connected
  components and we can directly apply Theorem \ref{thm:conncomp}.
\end{proof}

\begin{definition} Let $G = (V,E)$ be be a locally finite simple graph
  and $e \in E$. The \emph{girth of $e$}, denoted by ${\rm girth}(e)$
  is the length of the shortest circuit in $G$ containing $e$. If $e$
  is not contained in any circuit, we define
  ${\rm girth}(e) = \infty$.
\end{definition}

\begin{corollary} Let $G = (V,E)$ be be a locally finite simple graph
  and $e = \{x,y\} \in E$. If $5 \le {\rm girth}(e) < \infty$, we have
  $\K_{G,x}(\infty), \K_{G,y}(\infty) \le 0$.

  If ${\rm girth}(e) = \infty$, the only exceptions for
  $\K_{G,x}(\infty) < 0$ are: (i) $d_x \ge 2$ and we have one of the
  situations (a)-(e) in Figure \ref{Fexceptions} for $B_2(x)$, (ii) $x$ is a
  leaf and the other vertex has degree $\le 5$.
\end{corollary}

\begin{proof} The girth condition $5 \le {\rm girth}(e) < \infty$
  implies that $\mathring{B}_2(x)$ has at least two connected
  components and we can apply Theorem \ref{thm:conncomp}. So we only
  have to consider the exceptional cases. The exceptional cases (a)
  and (b) imply ${\rm girth}(e) = \infty$, and we have in the cases
  (c), (d), (e) that $\K_{G,x}(\infty) = 0$. The same holds true for
  the vertex $y$.

  Now we assume ${\rm girth}(e) = \infty$. If $d_x \ge 2$,
  $\mathring{B}_2(x)$ must have at least two connected components and
  we can apply Theorem \ref{thm:conncomp}. The only exceptions of
  $\K_{G,x}(\infty) < 0$ are then situations (a)-(e) in Figure
  \ref{Fexceptions}. If $d_x=1$, we know from Example
  \ref{example:leaf} that it has only non-negative curvature if the
  degree of its neighbour is $\le 5$, which is precisely the case
  (ii).
\end{proof}

\begin{remark}
  In \cite{HL16}, B. Hua and Y. Lin classified graphs with girth at least
  $5$ and satisfying $CD(0,\infty)$. Their work was carried out
  independently, provides related but different results, and
  considers, in contrast to our context, the normalized case.
\end{remark}

\subsection{Operations that do not decrease the curvature}

In this subsection, we discuss operations that does not decrease the curvature. The first one is adding new spherical edges in $S_1(x)$.

\begin{proposition}
Let $G=(V,E)$ be a graph and $x\in V$ be a vertex. Let $G'=(V, E')$ be the graph obtained from $G$ by adding a new spherical edge in $S_1(x)$. Then we have for any $\N\in (0,\infty]$,
$$\K(G',x;\N)\geq \K(G,x;\N).$$
\end{proposition}
\begin{proof}
Suppose that $G'$ is obtained from $G$ by adding a new edge $\{y_1,y_2\}$, where $y_1,y_2\in S_1(x)$. Then we have
$$\Delta'(x)=\Delta (x), \,\, \Gamma'(x)=\Gamma(x),$$
and
$$\Gamma_2'(x)-\Gamma_2(x)=\begin{pmatrix}
0 & 0 & 0 \\
0 & (\Gamma_2)_{S_1,S_1}'-(\Gamma_2)_{S_1,S_1} & 0\\
0 & 0 & 0
\end{pmatrix}.$$
By (\ref{eq:Gamma2S1S1}), we obtain that
\begin{equation}\label{eq:differenceMatrix}
(\Gamma_2)_{S_1,S_1}'-(\Gamma_2)_{S_1,S_1}=\begin{pmatrix}
1 & -1 &\cdots &0\\
-1 & 1 & \cdots & 0\\
\vdots &\vdots &\ddots &\vdots\\
0 & 0 & \cdots & 0
\end{pmatrix},
\end{equation}
is positive semidefinite. This completes the proof.
\end{proof}
\begin{remark} By Corollary \ref{cor:preciseCur}, $\K(G',x;\N)=\K(G,x;\N)$ if and only if the multiplicity of zero eigenvalue of $\Gamma_2'(x)-\frac{1}{\N}\Delta'(x)^\top\Delta'(x)-\K(G,x;\N)\Gamma'(x)$ is no smaller than $2$. Recall this happens in Example \ref{example:exceptions} \textbf{(e)}.
%
\end{remark}

Merging two vertices $z_1,z_2$ in $S_2(x)$ also does not decrease the curvature. Here we assume the two vertices $z_1,z_2$ do not have common neighbours. By merging two vertices $z_1,z_2$, we mean the following operations: Remove the possible edge connecting $z_1$ and $z_2$ and identify $z_1,z_2$ as a new vertex $z$, where edges incident to $z$ each correspond to an edge incident to either $z_1$ or $z_2$. The assumption that $z_1,z_2$ do not have common neighbours ensures that no multi-edge is produced by this operation.

\begin{proposition}\label{prop:S2merg}
Let $G=(V,E)$ be a graph and $x\in V$ be a vertex. Let $G''=(V'', E'')$ be the graph obtained from $G$ by merging two vertices in $S_2(x)$ which do not have common neighbours. Then we have for any $\N\in (0,\infty]$,
$$\K(G'',x;\N)\geq \K(G,x;\N).$$
\end{proposition}
\begin{proof}
Suppose that $G''$ is obtained from $G$ by merging the two vertices $z_{|S_2(x)|-1},z_{|S_2(x)|}$ in $S_2(x)$ which do not have common neighbours. Then we have $\D''(x)=\D(x)$, $\Gamma''(x)=\Gamma(x)$. Let $C''$ be a $|B_2(x)|-1$
by $|B_2(x)|$ matrix
$$C''=\begin{pmatrix}
I_{|B_2(x)|-2} & \mathbf{0}_{|B_2(x)|-2}& \mathbf{0}_{|B_2(x)|-2}\\
\mathbf{0}_{|B_2(x)|-2}^\top & 1 & 1
\end{pmatrix}.
$$
Then we have
$$\Gamma_2''(x)=C''\Gamma_2(x)(C'')^\top.$$
Therefore, we have
\begin{align*}
&\Gamma_2''(x)-\frac{1}{\N}(\D''(x))^\top\D''(x)-\K(G,x,\N)\Gamma''(x)\\
=&C''(\Gamma_2(x)-\frac{1}{\N}\D(x)^\top\D(x)-\K(G,x,\N)\Gamma(x))(C'')^\top\geq 0.
\end{align*}
This finishes the proof.
\end{proof}

We believe the following property is true.
\begin{conjecture}
Let $G=(V,E)$ be a graph and $x\in V$ be a vertex. Let $G'''=(V''', E''')$ be the graph obtained from $G$ by one of the following two operations:
\begin{itemize}
\item Delete a leaf in $S_2(x)$ and its incident edge.
\item Delete $z\in S_2(x)$ and its incident edges $\{\{y,z\}\in E: y\in S_1(x)\}$; Adding an new edge between every two of $\{y\in S_1(x): \{y,z\}\in E\}$.
\end{itemize}
Then we have for any $\N\in (0,\infty]$,
$$\K(G''',x;\N)\geq \K(G,x;\N).$$
\end{conjecture}


\section{Curvature functions of Cartesian products}\label{section:Cartesian}
In this section, we show that the curvature functions of a Cartesian product can be explicitly determined from curvature functions of the factors.

\subsection{$*$-product of functions}
We first discuss an abstract product between functions. Let us denote by $\mathcal{FK}$ the set of continuous, monotone non-decreasing functions $f:(0,\infty]\to \mathbb{R}$ with $$\lim_{\N\to 0}f(\N)=-\infty.$$
Recalling Proposition \ref{prop:cur_function}, curvature functions lie in $\FK$.

For any $\N\in(0,\infty)$, let $D_k(\N)$ be the following set:
\begin{equation}\label{eq:DkN}
D_k(\N):=\left\{(\N_1,\ldots,\N_k): \N_j > 0 \,\,\text{ for all }\,\,j\,\,\text{and}\,\,\N_1+\cdots+\N_k=\N.\right\}.
\end{equation}

\begin{definition}[$*$-product]\label{defn:star_operation} For two functions $f_1, f_2\in\FK$, we define $f_1*f_2:(0,\infty]\to \mathbb{R}$ as follows. For any $\N\in (0,\infty)$, let
\begin{equation}\label{eq:star_oper_defn}
f_1*f_2(\N):=f_1(\N_1)=f_2(\N_2),
\end{equation}
where $(\N_1,\N_2)\in D_2(\N)$ is chosen such that
\begin{equation}\label{eq:star_oper_h}
f_1(\N_1)=f_2(\N_2).
\end{equation}
For $\N=\infty$, we define
\begin{equation}\label{eq:star_oper_infty}
f_1*f_2(\infty):=\lim_{\N\to \infty}f_1* f_2(\N).
\end{equation}
\end{definition}

\begin{remark}\label{remark:well_definedness}
Let us first verify that the \emph{$*$-product is well defined}. 

For two functions $f_1, f_2\in\FK$ and $\N\in~(0,\infty)$, there always exists a pair $(\N_1,\N_2)\in D_2(\N)$ satisfying (\ref{eq:star_oper_h}): We assume 
$f_1(\N/2) \le f_2(\N/2)$ without loss of generality (otherwise, we interchange the functions).
By the Intermediate Value Theorem, the monotonicity of $f_1$, and $\lim_{h \to \N/2} f_2(\N/2-h) = -\infty$, we can find
$0 < h < \N/2$ such that 
$$ f_1(\N_1 + \delta) = f_2(\N_2 - \delta). $$
Note that $(\N_1+\delta,\N_2-\delta) \in D_2(\N)$ is then the required pair.

Secondly, if there exist two pairs $(\N_1,\N_2), (\N_1', \N_2')\in D_2(\N)$ such that (\ref{eq:star_oper_h}) holds, then we have
\begin{equation}\label{eq:uniqueness}
f_1(\N_1)=f_1(\N_1') \,\, \text{and} \,\, f_2(\N_2)=f_2(\N_2').
\end{equation}
Due to the definition of $D_2(\N)$, we can choose $i,j$ such that $\{i,j\} = \{1,2\}$ and $\N_i\leq \N_i'$, $\N_j\geq \N_j'$. Then (\ref{eq:uniqueness}) follows from the monotonicity of $f_1$ and $f_2$.

Thirdly, the limit $\lim_{\N\to \infty}f_1*f_2(\N)$ exists. For any $0<\N_1\leq \N_2<\infty$, Let $(\N_{11}, \N_{12})\in D_2(\N_1), (\N_{21}, \N_{22})\in D_2(\N_2)$ be two tuples such that
$$f_1*f_2(\N_i)=f_1(\N_{i1})=f_2(\N_{i2}),\,\,\,i=1,2.$$
Since there always exists $j\in \{1,2\}$ with $\N_{1j}\leq \N_{2j}$, we have, from the monotonicity of $f_j$,
$$f_1*f_2(\N_1) = f_j(\N_{1j}) \le f_j(\N_{2j}) = f_1*f_2(\N_2) \le f_j(\infty).$$
That is, $f_1*f_2$ is a monotone non-decreasing function, bounded above by $f_j(\infty) < \infty$. Therefore, the limit exists.
\end{remark}

We have the following equivalent definition of the $*$-product.
\begin{proposition}\label{prop:product_min_max}
Let $f_1, f_2\in \FK$ and $F:(0,\infty]\to \mathbb{R}$. Then $F=f_1*f_2$ if and only if  we have, for any $\N_1,\N_2\in (0,\infty)$,
\begin{equation}\label{eq:product_min_max}
\min\{f_1(\N_1),f_2(\N_2)\}\leq F(\N_1+\N_2)\leq \max \{f_1(\N_1),f_2(\N_2)\},
\end{equation}
and $F(\infty)=\lim_{\N\to \infty}F(\N)$.
\end{proposition}
\begin{proof}
If $F=f_1*f_2$, there exists a tuple $(\N_1', \N_2')\in D_2(\N_1+\N_2)$ such that
$$F(\N_1+\N_2)=f_1(\N_1')=f_2(\N_2').$$
Recall there always exist $\{i,j\}= \{1,2\}$ such that $\N_i\leq \N_i'$ and $\N_j\geq \N_j'$. Then we obtain (\ref{eq:product_min_max}) by the monotonicity of $f_1,\ldots, f_k$. $F(\infty)=\lim_{\N\to \infty}F(\N)$ follows directly from Definition \ref{defn:star_operation}.

Let $F:(0,\infty]\to \mathbb{R}$ be a function satisfying (\ref{eq:product_min_max}) and $F(\infty)=\lim_{\N\to \infty}F(\N)$. By Remark \ref{remark:well_definedness}, for any $\N\in (0,\infty)$, there exists a pair $(\N_1,\N_2)\in D_2(\N)$, such that
$$f(\N_1)=f_2(\N_2).$$
Therefore, (\ref{eq:product_min_max}) implies $F(\N)=f_1*f_2(\N)$, finishing the proof.
\end{proof}

\begin{corollary}\label{cor:product_infty}
Let $f_1,f_2\in \FK$. Then for $\N\in (0,\infty)$, we have
\begin{equation}\label{eq:product_infty_1}
f_1*f_2(\N)\leq \min\{f_1(\N), f_2(\N)\}.
\end{equation}
For $\N=\infty$, we have
\begin{equation}\label{eq:product_infty_2}
f_1*f_2(\infty)=\min\{f_1(\infty), f_2(\infty)\}.
\end{equation}
\end{corollary}
\begin{proof}
For $\N \in (0,\infty)$, assume $f_1(\N) \le f_2(\N)$ without loss of generality. Note from the definition of the $*$-product that there exists $\N_1 \in (0,N]$ with
$$ f_1*f_2(\N) = f_1(\N_1). $$
Using monotonicity of $f_1$, we conclude that
$$ f_1*f_2(\N) \le f_1(\N) = \min\{ f_1(\N),f_2(\N) \}, $$
proving (\ref{eq:product_infty_1}). Furthermore, we have, by (\ref{eq:product_min_max}) and (\ref{eq:product_infty_1}),
$$\min\{f_1(\N/2),f_2(\N/2)\}\leq f_1*f_2(\N)\leq \min\{f_1(\N),f_2(\N)\}.$$
Letting $\N\to \infty$, we prove (\ref{eq:product_infty_2}).
\end{proof}

We further have the following basic properties of the $*$-product.
\begin{proposition}\label{prop:closeness}
Let $f_1,f_2\in \FK$. Then we have the following properties:
\begin{itemize}
\item [(i)] (Commutativity) $f_1*f_2=f_2*f_1$.
\item [(ii)](Closedness) $f_1*f_2\in \FK$.
\end{itemize}
\end{proposition}

\begin{proof}
(i) is obvious from the definition.

For (ii), recall first that we have shown monotonicity of $f_1*f_2$ in Remark \ref{remark:well_definedness}.

For any $\N\in (0,\infty)$, let $(\N_1,\N_2)\in D_2(\N)$ be the pair with
$$f_1*f_2(\N)=f_1(\N_1)=f_2(\N_2).$$
Then by monotonicity of $f_2$ and Proposition \ref{prop:product_min_max}, we have for any $\epsilon>0$,
$$f_2(\N_2)\leq f_1*f_2(\N+\epsilon)\leq f_2(\N_2+\epsilon).$$
By continuity of $f_2$, we see $f_1*f_2$ is continuous at $\N$. Since by definition, $f_1*f_2(\infty):=\lim_{\N\to \infty}f_1*f_2(\N)$, we prove that $f_1*f_2:(0,\infty]\to\mathbb{R}$ is continuous.

The property $\lim_{\N\to 0}f_1*f_2(\N)=-\infty$ follows directly from Proposition \ref{prop:product_min_max} and $\lim_{\N\to 0}f_i(\N)=-\infty$, $i=1,2$. This finishes the proof of $f_1*f_2\in \FK$.
\end{proof}

The following proposition shows associativity of the $*$-product and implies, therefore,
the well-definedness of $k$-fold products. 

\begin{proposition}\label{prop:higherprod}
For $f_1,f_2,f_3 \in \FK$, we have associativity
$$ (f_1*f_2)*f_3  = f_1*(f_2*f_3). $$
For functions $f_1,\dots,f_k \in \FK$, $k \ge 2$, and $\N \in (0,\infty)$, we have
$$ f_1*\cdots*f_k(\N) = f_1(\N_1) = \cdots f_k(\N_k), $$
for any tuple $(\N_1,\dots,\N_k) \in D_k(\N)$ with
$$ f_1(\N_1) = \cdots =f_k(\N_k). $$
Moreover, such a tuple always exists. For $\N = \infty$, we have
\begin{equation}\label{eq:higherprod_infty}
f_1*\cdots*f_k(\infty) = \min\{ f_1(\infty),\cdots,f_k(\infty) \}. 
\end{equation}
In particular, we have
\begin{equation} \label{eq:higherprod} 
\underbrace{f*\cdots *f}_{k}(\N)=f(\N/k), \quad \forall\,\, \N\in (0,\infty]. 
\end{equation}
\end{proposition}

\begin{proof} Let $f_1,f_2,f_3 \in \FK$ and $\N \in (0,\infty)$. Closedness implies that we have $f_1 * f_2 \in \FK$ and, by Definition \ref{defn:star_operation}, we have $\N_1',\N_3 > 0$ with $\N_1' + \N_3 = \N$ such that
$$ (f_1*f_2) * f_3 (\N) = f_1*f_2(\N_1') = f_3(\N_3). $$
Applying Definition \ref{defn:star_operation} again, we find $\N_1, \N_2 > 0$ with
$\N_1 + \N_2 = \N_1'$ such that
$$ f_1*f_2(\N_1') = f_1(\N_1) = f_2(\N_2). $$
Combining these facts, we obtain
$$ (f_1*f_2) * f_3 (\N) = f_1(\N_1) = f_2(\N_2) = f_3(\N_3) $$
with $\N_1 + \N_2 + \N_3 = \N$. Similarly, we show the same result for the product
$f_1*(f_2*f_3)$, which implies associativity for arguments $\N \in (0,\infty)$. The same arguments lead to the the statement about the $*$-product of $k$ functions $f_1,\dots,f_k \in \FK$ for arguments $\N \in (0,\infty)$ in the proposition. Associativity and the statements about $*$-products of $k$ functions for the argument $\N = \infty$ in the proposition are easy applications of (\ref{eq:product_infty_2}). 

Finally, \eqref{eq:higherprod} can be shown by taking the tuple $(\N/k,\ldots, \N/k)\in D_k(\N)$ when $\N\in (0,\infty)$, and by using \eqref{eq:higherprod_infty} when $\N=\infty$.
\end{proof}


\begin{proposition}\label{prop:product_distinguish}
Let $f_1,f_2,g\in \FK$. The following are equivalent:
\begin{itemize}
\item [(i)]$f_1*g\neq f_2*g$.
\item [(ii)] There exist $\N_0\in(0,\infty]$ such that
$f_1(\N_0)\neq f_2(\N_0)$ and $f_1(\N_0),f_2(\N_0)\leq g(\infty)$.
\end{itemize}
\end{proposition}
\begin{proof}
(ii) $\Rightarrow$ (i). W.o.l.g., we can suppose $f_1(\N_0)<f_2(\N_0)$. We can further suppose
$$f_2(\N_0)<g(\infty).$$
This is because, if $f_2(\N_0)=g(\infty)$, we can choose an $\N_0'<\N_0$, such that $f_2(\N_0')\in (f_1(\N_0),g(\infty))$. Then we have
$f_1(\N_0')\leq f_1(\N_0)<f_2(\N_0')<g(\infty)$.

When $\N_0\in(0,\infty)$, let
$\M_0=\min\{\M'\in (0,\infty): g(\M')=f_2(\N_0)\}.$
Then we have
$$f_2*g(\N_0+\M_0)=f_2(\N_0)=g(\M_0).$$
By assumption, we know
$f_1(\N_0)<g(\M_0)$.
Hence, there exists $\epsilon>0$ such that
$$f_1*g(\N_0+\M_0)=f_1(\N_0+\epsilon)=g(\M_0-\epsilon).$$
By the choice of $M_0$, we have
$$f_1*g(\N_0+\M_0)=g(\M_0-\epsilon)<g(\M_0)=f_2*g(\N_0+\M_0).$$

When $\N_0=\infty$, we have, by Corollary \ref{cor:product_infty},
$$f_1*g(\infty)=f_1(\infty)<f_2(\infty)=f_2*g(\infty).$$
This finishes the proof of $(i)$ assuming (ii).

(i) $\Rightarrow$ (ii). We prove this by showing its contrapositive. Suppose, for any $\N\in (0,\infty]$, we have
\begin{equation}\label{eq:product_assumption}
f_1(\N)=f_2(\N)\,\,\text{or}\,\,\max\{f_1(\N),f_2(\N)\}>g(\infty).
\end{equation}
%

For any $\N\in (0,\infty)$, there exists $\M_1,\M_2\in (0,\N)$ such that
$$f_i*g(\N)=f_i(\M_i)=g(\N-\M_i),\,\,i=1,2.$$
If $\M_1=\M_2$, then $f_1*g(\N)=f_2*g(\N)$. Otherwise, we suppose $\M_1<\M_2$ without loss of generality. By monotonicity, we have $$f_1(\M_1)=g(\N-\M_1)\geq g(\N-\M_2)=f_2(\M_2).$$
This implies $g(\infty)\geq f_1(\M_1)\geq f_2(\M_2)\geq f_2(\M_1)$.
By (\ref{eq:product_assumption}), we have $f_1(\M_1)=f_2(\M_1)=f_2(\M_2)$ and therefore, $f_1*g(\N)=f_2*g(\N)$.

Letting $\N\to \infty$, we obtain $f_1*g(\infty)=f_2*g(\infty)$. This finishes the proof.
\end{proof}

\begin{example}\label{example:Product_cur_sharp}
Given $a\in \mathbb{R}, b_1,b_2\in (0,\infty)$, let $f_1,f_2\in FK$ be the following two functions:
$$f_1(\N):=a-\frac{b_1}{\N}\,\,\,\text{and}\,\,\,f_2(\N):=a-\frac{b_2}{\N}.$$
Then we have
\begin{equation}\label{eq:Product_cur_sharp}
f_1*f_2(\N)=a-\frac{b_1+b_2}{\N},\,\,\,\forall \,\,\N\in (0,\infty].
\end{equation}
This can be verified by taking the tuple $\left(\frac{b_1}{b_1+b_2}\N, \frac{b_2}{b_1+_2}\N\right)\in D_2(\N)$ when $\N\in(0,\infty)$.
\end{example}

\subsection{Main results}
Given two locally finite simple graphs $G_1=(V_1, E_1)$ and
$G_2=(V_2, E_2)$, their Cartesian product $G_1\times
G_2=(V_1\times V_2, E_{12})$ is a locally finite graph with vertex
set $V_1\times V_2$ and edge set $E_{12}$ given by the following
rule. Two vertices $(x_1,y_1), (x_2, y_2)\in V_1\times V_2$ are
connected by an edge in $E_{12}$ if $$x_1=x_2, \{y_1, y_2\}\in E_2\,\,\,\text{ or }\,\,\,\{x_1, x_2\}\in E_1, y_1=y_2.$$

We have the following result.

\begin{theorem}\label{thm:CartesianFormula}
Let $G_i=(V_i,E_i), i=1,2$ be two locally finite simple graphs. Then for any $x\in V_1$ and $y\in V_2$, we have
\begin{equation}
\K_{G_1\times G_2, (x,y)}=\K_{G_1,x}*\K_{G_2,y}.
\end{equation}
\end{theorem}

Theorem \ref{thm:CartesianFormula} is derived from the following two Theorems, using Proposition \ref{prop:product_min_max}, .

\begin{theorem}[\cite{LP14}]\label{thm:CarProlb}
Let $G_i=(V_i,E_i), i=1,2$ be two locally finite simple graphs. Then we have, for any $x\in V_1, y\in V_2$ and $\N_1, \N_2\in (0,\infty]$,
\begin{equation}
\K_{G_1\times G_2, (x,y)}(\N_1+\N_2)\geq \min\{\K_{G_1,x}(\N_1), \K_{G_2,y}(\N_2)\}.
\end{equation}
\end{theorem}
Theorem \ref{thm:CarProlb} has been shown in \cite[Theorem 2.5]{LP14}.
In this section, we further prove the following estimate.
\begin{theorem}\label{thm:CarProub}
Let $G_i=(V_i,E_i), i=1,2$ be two locally finite simple graphs. Then we have, for any $x\in V_1, y\in V_2$ and $\N_1, \N_2\in (0,\infty]$,
\begin{equation}\label{eq:CarProub}
\K_{G_1\times G_2, (x,y)}(\N_1+\N_2)\leq \max \{\K_{G_1,x}(\N_1), \K_{G_2,y}(\N_2)\}.
\end{equation}
\end{theorem}

We first recall a technical lemma from \cite[Lemma 2.6]{LP14}.
Let $F:V_1\times V_2\rightarrow\mathbb{R}$ be a function on the
product graph. For fixed $y\in V_2$, we will write
$F_y(\cdot):=F(\cdot,y)$ as a function on $V_1$. Similarly,
$F^x(\cdot):=F(x,\cdot)$. It is straightforward to check
\begin{equation}\label{eq:Car_D_Gamma}
\D F(x,y)=\D F_y(x)+\D F^x(y)\,\,\,\,\text{and}\,\,\,\,\Gamma(F)(x,y)=\Gamma(F_y)(x)+\Gamma(F^x)(y).
\end{equation}

\begin{lemma}[\cite{LP14}]\label{lemma:cartesianproduct}
  For any function $F:V_1\times V_2\rightarrow\mathbb{R}$ and any
  $(x,y)\in V_1\times V_2$, we have
  \begin{align}
    \Gamma_2(F)(x,y)= &\Gamma_2(F_y)(x)+\Gamma_2(F^x)(y)\notag\\
    &+\frac{1}{2}\sum_{x_i\sim x}\sum_{y_k\sim y}\left(F(x_i,y_k)-F(x,y_k)-F(x_i,y)+F(x,y)\right)^2.
  \end{align}
  where the operators $\Gamma_2$ are understood to be on different
  graphs according to the functions they are acting on.
\end{lemma}

\begin{proof}[Proof of Theorem \ref{thm:CarProub}] Let $f_1:V_1\to \mathbb{R}$ be a function with $\Gamma(f_1)(x)\neq 0$, such that
\begin{equation}\label{eq:CD_G_1}
\Gamma_2(f_1)(x)=\frac{1}{\N_1}(\D f_1(x))^2+\K_{G_1,x}(\N_1)\Gamma(f_1)(x).
\end{equation}
Let $f_2:V_2\to \mathbb{R}$ be a function with $\Gamma(f_2)(x)\neq 0$, such that
\begin{equation}\label{eq:CD_G_2}
\Gamma_2(f_2)(x)=\frac{1}{\N_2}(\D f_2(x))^2+\K_{G_2,x}(\N_2)\Gamma(f_2)(x).
\end{equation}
Let $c_1, c_2$ be two nonnegative real number. We define a function $F:V_1\times F_2\to \mathbb{R}$ such that
\begin{equation}\label{eq:Car_F_from_f_i}
F_y(\cdot):=c_1f_1(\cdot)+c_2f_2(y),\,\,\,\,F^x(\cdot):=c_2f_2(\cdot)+c_1f_1(x),
\end{equation}
and, for any $\{x_i,x\}\in E_1$ and any $\{y_k,y\}\in E_2$,
\begin{equation}\label{eq:Car_Fik_from_f_i}
F(x_i,y_k):=c_1f_1(x_i)+c_2f_2(y_k).
\end{equation}
Then we obtain, from Lemma \ref{lemma:cartesianproduct},
\begin{equation*}
\Gamma_2(F)(x,y)=\Gamma_2(F_y)(x)+\Gamma_2(F^x)(y)=c_1^2\Gamma_2(f_1)(x)+c_2^2\Gamma_2(f_2)(y).
\end{equation*}
Inserting (\ref{eq:CD_G_1}) and (\ref{eq:CD_G_2}), we have
\begin{align}
\Gamma_2(F)(x,y)=\frac{1}{\N_1}\D&(c_1f_1(x))^2+\frac{1}{\N_2}\D(c_2f_2(y))^2\notag\\
&+\K_{G_1,x}(\N_1)\Gamma(c_1f_1)(x)+\K_{G_1,y}(\N_2)\Gamma(c_2f_2)(y).\label{eq:Gamma2Fxy}
\end{align}
If $0<\N_1,\N_2<\infty$ and $\D f_1(x)$, $\D f_2(y)$ are not simultaneously zero, say $\D f_1(x)\neq 0$, we set
$$c_1=\frac{\N_1^2\D f_2(y)}{\N_2^2\D f_1(x)}, \,\,\,\, c_2=1.$$
Then we have
$$\frac{\N_2}{\N_1}c_1\D f_1(x)=\frac{\N_1}{\N_2}c_2\D f_2(y),$$
and hence
\begin{equation}\label{eq:Delta_Cauchy_Schwarz}
\frac{1}{\N_1}\D(c_1f_1(x))^2+\frac{1}{\N_2}\D(c_2f_2(y))^2=\frac{1}{\N_1+\N_2}\left(c_1\D f_1(x)+c_2\D f_2(x)\right)^2
\end{equation}
Using (\ref{eq:Car_D_Gamma})and (\ref{eq:Delta_Cauchy_Schwarz}), we derive from (\ref{eq:Gamma2Fxy})
\begin{equation}
\Gamma_2(F)(x,y)\leq \frac{1}{\N_1+\N_2}\left(\D F(x,y)\right)^2+\max\left\{\K_{G_1,x}(\N_1),\K_{G_1,y}(\N_2)\right\}\Gamma(F)(x,y).
\end{equation}
This implies (\ref{eq:CarProub}) in this case.

If, otherwise, $\D f_1(x)=\D f_2(y)=0$ or at least one of $\N_1,\N_2$ equals $\infty$, say, $\N_1=\infty$, we set
$$c_1=1,\,\,\,\,c_{2}=0.$$
Then we show (\ref{eq:CarProub}) in this case similarly.
\end{proof}

By Corollary \ref{cor:product_infty}, the following result is a straightforward consequence of Theorem \ref{thm:CartesianFormula}.
\begin{corollary}
Let $G_i=(V_i,E_i), i=1,2$ be two locally finite  simple graphs. Then we have, for any $x\in V_1, y\in V_2$ and $\N\in (0,\infty]$,
\begin{equation}
\K_{G_1\times G_2, (x,y)}(\N)\leq \min\{\K_{G_1,x}(\N), \K_{G_2,y}(\N)\}.
\end{equation}
When $\N=\infty$, the equality holds, i.e.,
\begin{equation}
\K_{G_1\times G_2, (x,y)}(\infty)=\min \{\K_{G_1,x}(\infty), \K_{G_2,y}(\infty)\}.
\end{equation}
\end{corollary}

By Example \ref{example:Product_cur_sharp}, we derive the following result.
\begin{corollary}\label{cor:Cartesian_cur_sharp}
Let $G_i=(V_i,E_i), i=1,2$ be two locally finite  simple $d$-regular, $\infty$-curvature sharp graphs. If either $G_1, G_2$ are both triangle free, or, $G_1=G_2$ with $\#_{\D}(x)\equiv const.\,\,\forall x$,
then $G_1\times G_2$ is also $\infty$-curvature sharp.
\end{corollary}
\begin{proof} If $G_1,G_2$ are both triangle free, we have for $x\in V_1$, $y\in V_2$
$$\K_{G_1,x}(\N)=2-\frac{2d_x}{\N},\,\,\,\K_{G_2,y}(\N)=2-\frac{2d_y}{\N}.$$
Theorem \ref{thm:CartesianFormula} and  Example \ref{example:Product_cur_sharp} tells
$$\K_{G_1\times G_2, (x,y)}=2-\frac{2(d_x+d_y)}{\N}.$$
That is, $(x,y)$ is $\infty$-curvature sharp, since $G_1\times G_2$ is also triangle free. The other case can be shown similarly.
\end{proof}

Corollary \ref{cor:Cartesian_cur_sharp} implies, for example, the Cartesian product of the crown graph $\mathrm{Crown}(n,n)$ and the complete bipartite graph $K_{m,m}$ is $\infty$-curvature sharp. The Cartesian product of $K_n,n\geq 2$ with itself is $\infty$-curvature sharp. In fact, the latter is the line graph of $K_{n,n}$, which will be discussed in detail in the next subsection.



\subsection{Examples}\label{subsec:examples}
In this subsection, we present examples illustrating how Theorem \ref{thm:CartesianFormula} can be used to calculate explicit curvature functions.
\begin{example}[Hypercubes]\label{example:hypercube} An $n$-dimensional hypercube $Q^n=(V,E)$, $n\geq 1, n\in \mathbb{Z}$, is the Cartesian product $\underbrace{K_2\times \cdots\times K_2}_{n}$. For any $x\in V$, we have
\begin{equation}\label{eq:hypercube}
\K_{Q^n,x}(\N)=2-\frac{2n}{\N}\,\,\,\,\forall \,\,\N\in (0,\infty].
\end{equation}
\end{example}
\begin{proof}
Recall from Example \ref{example:Kn} that the two vertices in $K_2$ has the same curvature function
$$\K_{K_2,\cdot}(\N)=2-\frac{2}{\N}\,\,\,\,\forall \,\,\N\in (0,\infty].$$
By Theorem \ref{thm:CartesianFormula}, and the associativity of $*$-product, we know $\K_{Q^n,x}=\K_{K_2,\cdot}*\cdots *\K_{K_2,\cdot}$. Hence, (\ref{eq:hypercube}) follows from \eqref{eq:higherprod}.
\end{proof}

Given a graph $G$, its \emph{line graph} $L(G)$ is defined in the following way:
Each vertex of L(G) represents an edge of G;
Two vertices of L(G) are adjacent if and only if their corresponding edges are incident in $G$. For other names of line graphs and references, we refer to \cite[Chapter 8]{Harary69}.

\begin{example}[Line graphs of complete bipartite graphs]\label{example:LKmn} The line graph $L(K_{m,n})=(V,E)$ of $K_{m,n}$ is the Cartesian product of $K_m$ and $K_n$. When one of $n,m$ equals $1$, the line graph is a cycle, whose curvature functions has been given in Example \ref{example:Cn}.
When $n,m\geq 2$, we have, for any $x\in V,\N\in(0,\infty]$,
\begin{align}
\K_{L(K_{m,n}),x}&(\N)\notag\\
=&\frac{n+m+4}{4}-\frac{4(n+m-2)+\sqrt{16(n+m-2)^2+(n-m)^2\N(\N-8)}}{4\N}.\label{eq:LKmn}
\end{align}
\end{example}

We remark that $L(K_{m,n})$ is not $S_1$-out regular at any vertex when $n\neq m$, in which case Theorem \ref{thm:OutRegularFormula} can not be applied.

\begin{proof}
Recall from Example \ref{example:Kn}, each vertex in $K_n$, $n\geq 2$ has the same curvature function
$$\K_{K_n,\cdot}(\N)=\frac{n+2}{2}-\frac{2(n-1)}{\N}\,\,\,\,\forall\,\,\N\in(0,\infty].$$
Theorem \ref{thm:CartesianFormula} tells
$$\K_{L(K_{m,n}),x}=\K_{K_n,\cdot}*\K_{K_m,\cdot}.$$
When $n=m$, we have, by \eqref{eq:higherprod},
$$\K_{L(K_{n,n}),x}(\N)=\K_{K_n,\cdot}(\N/2)=\frac{n+2}{2}-\frac{4(n-1)}{\N}\,\,\,\,\forall\,\,\N\in(0,\infty].$$
It remains to prove (\ref{eq:LKmn}) when $n\neq m$.
It is sufficient to check the case $m<n$, and $\N\in (0,\infty)$.
We are going to find $|h|<\frac{\N}{2}$, such that $\K_{K_n,\cdot}(\N/2-h)=\K_{K_m,\cdot}(\N/2+h)$, i.e.,
$$\frac{n+2}{2}-\frac{4(n-1)}{\N-2h}=\frac{m+2}{2}-\frac{4(m-1)}{\N+2h}.$$
Solving $h$, we arrive at
$$4(n-m)h^2+16(n+m-2)h+(n-m)(8\N-\N^2)=0.$$
Recall from Remark \ref{remark:well_definedness} such a $h$ satisfying the above equation with $|h|<\frac{\N}{2}$ always exists. Hence we can pick the one , with smaller absolute value, of two solutions of the above equation:
$$h=\frac{-4(n+m-2)+\sqrt{16(n+m-2)^2+(n-m)^2(\N^2-8\N)}}{2(n-m)}.$$
Hence we obtain
\begin{align}
&\K_{L(K_{m,n}),x}(\N)=\K_{K_n,\cdot}(\N/2-h)\notag\\
=&\frac{n+2}{2}-\frac{4(n-1)(n-m)}{\N(n-m)+4(n+m-2)-\sqrt{16(n+m-2)^2-(n-m)^2\N(8-\N)}}\notag\\
=&\frac{n+2}{2}-\frac{\N(n-m)+4(n+m-2)+\sqrt{16(n+m-2)^2+(n-m)^2\N(\N-8)}}{4\N}.\notag
\end{align}
This proves (\ref{eq:LKmn}) for $m<n$ and $\N\in (0,\infty]$, finishing the proof.
\end{proof}

\begin{remark}\label{remark:iso_curvature_function}
Curvature functions are useful isomorphism invariants.
We see, from Examples \ref{example:Kn} and \ref{example:LKmn}, although vertices in the complete graph $\K_n$ and the line graph of $K_{n,n}$, $n\geq 2$, share the same value at $\infty$: $\K_{K_n,\cdot}(\infty)=\K_{L(K_{n,n}),\cdot}(\infty)=(n+2)/2$, their curvature functions are distinguished.

But it can happen that non-isomorphic graphs share the same curvature function. We see, from Examples \ref{example:Kmn} and \ref{example:hypercube}, vertices in  the complete bipartite graph $K_{n,n}$ and the $n$-dimensional hypercube $Q^n$, $n\geq 1$, share the same curvature function:
$$f(\N)=2-\frac{2n}{\N}\,\,\,\,\forall \,\,\N\in(0,\infty].$$
By Corollaries \ref{cor:read_cur_dx}, \ref{cor:read_cur_av1},and \ref{cor:Outregular_characterization} we know, from their curvature function, the vertices in $K_{n,n}$ and $Q^n$ share the same vertex degree $n$, the same average $S_1$-out degree $n-1$, and they are all $S_1$-out regular.
\end{remark}

\section{Global structure of graphs satisfying $CD(0,\infty)$}\label{gc}

Embarrassingly, we know very little about the global structure of
graphs satisfying $CD(K,\infty)$ for a fixed $K \ge 0$. Let us start with the
following two fundamental facts:

\begin{itemize}
\item[(a)] The class $CD(K,\infty)$ is closed under taking Cartesian products.
\item[(b)] All abelian Cayley graphs lie in the class $CD(0,\infty)$ (\cite{CY96, LY10,KKRT16}).
\end{itemize}

Some local properties of graphs in the class $CD(0,\N)$ for finite
dimension $\N$ were already discussed in Section
\ref{sec:graphsCD0N}. One might hope that every graph in
$CD(0,\infty)$ lies also in $CD(0,\N)$ for some large enough finite
$\N$, but we saw that the complete bipartite graph $K_{2,6}$ is a
counterexample (see example \ref{example:K26}). Other local
obstructions to satisfy $CD(0,\infty)$ were discussed in Section
\ref{section:connComp}.

It is natural to investigate whether known global results for Riemannian manifolds
with lower bounds on their Ricci curvature have analogues in the graph theoretical setting.
The following result on the global structure of positively curved graphs can be found in \cite{LMP17} and can be viewed as some analogue of a Bonnet-Myers' Theorem. (It was included as a conjecture  in an earlier version of this paper before it became a theorem; back then it was Conjecture 8.1.) 

\begin{theorem}[Corollary 2.2, \cite{LMP17}]
Let $G = (V,E)$ be a graph satisfying $CD(K,\infty)$ for $K>0$. If $G$ has bounded vertex degree $\le d_{\max}$, then $G$ is a finite graph and we have
$${\rm{diam(G)}}\leq \frac{2d_{\max}}{K}.$$
\end{theorem}

Examples of finite regular graphs in $CD(2,\infty)$
with increasing diameters are the hypercubes $Q_n$ which are Cartesian
products of $n$ copies of $K_2$. Both the vertex degree and diameter
of $Q_n$ is $n$, showing that the diameter bound in the above theorem must depend on $d_{\max}$.

%

The following conjecture can be viewed as some analogue of Bishop's
Comparison Theorem:

\begin{conjecture} \label{conj:Bish} Let $d_{max} \in \NN$. Then
  every graph $G=(V,E)$ with $d_x \le d_{\max}$ for all $x \in V$
  satisfying $CD(0,\infty)$ has polynomial volume growth. Moreover,
  there are constants $C_1, C_2 > 0$, only depending on $d_{max}$
  that for all $x \in V$ and $r \in \NN$.
  $$ | B_r(x) | \le C_1 (1 + r^{C_2}). $$
\end{conjecture}

Examples of infinite regular graphs in
$CD(0,\infty)$ with polynomial volume growth are abelian Cayley
graphs. The dependence on the maximal degree $d_{max}$
follows easily from the construction of taking Cartesian products.

The following conjecture would be a direct consequence of Conjecture
\ref{conj:Bish}, since the Cheeger isoperimetric constants $h(G_n)$ of
increasing $d$-regular graphs $G_n$ with uniform upper polynomial
volume growth must tend to zero.

\begin{conjecture} \label{conj:exp} Let $d \in \NN$. No infinite
  family of finite increasing $d$-regular graphs satisfying
  $CD(0,\infty)$ can be a family of expander graphs.
\end{conjecture}

A corresponding statement for $d$-regular abelian Cayley graphs is
well known (\cite{AR94}). Therefore, it is natural to ask whether $CD(0,\infty)$ is
really a substantially larger class than the class of all abelian
Cayley graphs. To finish this section, we present an infinite family
of finite increasing $6$-regular non-Cayley graphs $G_n$ satisfying
$CD(0,\infty)$.

\begin{example}\label{example:inffam6reg}
  [Infinite family of $6$-regular non-Cayley graphs
  satisfying $CD(0,\infty)$]\label{example:family_non_Cay} Let $G = (V,E)$ be the $4$-regular graph
  introduced in Example \ref{ex:norbexample}. The curvature
  calculations there showed that $G$ satisfies $CD(5/2,\infty)$. Let
  $G_n$ be the Cartesian product of $G$ with the cycle $C_n$. Since
  $C_n$ satisfies $CD(0,\infty)$, the graphs $G_n$ are a family of
  increasing $6$-regular graphs satisfying also $CD(0,\infty)$. Each of the graphs $G_n$ is not a Cayley graph.
\end{example}
\begin{proof} From Example \ref{example:Cn}, we see $0\leq \K_{C_n,\cdot}(\infty)$.
Recall from Example \ref{ex:norbexample},
there are vertices in $G$ with different curvature functions, i.e.,
$$\K_{G,x_1}(\N)=\frac{5}{2}-\frac{8}{\N}\,\,\text{and}\,\,\K_{G,y_1}(\N)=5-\frac{8+\sqrt{21\N^2+72\N+64}}{2\N}.$$
Note that $\K_{G,x_1}(1)<\K_{G,y_1}(1)<0\leq \K_{C_n,\cdot}(\infty)$. We have, by Proposition \ref{prop:product_distinguish},
$$\K_{G,x_1}*\K_{C_n,\cdot}\neq \K_{G,y_1}*\K_{C_n,\cdot}.$$
  Therefore, there
  are vertices in $G_n = G \times C_n$ with different curvature
  functions. This rules out
  that any of the graphs $G_n$ is a Cayley graph.
\end{proof}

\section{Curvature and spectral gaps at $\mathbf{S_1}$-out regular
  vertices}
\label{sec:S1outreg}

Let $G=(V,E)$ be a locally finite simple graph and $x \in V$. We
assume henceforth that $x$ is an $S_1$-out regular vertex of degree
$d=d_x \ge 2$, i.e., $d_{y_j}^+ = av_1^+(x)$ for all neighbours
$y_1, y_2, \dots, y_d$ of $x$ in $G$. Recall from Corollary
\ref{cor:full_sharpness_characterization} that $S_1$-out regularity of
$x$ is equivalent to $\N$-curvature sharpness for some
$\N \in (0,\infty]$.

Recall also that we provided an explicit expression for the curvature
function at an $S_1$-out regular vertex $x$ in Theorem
\ref{thm:OutRegularFormula} in terms of the lowest eigenvalue of a
matrix $\P_\infty(x)$. Our aim in this section is to express
$\P_\infty(x)$ via suitable Laplacians, defined on graphs with vertex
set $S_1(x)$. Our main result is Theorem \ref{thm:curvs1out}
below. This theorem implies a quantitative version of Theorem
\ref{thm:conncomp} in the special case of $S_1$-out regularity (see
Corollary \ref{cor:curvS1out}). We will also use this result to deal
with the example of Johnson graphs and will finish this section with
specific cardinality estimates of $2$-spheres $S_2(x)$ of $S_1$-out
regular vertices $x \in V$.

Let $\Delta_{S_1(x)}$ be the non-norma\-lized Laplacian of the subgraph
$S_1(x)$ induced by the vertices $\{ y_1,\dots, y_d \}$, i.e., written
as an operator,
$$ \Delta_{S_1(x)} f(y_i) = \left( \sum_{j \neq i} w_{y_i y_j}f(y_j) \right) -
d_{y_i}^0 f(y_i). $$
Let $S_1'(x)$ be the graph with the same vertex set
$\{ y_1,\dots, y_d \}$ and an edge between $y_i$ and $y_j$ iff
$\# \{ z \in S_2(x) \mid y_i \sim z \sim y_j \} \ge 1$, where $\sim$
describes adjacency in the original graph $G$. We introduce the
following weights $w_{y_i y_j}'$ on the edges of $S_1'(x)$:
$$ w_{y_i y_j}' = \sum_{z \in S_2(x)} \frac{w_{y_i z}w_{z y_j}}{d_z^-}. $$
The corresponding weighted Laplacian is then given by
$$ \Delta_{S_1'(x)} f(y_i) = \left( \sum_{j \neq i} w_{y_i y_j}' f(y_j) \right)
- d_{y_i}' f(y_i), $$
where
$$ d_{y_i}' = \sum_{z \in S_2(x)} \frac{w_{y_i z}}{d_z^-} \sum_{j \neq i} w_{z y_j} =
av_1^+(x) - \sum_{z \in S_2(x), z \sim y_i} \frac{1}{d_z^-}. $$ Let
$S_1''(x) = S_1(x) \cup S_1'(x)$, i.e., the vertex set of $S_1''(x)$
is $\{ y_1,\dots, y_{d_x} \}$ and the edge set is the union of the
edge sets of $S_1(x)$ and $S_1'(x)$. Then the sum
$\Delta_{S_1(x)} + \Delta_{S_1'(x)}$ can be understood as the weighted
Laplacian $\Delta_{S_1''(x)}$ on $S_1''(x)$ with weights
$w'' = w + w'$. Note that all our Laplacians $\Delta$ are defined on
functions on the vertex set of $S_1(x)$. They are non-positive
operators, by definition, and we refer to their eigenvalues $\lambda$
as solutions of $\Delta f + \lambda f = 0$, to make these eigenvalues
non-negative. We denote and order these eigenvalues (with their
multiplicities) by
$$ 0 = \lambda_0(\Delta) \le \lambda_1(\Delta) \le \cdots \le
\lambda_{d-1}(\Delta). $$
The second-smallest eigenvalue $\lambda_1(\Delta) \ge 0$ is also called
the \emph{spectral gap} of $\Delta$ and plays an important role in
spectral graph theory. With these operators in place, we have
$$ \P_\infty(x) = - 4 \Delta_{S_1''(x)} - 2 (d I_d - J_d), $$
and Theorem \ref{thm:OutRegularFormula} implies that
$$ \K_{G,x}(\N) = \frac{3+d-av_1^+(x)}{2} - \frac{2d}{\N} +
2 \underbrace{\lambda_{\min}\left( -\Delta_{S_1''(x)} + \left(
      \frac{1}{\N} - \frac{1}{2} \right)(d I_d - J_d) \right)}_{\le
  0}. $$
Then we have
$$ \lambda_{\min}\left( -\Delta_{S_1''(x)} + \left(
      \frac{1}{\N} - \frac{1}{2} \right)(d I_d - J_d) \right) =
\begin{cases} 0, & \text{if $\lambda_1 \ge \frac{d}{2} - \frac{d}{\N}$,} \\
\lambda_1 + \left( \frac{d}{\N} - \frac{d}{2} \right), & \text{if $\lambda_1 < \frac{d}{2} - \frac{d}{\N}$.} \end{cases}
$$
This leads directly to the following result.

\begin{theorem} \label{thm:curvs1out} Let $G=(V,E)$ be a locally finite
  simple graph and $x \in V$ be an $S_1$-out regular vertex of degree
  $d \ge 2$. Let $\Delta_{S_1''(x)}$ be the weighted Laplacian defined
  above with second-smallest eigenvalue
  $\lambda_1= \lambda_1(\Delta_{S_1''(x)})$.
  \begin{itemize}
  \item[(a)] The case $\lambda_1 \ge \frac{d}{2}$ is equivalent to
  $\infty$-curvature sharpness of $x$. Then we have
  $$ \K_{G,x}(\N) = \frac{3+d-av_1^+(x)}{2} - \frac{2d}{\N}. $$
  \item[(b)] If $\lambda_1 < \frac{d}{2}$, we have
  $$ \K_{G,x}(\N) = \begin{cases}
    \frac{3+d-av_1^+(x)}{2} - \frac{2d}{\N}, & \text{if $\N \le \frac{2d}{d-2\lambda_1}$,} \\
    \frac{3-d-av_1^+(x)}{2} + 2 \lambda_1, & \text{if
      $\N > \frac{2d}{d-2\lambda_1}$.} \end{cases}
  $$
  \end{itemize}
\end{theorem}

\begin{remark} In certain cases, it is advantageous to use the
  decomposition $\Delta_{S_1''(x)} = \Delta_{S_1(x)} + \Delta_{S_1'(x)}$ and to
  investigate the spectral gaps $\lambda_1(\Delta_{S_1(x)})$ and
  $\lambda_1( \Delta_{S_1'(x)})$ separately, making use of specific
  geometric properties of the graphs $S_1(x)$ and $S_1'(x)$. The following
  examples illustrate this point.
\end{remark}

\begin{examples} \label{examples:spectralgaps}
  Let $x$ be a $S_1$-out regular vertex of the graph $G=(V,E)$ of
  degree $d \ge 2$.

  (a) In our first example we assume $S_1(x) = K_d$ and
  $S_2(x) = \emptyset$. (This means that $G = K_{d+1}$.) In this case,
  $S_1'(x)$ is totally disconnected and
  $\Delta_{S_1''(x)} = \Delta_{S_1(x)}$. The eigenvalues of the
  non-normalized Laplacian $\Delta_{K_n}$ on the complete graph $K_d$
  are known to be $\lambda_0(\Delta_{K_n}) = 0$ and
  $\lambda_1(\Delta_{K_n}) = \dots = \lambda_{d-1}(\Delta_{K_n}) = d$
  and, therefore, Theorem \ref{thm:curvs1out} yields
  $$ \K_{G,x}(\N) = \frac{3+d}{2} - \frac{2d}{\N}, $$
  confirmed by Example \ref{example:Kn}.

  (b) Next, we assume $S_1(x)$ to be totally disconnected (i.e.,
  $\#_{\Delta}(x) = 0$) and $|S_2(x)| = {d \choose 2}$, where every
  pair of vertices $y,y' \in S_1(x)$ is connected to precisely one
  vertex $z \in S_2$ with $d_z^- = 2$. Then
  $\Delta_{S_1''(x)} = \Delta_{S_1'(x)}$ and $\Delta_{S_1'(x)}$ agrees
  with $\frac{1}{2} \Delta_{K_n}$. Therefore, we have
  $\lambda_1(\Delta_{S_1''(x)}) = d/2$ and $av_1^+(x) = d-1$, and
  Theorem \ref{thm:curvs1out} yields
  $$ \K_{G,x}(\N) = 2 - \frac{2d}{\N}. $$

  (c) Finally, we assume $S_1(x)$ to be totally disconnected and
  $|S_2(x)| = 1$, i.e., $S_2(x) = \{ z \}$, where every vertex
  $y \in S_1(x)$ is connected to $z$. Then
  $\Delta_{S_1''(x)} = \Delta_{S_1'(x)}$ and $\Delta_{S_1'(x)}$ agrees
  with $\frac{1}{d} \Delta_{K_n}$. Therefore, we have
  $\lambda_1(\Delta_{S_1''(x)}) = 1 \le d/2$ and $av_1^+(x) = 1$, and
  Theorem \ref{thm:curvs1out} yields
  $$ \K_{G,x}(\N) = \begin{cases} \frac{2+d}{2} - \frac{2d}{\N} &
  \text{if $\N \le \frac{2d}{d-2}$,} \\ \frac{6-d}{2} &
  \text{if $\N > \frac{2d}{d-2}$.} \end{cases} $$
\end{examples}

An immediate consequence of Theorem \ref{thm:curvs1out} is a stronger
quantitative version of Theorem \ref{thm:conncomp} from Section
\ref{section:connComp} in the case of $S^1$-out regularity.

\begin{corollary} \label{cor:curvS1out} Let $G=(V,E)$ be a locally finite simple graph and
  $x \in V$ an \emph{$S_1$-out regular} vertex. If $\mathring{B}_2(x)$ has
  more than one connected component then
  $$ \K_{G,x}(\N) =  \begin{cases}
    \frac{3+d-av_1^+(x)}{2} - \frac{2d}{\N}, & \text{if $\N \le 2$,} \\
    \frac{3-d-av_1^+(x)}{2}, & \text{if
      $\N > 2$.} \end{cases}
  $$
  In particular, we have
  $$ \K_{G,x}(\infty) = \frac{3-d-av_1^+(x)}{2}, $$
  and $\K_{G,x}(\infty) < 0$ except for the four cases (a), (c), (d), and (e)
  presented in Figure \ref{Fexceptions}.
\end{corollary}

\begin{proof} Note that if $\mathring{B}_2(x)$ has more than one
  connected component then $d \ge 2$ and $S_1''(x)$ has also more than
  one connected component. This implies that the eigenvalue $0$ if
  $\Delta_{S_1''(x)}$ has multiplicity at least $2$. Then we can apply
  Theorem \ref{thm:curvs1out} with $\lambda_1 = 0 < d/2$. The case
  analysis for $3-d-av_1^+(x) \ge 0$ is very similar as in the proof
  of Theorem \ref{thm:conncomp}.
\end{proof}

Another straightforward consequence of Theorem \ref{thm:conncomp} is
the following.

\begin{corollary} \label{cor:poscurvS1reg}Let $G=(V,E)$ be a locally finite
  simple graph and $x \in V$ be an $S_1$-out regular vertex of degree
  $d \ge 2$. Let $\lambda_1 \ge 0 $ be the second-smallest eigenvalue
  of $\Delta_{S_1''(x)}$. Then $\K_{G,x}(\infty) \ge 0$ is equivalent to
  \begin{equation} \label{eq:nonnegcurveq} av_1^+(x) \le 3+d \quad
    \text{and} \,\,\, \lambda_1 \ge \frac{d+av_1^+(x)-3}{4}.
  \end{equation}
  In particular, we have $\K_{G,x}(\infty) \ge 0$ if
  $av_1^+(x) \le 3+d$ and $\lambda_1 \ge d/2$. (Note that the second
  condition is equivalent to $\infty$-curvature sharpness.)
\end{corollary}

\begin{remark} The importance of Corollary \ref{cor:poscurvS1reg} is
  that it relates \emph{non-negativity of the curvature} at a vertex
  $x \in V$ to a large enough \emph{spectral gap} of the weighted
  Laplacian $\Delta_{S_1''(x)}$ on the 1-sphere $S_1(x)$ in the case
  of $S_1$-out regularity.
\end{remark}

Using Theorem \ref{thm:curvs1out}, we can show all Johnson graphs are $\infty$-curvature sharp, which is an extension of Example \ref{example:Kn}.

\begin{example}[Johnson graphs]\label{example:Johnson}
Johnson graph $J(n,k)$ is the graph with the set of $k$-element subsets of an $n$-element set as the vertex set, where two $k$-element subsets are adjacent when they have $k-1$ elements in common, see, e.g., \cite[Section 12.4.2]{BH12}. In particular, $J(n,1)$ is the complete graph $K_n$, and $J(n,2)$ is the line graph $L(K_n)$ of $K_n$. Each Johnson graph $J(n,k)$ is $\infty$-curvature sharp. Moreover, for any vertex $x$ of $J(n,k)$, the curvature function is given by
\begin{equation}\label{eq:Johnson}
\K_{J(n,k),x}(\N)=\frac{n+2}{2}-\frac{2k(n-k)}{\N}\,\,\,\,\forall\,\,\N\in(0,\infty].
\end{equation}
\end{example}
\begin{proof}
Denote the $n$-element set by $[n]:=\{1,2,\ldots,n\}$.
Let vertex $x$ be the $k$-element subset $\{i_1,i_2,\ldots,i_k\}\subset [n]$. Then the vertices in $S_1(x)$ can be listed as
$$y_{\ell,m}:=\{i_1,\ldots,\widehat{i}_{\ell},\ldots, i_k,j_m\}, \,\,\,\,\ell\in [k],\,\, m\in[n-k],$$ where $\{j_{1},\ldots, j_{n-k}\}$ is the complement of $x$ in $[n]$. By a hat, we mean to delete the corresponding element from the set. Therefore, we have $d:=d_x=k(n-k)$.

Any vertex $y_{\ell,m}\in S_1(x)$ is adjacent to the following vertices in $S_2(x)$:
$$z_{\ell,\ell',m,m'}:=\{i_1,\ldots,\widehat{i}_{\ell},\ldots,\widehat{i}_{\ell'},\ldots, i_k,j_m,j_{m'}\},\,\,\ell'\in [k]\setminus\{\ell\},\,\,m'\in [n-k]\setminus \{m\}.$$ Therefore, $J(n,k)$ is $S_1$-out regular at $x$ with $av_1^+(x)=(k-1)(n-k-1)$. This implies
$$\K_\infty^0(x):=\frac{3+d_x-av_1^+(x)}{2}=\frac{n+2}{2}.$$

Any vertex $z_{\ell,\ell',m,m'}\in S_2(x)$ is adjacent to  $y_{\ell,m},\,y_{\ell',m},\,y_{\ell,m'},\,y_{\ell',m'}\in S_1(x)$. That is, vertices in $S_2(x)$ have constant in degree $4$.

We now figure out the weighted Laplacian $\D_{S''_1(x)}$. First observe that two vertices $y_{\ell_1,m_1}, y_{\ell_2,m_2}\in S_1(x)$ are adjacent if and only if
$$\ell_1=\ell_2, \,m_1\neq m_2\,\,\,\,\text{ or }\,\,\,\,\ell_1\neq \ell_2,\, m_1=m_2.$$
That is, the subgraph $S_1(x)$ is the Cartesian product $K_k\times K_{n-k}$  of two complete graphs. When $\ell_1=\ell_2:=\ell, \,m_1\neq m_2$, $y_{\ell_1,m_1}, y_{\ell_2,m_2}$ have $k-1$ common neighbours in $S_2(x)$: $z_{\ell,\ell',m_1,m_2}, \,\ell'\in [k]\setminus \{\ell\}$. When $\ell_1\neq \ell_2,\, m_1=m_2:=m$, $y_{\ell_1,m_1}, y_{\ell_2,m_2}$ have $n-k-1$ common neighbours in $S_2(x)$: $z_{\ell_1,\ell_2,m,m'}, \,m'\in[n-k]\setminus\{m\}$.

Furthermore, any two vertices $y_{\ell_1,m_1}, y_{\ell_2,m_2}\in S_1(x)$ with $\ell_1\neq \ell_2,\, m_1\neq m_2$ has exactly one common neighbour in $S_2(x)$, that is, $z_{\ell_1,\ell_2,m_1,m_2}$. We like to mention that pairs of vertices of this kind are not adjacent in the subgraph $S_1(x)$.

We denote the weighted complete graph $K_k$ with constant edge weight $c_1$ by $(K_k,c_1)$. Let $(K_k,c_1)\times (K_{n-k},c_2)$ be the weighted Cartesian product $K_k\times K_{n-k}$, whose edge weights are naturally inherited from $(K_k,c_1)$ and $(K_{n-k},c_2)$. Then we have
\begin{align*}
\D_{S_1''(x)}&=\D_{\left(K_{d},\frac{1}{4}\right)}+\D_{\left(K_k,1+\frac{n-k-1}{4}-\frac{1}{4}\right)\times \left(K_{n-k},1+\frac{k-1}{4}-\frac{1}{4}\right)}\\
&=-\frac{1}{4}(dI_{d}-J_d)+\D_{\left(K_k,\frac{n-k+2}{4}\right)\times \left(K_{n-k},\frac{k+2}{4}\right)}.
\end{align*}
Recalling that the Laplacian eigenvalues of the Cartesian product are given by all possible sums of the Laplacian eigenvalues of the two original graphs, it is straightforward to check
$$\lambda_1(\D_{S_1''(x)})
=\min\left\{\frac{k(n-k+1)}{2},\frac{(n-k)(k+1)}{2}\right\}> \frac{k(n-k)}{2}=\frac{d}{2}.$$
Therefore, we conclude from Theorem \ref{thm:curvs1out} that $J(n,k)$ is $\infty$-curvature sharp, and the curvature function (\ref{eq:Johnson}) then follows immediately.
\end{proof}
\begin{remark}
From (\ref{eq:Johnson}), we see the graphs $J(n,k)$ and $J(n,n-k)$ share the same curvature function. In fact, they are isomorphic. The isomorphism is given by sending a $k$-element subset to its complement.
Johnson graphs have many other interesting properties. They are distance-regular, but not always strongly regular. The diameter of $J(n,k)$ is $\min\{k,n-k\}$. Johnson graphs are important in translating many combinatorial problems about sets into graph theory. We refer to \cite[Section 12.4.2]{BH12}, \cite[Section 1.6]{GR01} for more discussions.
\end{remark}

We finish this section by a discussion of upper bounds for the
cardinality of $S_2(x)$. A trivial upper bound, assuming nothing
besides the $S_1$-regularity of $x$, is given by
\begin{equation} \label{eq:trivS2card}
|S_2(x)| \le d \cdot av_1^+(x).
\end{equation}
This estimate does not take into account that different vertices of $S_1(x)$
may be connected to the same vertex in $S_2(x)$. The next result gives another
upper bound in terms of the spectral gap. This result is sometimes better than
\eqref{eq:trivS2card} and, at other times, worse.

\begin{proposition} \label{prop:S2card}
  Let $G=(V,E)$ be a locally finite simple graph and
  $x \in V$ be an $S_1$-out regular vertex of degree $d \ge 2$. Let
  $\lambda_1 \ge 0 $ be the second-smallest eigenvalue of
  $\Delta_{S_1''(x)}$. Then we have
  $$ |S_2(x)| \le \left( \sum_{j=1}^d (d_{y_j}-1) \right) - (d-1)\lambda_1. $$
  Note that $\sum_{j=1}^d (d_{y_j} -1)$ would be the cardinality of
  $|S_2(x)|$ in the case when $B_2(x)$ were a tree.
\end{proposition}

\begin{proof} Let
  $\lambda_0 = 0 \le \lambda_1 \le \cdots \le \lambda_{d-1}$ be the
  eigenvalues of $\Delta_{S_1''(x)}$. Then we have
  $$ (d-1) \lambda_1 \le \sum_{j=1}^{d-1} \lambda_j = {\rm Trace}(- \Delta_{S_1''(x)}) = \left( \sum_{j=1}^d d_{y_j}^0 \right) + \left( \sum_{j=1}^d d_{y_j}' \right). $$
  For the second sum on the right hand side we obtain
  \begin{multline*}
  \sum_{j=1}^d d_{y_j}' = \sum_{j=1}^d \left( av_1^+(x) - \sum_{z \in _2(x), z \sim y_j} \frac{1}{d_z^-} \right) = \\ = d av_1^+(x) - \sum_{z \in S_2(x)} \sum_{j: y_j \sim z} \frac{1}{d_z^-} = d av_1^+(x) - |S_2(x)|.
  \end{multline*}
  Combination with the first sum on the right hand side leads to
  \begin{equation} \label{eq:lambda1}
  (d-1)\lambda_1 \le \left( \sum_{j=1}^d (d_{y_j}-1) \right) - |S_2(x)|.
  \end{equation}
  Rearranging this inequality finishes the proof.
\end{proof}

Combining Proposition \ref{prop:S2card} with inequality \eqref{eq:nonnegcurveq}
leads directly to the following result in the case of non-negative curvature.

\begin{corollary} Let $G=(V,E)$ be a locally finite simple graph and
  $x \in V$ be an $S_1$-out regular vertex of degree $d \ge 2$ satisfying
  $\K_{G,x}(\infty) \ge 0$. Then we have $av_1^+(x) \le d + 3$ and
  $$ |S_2(x)| \le \left( \sum_{j=1}^d (d_{y_j}-1) \right) -
  (d-1) \frac{d+ av_1^+(x)-3}{4}. $$
  This simplifies in the case of a $d$-regular graph to
  $$ |S_2(x)| \le (d-1) \left( \frac{3}{4}d - \frac{av_1^+(x)-3}{4} \right). $$
\end{corollary}

Note in comparison that the $2$-sphere of a $d$-regular tree has
cardinality $(d-1)d$. Another immediate consequence is the following result.

\begin{corollary} Let $G=(V,E)$ be a locally finite simple graph and
  $x \in V$ be an $S_1$-out regular vertex of degree $d \ge 2$. If $x$
  is \emph{$\infty$-curvature sharp}, we have
  $$ |S_2(x)| \le \left( \sum_{j=1}^d (d_{y_j}-1) \right) - \frac{d(d-1)}{2}. $$
  This, together with our trivial estimate \eqref{eq:trivS2card}
  yields, in the case of a $d$-regular graph
  $$ |S_2(x)| \le \min \left\{ \frac{d(d-1)}{2}, d \cdot av_1^+(x) \right\}. $$
\end{corollary}

\begin{proof} Recall that $\infty$-curvature sharpness of $x$ is
  equivalent to $\lambda_1 \ge d/2$. Then apply Proposition \ref{prop:S2card}.
\end{proof}

\begin{remark} (a) The $\infty$-curvature sharpness condition at a
  vertex $x$ implies for $d$-regular graphs that the cardinality of
  $|S_2(x)|$ is at most half as large as the corresponding $2$-sphere
  of a $d$-regular tree. Tightness of this estimate was obtained in
  Example \ref{examples:spectralgaps}(b).

  (b) Note that in the case of a $d$-regular graph $G=(V,E)$ the
  inequalities \eqref{eq:nonnegcurveq} and \eqref{eq:lambda1} can also
  be understood as the following bounds of the spectral gap of
  $\Delta_{S_1''(x)}$ at an $S_1$-out regular vertex $x \in V$ with
  $\K_{G,x}(\infty) \ge 0$:
  $$ \frac{d}{4} + \frac{av_1^+(x)-3}{4} \le \lambda_1 \le
  d - \frac{1}{d-1} |S_2(x)|. $$
  Moreover, the upper bound for $\lambda_1$ holds generally without any
  curvature restriction.
\end{remark}

\section{Curvature of Cayley graphs}\label{section:Cayley}

It is a well known general fact that all abelian Cayley graphs or,
more generally, all Ricci-flat graphs lie in the class
$CD(0,\infty)$ \cite{CY96,LY10}. In this section, we
will consider more subtle curvature properties of abelian and
especially also non-abelian Cayley graphs.

Generally, a Cayley graph $G = {\rm Cay}(\Gamma,S)$ is determined by a
discrete group $\Gamma$ and a finite set of generators
$S \subset \Gamma$ which is symmetric (i.e., if $s \in S$ then also
$s^{-1} \in S$). The vertices of a Cayley graph are the elements of
$\Gamma$ and every directed edge can be labelled by one of the
generators in $S$.

Henceforth, $e \in \Gamma$ denotes the identity element and $S$ does
never contain $e$ (which would give rise to a loop). Note that if a
generator $s \in S$ has order $2$, i.e., if $s = s^{-1}$, this element
is contained only once in $S$ and gives rise to a single edge
emanating from any given vertex, and labelled in both directions by
$s$. Any non-empty word $r(S)$ in the generators $S$ is called a {\em
  relation} if it represents the identity element. Geometrically,
every relation corresponds one-to-one to a closed walk in the
corresponding Cayley graph starting from and returning to, say, the
identity $e \in \Gamma$. We call a relation {\em reduced}, if it does
not contain a subword of the form $s s^{-1}$. Geometrically, the
corresponding closed walk is without backtracking. Note that Cayley
graphs are $|S|$-regular and vertex transitive. Since every regular
graph without triangles is $S^1$-out regular, the same holds true for
Cayley graphs ${\rm Cay}(\Gamma,S)$ which do not have relations of
length $3$. Most of our examples of Cayley graphs will have this
property.

Since all vertices have the same curvature function, we can drop the
reference to the vertex and simply write $\K_{G}$ for $\K_{G,e}$.

We start by presenting a general family of abelian Cayley graphs
having precisely zero curvature at infinity.

\begin{theorem} \label{thm:abcay} Let $G = {\rm Cay}(\Gamma,S)$ be an
  abelian Cayley graph with $2m = |S| \ge 2$. If all reduced relations
  of $S$ have length $\ge 5$ except for the relations
  $s_1 s_2 s_1^{-1} s_2^{-1}$ expressing the commutativity, then we
  have for all $g \in \Gamma$
  $$ \K_G(\N) = \begin{cases}
    2 - \frac{4m}{\N}, & \text{if $\N \le 2m$,} \\
    0, & \text{if $\N > 2m$.} \end{cases}  $$
  In particular, we have $\K_G(\infty) = 0$.
\end{theorem}

\begin{proof}
  First recall that $G$ is $d$-regular with $d = 2m$.

  If $S$ contains an element $s \in S$ of order $k \le 4$, the word
  $s^k$ is a reduced relation of length $\le 4$. Therefore, all
  elements of $S$ have oder at least $5$.

  Next we describe the structure of $\mathring{B}_2(e)$: We have
  $S_1(e) = S$ with no edges between two different vertices of
  $S_1(e)$. A pair of vertices $s_1, s_2 \in S_1(e)$ is connected via
  a unique vertex $s_1 s_2 = s_2 s_1$ in $S_2(e)$ iff
  $s_1 \neq s_2^{-1}$. Otherwise $s_1$ and $s_2$ are not connected via
  a vertex in $S_2(e)$. Moreover, every vertex $s \in S_1(e)$ is
  adjacent to its square $s^2 \in S_2(x)$. Since $G$ is $S_1$-out
  regular (it does not have relations of length $3$), we can
  apply the results of section \ref{sec:S1outreg}. In particular, we
  have $av_1^+(e) = 2m-1$ and $d_{s_1 s_2}^- = 2$ and $d_{s^2}^- = 1$
  for $s_1, s_2, s \in S$ with $s_1 \neq s_2^{\pm 1}$. As a
  consequence, $S_1'(e) = S_1''(e)$ is the {\em cocktail party graph}
  of $2m=|S|$ vertices. The spectrum of the adjacency matrix
  $A_{S_1''(e)}$ of $S_1''(e)$ is known to be (see \cite[Section
  8.1]{BH12})
  $$ \sigma(A_{S_1''(e)}) = \{ \underbrace{-2,\dots,-2}_{m-1},
  \underbrace{0,\dots,0,}_m, 2m-2 \}. $$
  The weights of the weighted Laplacian $\Delta_{S_1''(e)}$ are all
  equal to $1/2$ and, therefore, we have
  $\lambda_1(\Delta_{S_1''(e)}) = m-1 < d/2 = m$.  Note that
  $3+d-av_1^+(x) = 4$ and Theorem \ref{thm:curvs1out}(b) tells us that
  $$ \K_G(\N) = \begin{cases}
    2 - \frac{4m}{\N}, & \text{if $\N \le 2m$,} \\
    0, & \text{if $\N > 2m$.} \end{cases}  $$
\end{proof}

\begin{remark} In fact, there exists an absolute constant $C$ such that for any abelian Cayley graph $G=Cay(\Gamma, S)$ with degree $d=|S|$ and size $N$, we have
\begin{equation}\label{eq:Lich_Friedman}
0\leq \K_{G}(\infty)\leq CdN^{-\frac{4}{d}}.
\end{equation}
In particular, when the size $N\to \infty$, $\K_G(\infty)$ tends to zero. This follows from two estimates for the eigenvalue $\lambda_1(\D)$ of the Laplacian $\D$ of $G$: A Lichnerowicz type estimate (see \cite{CLY14, KKRT16}) tells $\lambda_1(\D)\geq \K_G(\infty)$; Friedman, Murty, and Tillich \cite{FMT06} prove that there exists an absolute constant $C$ such that for any abelian Cayley graph, $\lambda_1(\D)\leq CdN^{-\frac{4}{d}}$.

\end{remark}

Now we move on to general Cayley graphs which are no longer assumed to
be abelian. All our groups $\Gamma$ are finitely presented, i.e.,
$\Gamma = \langle S \mid R \rangle$, where
$S \subset \Gamma \setminus \{e\}$ is finite and symmetric and $R$
consists of finitely many words $r_1(S), \dots, r_N(S)$, the so-called
{\em defining relations} of (the presentation of) $\Gamma$. We
conjecture that the curvature function behaves monotonically under
addition of defining relations:

\begin{conjecture} \label{conj:monrel}Let $\Gamma$ be given by
$$ \Gamma = \langle S \mid r_1(S),\dots,r_N(S) \rangle, $$
and $\Gamma'$ be another group, differing from $\Gamma$ by one
additional defining relation:
$$ \Gamma' = \langle S \mid r_1(S),\dots,r_N(S),r'(S) \rangle. $$
Assume that none of the generators in $S \subset \Gamma$ is
redundant in $\Gamma'$, i.e., for any two different generators
$s_1, s_2 \in S$ in the original group $\Gamma$ we also have $s_1 \neq s_2$
in $\Gamma'$ and $s_1 \neq e_{\Gamma'} \neq s_2$.

Then the associated Cayley graphs $G = {\rm Cay}(\Gamma,S)$ and
$G' = {\rm Cay}(\Gamma',S)$ are both $|S|$-regular and we have:
$$ \K_{G'}(\N) \ge \K_G(\N) \quad \text{for all $\N \in (0,\infty]$.} $$
\end{conjecture}

\begin{example} The Cayley graph $G$ of the free abelian group
  $\Gamma = \langle \{a^{\pm 1}\} \mid \emptyset \rangle$ is the
  $2$-regular tree $T_2$.  Theorem \ref{thm:abcay} tells us that
  $$ \K_G(\N) = \K_{T_2}(\N) = \begin{cases}
    2 - \frac{4}{\N}, & \text{if $\N \le 2$,} \\
    0, & \text{if $\N > 2$.} \end{cases} $$
  Let
  \begin{equation} \label{eq:exGamma'}
  \Gamma' = \langle \{a^{\pm 1}\} \mid a^k \rangle
  \end{equation}
  for some $k \ge 3$.  The corresponding Cayley graph agrees with the
  $k$-cycle $C_k$, and Example \ref{example:Cn} shows that
  $$ \K_{C_3}(\N) \ge \K_{C_4}(\N) \ge \K_{C_n}(\N) = \K_{T_2}(\N)$$
  for all $n \ge 5$. If we choose $k=1$ or $k=2$ in
  \eqref{eq:exGamma'}, the original generator set
  $S = \{ a^{\pm 1} \}$ becomes redundant in $\Gamma'$ since then
  $a=e$ or $a=a^{-1}$, respectively.
\end{example}

Here is another fact supporting the above conjecture.

\begin{proposition} Let $\Gamma$ and $\Gamma'$ be defined as in
  Conjecture \ref{conj:monrel}. If every reduced relation $r(S)$ of
  $\Gamma'$ which is not a relation of $\Gamma$ (i.e., does not
  represent the identity in $\Gamma$) has length at least $\ge 4$,
  then the statement of the conjecture is true. (We refer to these
  relations as {\em new relations} in $\Gamma'$. Note this property is
  required for {\em all} relations of $\Gamma'$ and not only for the
  defining ones.)
\end{proposition}

\begin{proof}
  The additional defining relation $r'(S)$ of $\Gamma'$ leads to a
  surjective group homomorphism $\Gamma \to \Gamma'$. Correspondingly,
  there is a map from the vertices of the Cayley graph $G$ to the
  vertices of the Cayley graph $G'$. The non-redundancy condition on
  the set $S$ of generators implies that this map is a bijection
  between the $1$-balls $B_1^G(e)$ and $B_1^{G'}(e)$. Generally,
  different elements in $\Gamma$ may be mapped to the same element in
  $\Gamma'$. On the Cayley graph level, this corresponds to a
  merging/identification of different vertices in $G$ to end up with
  the graph $G'$. For our curvature consideration, we are only
  concerned with identifications appearing inside the punctured
  $2$-balls of the identities. The condition on the length of new
  relations guarantees that such an identification affects only
  vertices in the $2$-spheres $S_2^G(e)$ and $S_2^{G'}(e)$ (since any
  identification of a vertex in $B_2^G(e)$ with a vertex in $B_1^G(e)$
  would correspond to a new relation of length $\le 3$ in
  $\Gamma'$). Therefore, the transition from $\mathring{B}_2^G(x)$ to
  $\mathring{B}_2^{G'}(x)$ can be described via a succession of
  mergings of pairs of vertices in the $2$-sphere. Since the directed
  edges of both Cayley graphs $G$ and $G'$ are labelled by elements of
  $S$, two different vertices can only be merged if they do not have
  common neighbours. This allows us to apply Proposition
  \ref{prop:S2merg} repeatedly and we conclude that we have
  $$ \K_{G'}(\N) \ge \K_G(\N) $$
  for any $\N \in (0,\infty]$.
\end{proof}

The Cayley graph of a free group of $k \ge 2$ generators is the
$(2k)$-regular tree with $\K_{T_{2k}}(\infty) = 2 - 2k \le 0$. In
general, Cayley graphs of nonabelian groups are expected to have
negative curvature at infinity. But this is not always the case. The
following example presents an $\infty$-curvature sharp non-abelian
Cayley graph.

\begin{example} The dihedral group of order $14$ is given by
  $$
  D_{14} = \langle \{s_1,s_2\} \mid s_1^2, s_2^2, (s_1s_2)^7 \rangle,
  $$
  representing the symmetry group of a regular $7$-gon
  $\Pi \subset \RR^2$ centered at the origin. Then $s_1$ and $s_2$
  represent reflections in the lines through the origin passing
  through the midpoint and a vertex of a side of $\Pi$, and $s_1s_2$
  represents a $2\pi/7$-rotation around the origin. Let
  $S = \{ s_1, s_2, s_3, s_4 \}$ with $s_3 = s_1 s_2 s_1$ and
  $s_4 = s_1 s_2 s_1 s_2 s_1 s_2 s_1$. The Cayley graph
  $G = {\rm Cay}(S,D_{14})$ is illustrated in Figure \ref{fig:nonabcayex}.

\begin{figure}[h]
\begin{minipage}[t]{0.42\linewidth}
    \centering
    \includegraphics[height=6.8cm]{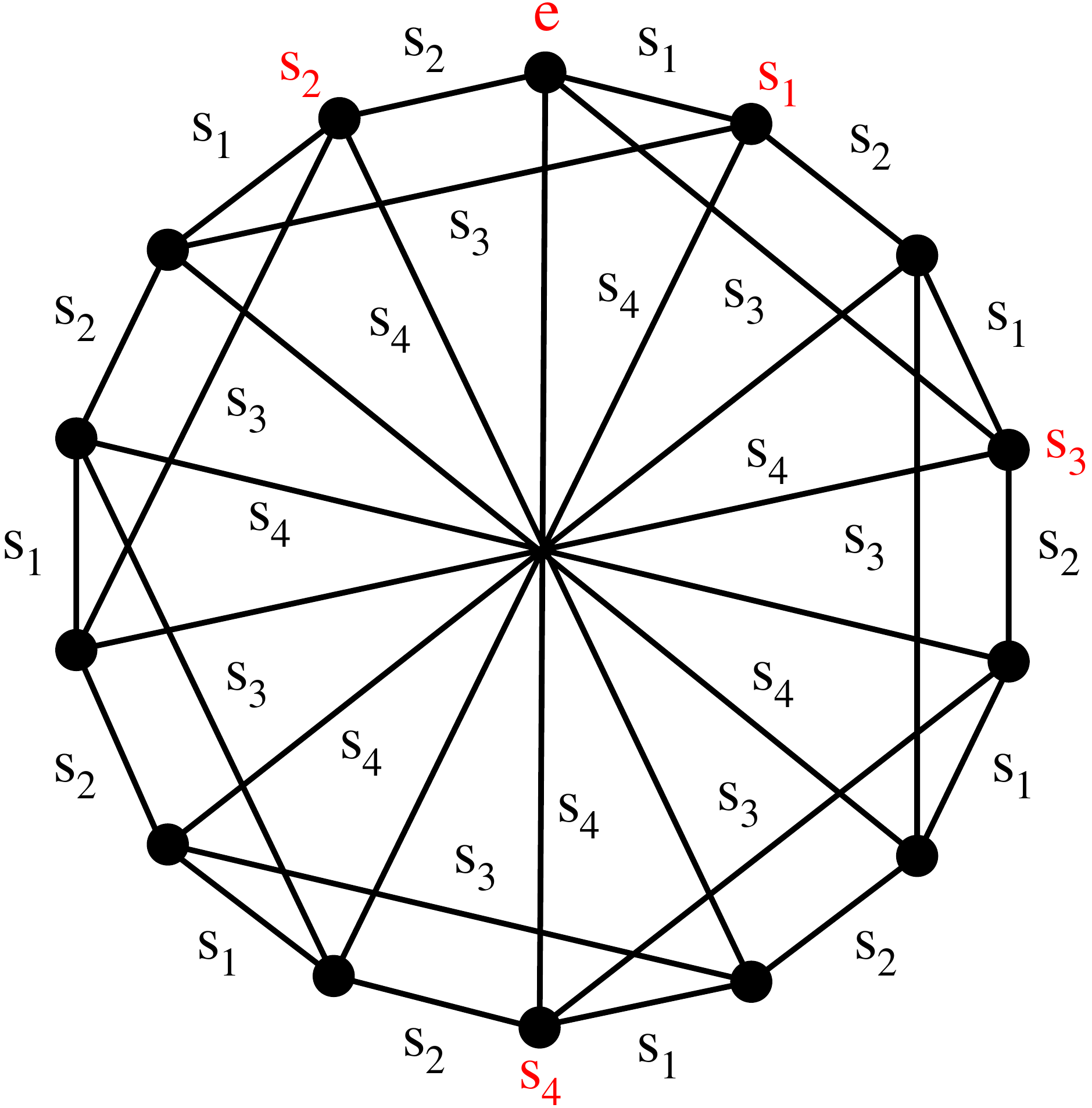}
    \caption{Cayley graph $G = {\rm Cay}(D_{14}, S)$ with edge labels and red
    vertex labels {\color{red}{$e,s_1,s_2,s_3,s_4$}}\label{fig:nonabcayex}}
\end{minipage}
\hfill
\begin{minipage}[b]{0.54\linewidth}
    \centering
    \includegraphics[width=\textwidth]{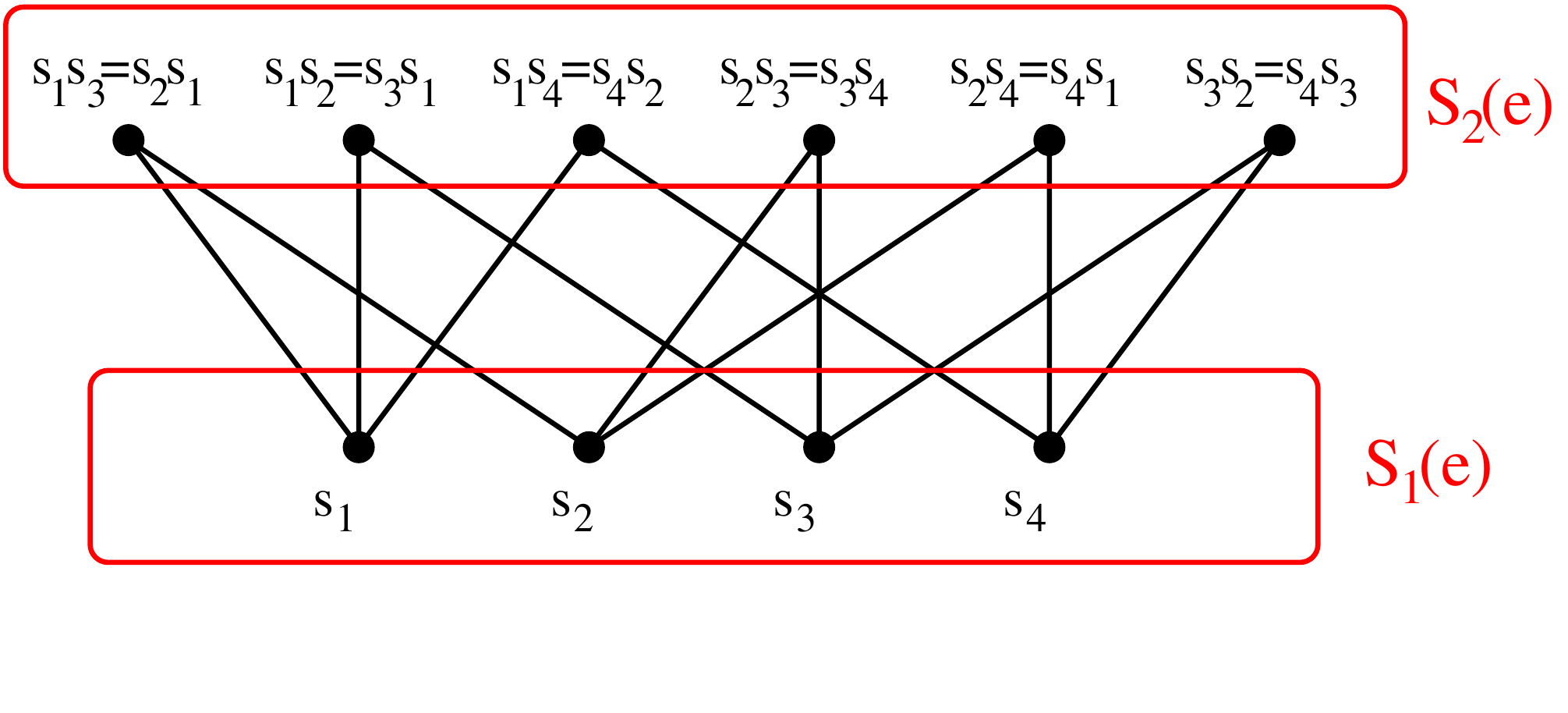}
    \caption{Illustration of $\mathring{B}_2(e)$\label{fig:nonabcayex2}}
\end{minipage}
\end{figure}

The structure of the punctured $2$-ball $\mathring{B}_2(e)$ is given
in Figure \ref{fig:nonabcayex2}. We see that $G$ is $S_1$-out regular,
$S_1(e)$ does not have any edges, $S_1'(e) = S_1''(e)$ agrees with the
complete graph $K_4$, and we have $av_1^+(e)=3$ and $d_g^-=2$ for all
$g \in S_2(e)$. As a consequence, all weights of $\Delta_{S_1''(e)}$
are equal to $1/2$ and we have
$\lambda_1(\Delta_{S_1''}) = 2 \ge 2 = d/2$. Therefore, $G$ is
$\infty$-curvature sharp by Theorem \ref{thm:curvs1out}(a), and we
have $\K_G(\N) = 2 - \frac{8}{\N}$ and $\K_g(\infty) = 2$.
\end{example}

Dihedral groups $D_{2k}$ are examples of {\em Coxeter groups}. Coxeter
groups are groups $W$ with generators $S = \{s_1,\dots,s_k\}$ and
presentations
\begin{equation} \label{eq:WCox}
W = \langle S \mid (s_is_j)^{m(i,j)} \rangle,
\end{equation}
where $m(i,i) = 1$ and $m(i,j) \ge 2$ for all $i \neq j$. It is also
possible to choose $m(i,j) = \infty$, which means that there is no
relation between the generators $s_i$ and $s_j$. The condition
$m(i,i)=1$ means that each generator $s_i$ has order $2$, and the
condition $m(i,j)=2$ means that the generators $s_i$ and $s_j$
commute. To a given Coxeter group \eqref{eq:WCox}, we associate a
Coxeter diagram (also called Dynkin diagram) as follows: It is a
finite graph with $k$ vertices (every vertex represents a generator
$s_j$) and the vertices corresponding to $s_i$ and $s_j$ are connected
by an edge if and only if $m(i,j) \ge 3$. For example, the Coxeter
diagram of the dihedral group
$$ D_{2k} = \langle \{s_1,s_2\} \mid s_1^2, s_2^2, (s_1s_2)^k \rangle $$
is the complete graph $K_2$. The Coxeter diagram carries all information
to calculate the curvature function of the Cayley graph ${\rm Cay}(W,S)$.

\begin{theorem} \label{thm:coxcurv}Let $(W,S)$ be a Coxeter group as
  defined in \eqref{eq:WCox} with $k = |S|$. Assume that the
  corresponding Coxeter diagram has at least one edge. Let
  $\mu^{max}(W,S) \in (0,k]$ be the maximal eigenvalue of the
  non-normalized Laplacian on the corresponding Coxeter diagram. Then
  the curvature function of $G = {\rm Cay}(W,S)$ is given by
  $$ \K_G(\N) = \begin{cases} 2 - \frac{2k}{\N}, & \text{if $\N \le \N_0$,} \\
    2-\mu^{max}(W,S), & \text{if $\N > \N_0$,} \end{cases} $$
  with $\N_0 = 2k/\mu^{max}(W,S)$.
\end{theorem}

\begin{proof} Let $(W,S)$ be a Coxeter group as defined in
  \eqref{eq:WCox} and $G = {\rm Cay}(W,S)$ be the associated Cayley
  graph.

  We extend first the notion of a {\em reduced relation} to be any
  relation which does neither contain subwords of the form $ss^{-1}$
  nor of the form $s^2$ for $s \in S$ (since $s^{-1}=s$). Then the
  Coxeter group $(W,S)$ does not have any non-trivial reduced relation
  of length $\le 4$ other than $(s_is_j)^{m(i,j)}$ with
  $m(i,j)=2$. This is a direct consequence of Tits' solution of the
  word problem for Coxeter groups (see \cite[Theorem
  3.4.2]{Davis}). In particular, $G$ is triangle free and $S^1$-out
  regular with $av_1^+(e) = k-1$. The induced subgraph $S_1(e)$ of $G$
  does not have any edges and its vertices correspond to the
  generators in $S$. Moreover, two different vertices $s_i,s_j \in S$
  are connected by an edge in $S_1'(e) = S_1''(e)$ if and only if
  $m(i,j) = 2$ (since the relation $(s_i s_j)^2$ means in the original
  graph $G$ that $s_i$ and $s_j$ have the unique common neighbour
  $s_i s_j = s_j s_i$ in $S_2(e)$). This implies that $S_1''(e)$ is
  the complement of the Coxeter diagram. Since the in degree of
  $s_i s_j \in S_2(e)$ in the case $m(i,j)=2$ is equal to $2$, all
  weights of the weighted Laplacian $\Delta_{S_1''(e)}$ are equal to
  $1/2$, and we have (see \cite[Section 1.3.2]{BH12} for the spectrum
  of the complement)
  $$ \lambda_1 = \lambda_1(\Delta_{S_1''(e)}) =
  \frac{1}{2}(k - \mu^{max}(W,S)).  $$
  Since $d=k$, we have $\lambda_1 <d/2$ and the statement of the theorem
  follows directly from Theorem \ref{thm:curvs1out}(b).
\end{proof}

\begin{remark} If the Coxeter diagram of $(W,S)$ has no edges, we are
  in the case of the Coxeter group
  $A_1^k = A_1 \times \dots \times A_1$ (with $k=|S|$) and we have
  $\mu^{max}(W,S) = 0$. It can be checked similarly as above that
  the Cayley graph $G={\rm Cay}(W,S)$ is then $\infty$-curvature sharp
  and
  $$ \K_G(\N) = 2 - \frac{2k}{\N}. $$
\end{remark}

\begin{example} The Coxeter diagrams of the finite Coxeter group $A_n$
  and the affine Coxeter group $\widetilde{A_n}$, $n \ge 2$, are the
  path $P_n$ and the cycle $C_{n+1}$.

  Let us now calculate the curvature functions of the corresponding
  Cayley graphs $G_n = {\rm Cay}(A_n,\{s_1,\dots,s_n\})$ and
  $G_n' = {\rm Cay}(\widetilde{A_n},\{s_1,\dots,s_{n+1}\})$. (In the
  particular case $n=2$, $G_2$ is the hexagon $C_6$ and $G_2'$ can be
  viewed as the $1$-skeleton of the regular hexagonal tiling of
  $\RR^2$.)

  The spectrum of the non-normalized Laplacian on $P_n$ is given by
  $\mu_j = 2-2\cos(\pi j/n)$ for $j=0,1,\dots,n-1$ (see \cite[Section
  1.4.4]{BH12}), and therefore
  $$ \mu^{max} = 2-2\cos(\frac{n-1}{n}\pi). $$
  Therefore, we have
  $$ \K_{G_n}(\N) = \begin{cases} 2 - \frac{2n}{\N}, &
    \text{if $\N \le \N_0$,} \\
    2\cos(\frac{n-1}{n}\pi), & \text{if $\N > \N_0$,} \end{cases} $$
  with $\N_0 = n/(1-\cos(\pi (n-1)/n))$. In particular, we have
  $$ \lim_{n \to \infty} \K_{G_n}(\infty) = -2. $$

  The spectrum of the non-normalized Laplacian on $C_{n+1}$ is
  $\mu_j = 2-2\cos(2\pi j/(n+1))$ for $j=0,1,\dots,n$ (see
  \cite[Section 1.4.3]{BH12}), which implies
  $$ \mu^{max} = 2-2\cos\left(\frac{n+\delta_{odd}(n)}{n+1} \pi\right), $$
  with $\delta_{odd}(n) \in \{0,1\}$ taking the value $1$ iff $n$ is odd.
  Therefore, we have
  $$ \K_{G_n'}(\N) = \begin{cases} 2 - \frac{2(n+1)}{\N}, &
    \text{if $\N \le \N_0$,} \\
    2-2\cos(\frac{n+\delta_{odd}(n)}{n+1}\pi), & \text{if
      $\N > \N_0$,} \end{cases} $$ with
  $\N_0 = (n+1)/(1-\cos(\pi (n+\delta_{odd}(n))/(n+1)))$. Again we have
  $$ \lim_{n \to \infty} \K_{G_n'}(\infty) = -2. $$
\end{example}

We end this section with a curvature calculation for a particular
Cayley graph acting transitively on the vertices of a Euclidean
building of type $\widetilde{A_2}$.

\begin{example} We consider the following infinite group $\Gamma$,
  given by seven generators $x_0, x_1$, $\dots, x_6$ (i.e.,
  $S = \{ x_0^{\pm 1}, \dots, x_6^{\pm 1}\}$) and seven defining
  relations
  \begin{equation*}
  \Gamma = \langle S \mid \text{$x_i x_{i+1} x_{i+3}$ for $i=0,1,\dots,6$}
  \rangle,
  \end{equation*}
  where the indices are understood to be taken mod $7$. This group
  belongs to a family of groups introduced in \cite[Section 4]{CMSZ},
  acting transitively on the vertices of a Euclidean building of type
  $\widetilde{A_2}$. So we can identify the vertices of this Euclidean
  building with the elements of $\Gamma$. Since the defining relations
  are all of length $3$, it is a priori not clear whether
  $G = {\rm Cay}(\Gamma,S)$ is $S^1$-out regular. The Euclidean building
  is a simplicial complex and all its faces are triangles. The building
  has thickness three, meaning that every edge is the boundary of precisely
  three triangles of the building.

  For the curvature calculation it is important to understand the
  structure of the punctured two ball $\mathring{B_2}(e)$. The induced
  subgraph $S_1(e)$ is $3$-regular with $14$ vertices and agrees with
  the Heawood graph $H$, as illustrated in Figure \ref{fig:heawood}. (In
  the terminology of simplicial complexes, $S_1(e)$ is the link of the
  vertex $e \in \Gamma$.)

  \begin{figure}[h]
  \centering
    \includegraphics[height=5cm]{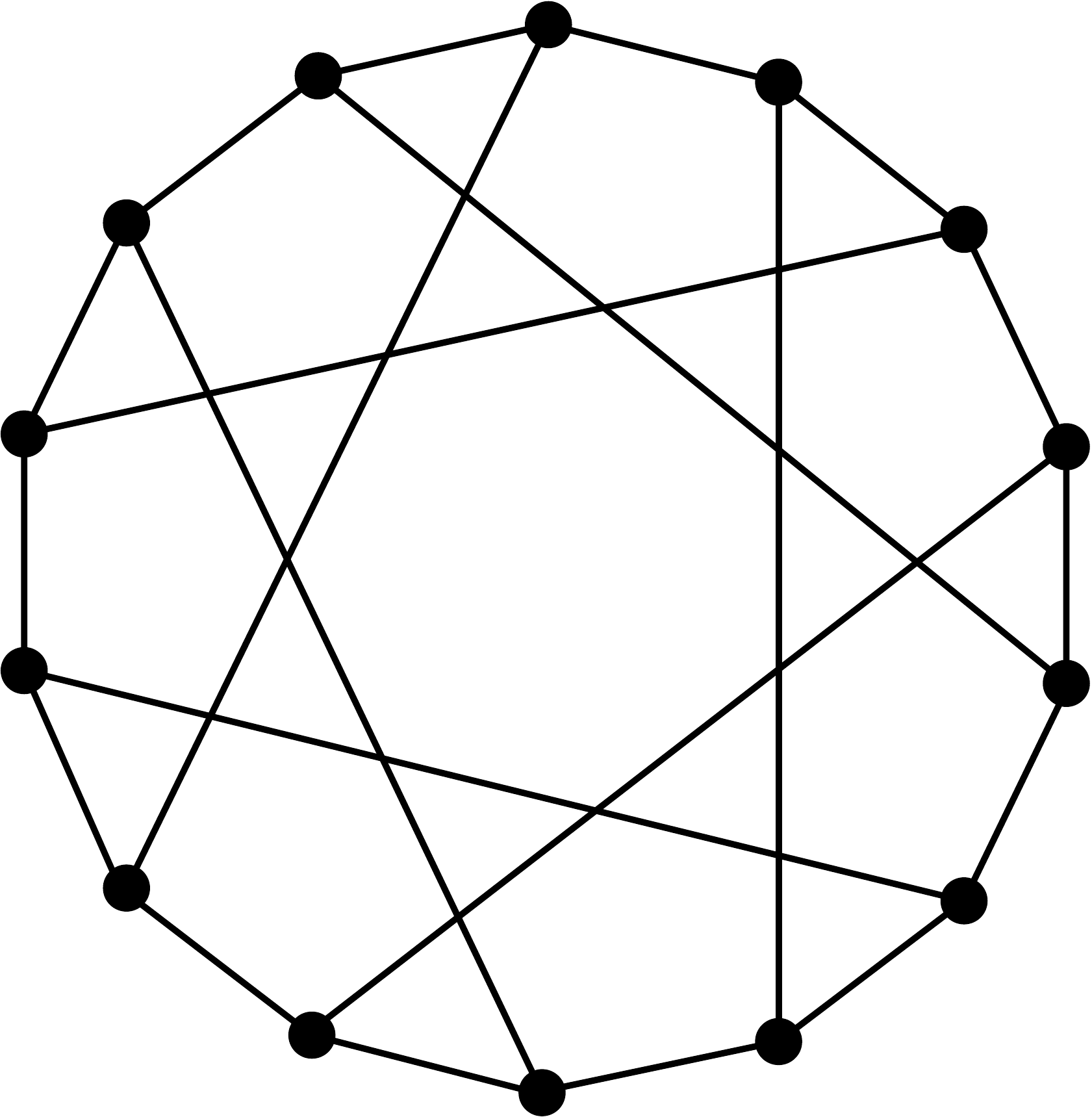}
    \caption{The Heawood graph $H$\label{fig:heawood}}
  \end{figure}

  The Heawood graph $H$ is bipartite, a generalized $3$-gon, and the
  spectrum of the non-normalized Laplacian $\Delta_H$ is given by (see
  \cite[Section 8.3]{Lub}):
  $$ \sigma(\Delta_H) = \{ 0, \underbrace{3-\sqrt{2}, \dots, 3-\sqrt{2}}_6,
  \underbrace{3+\sqrt{2}, \dots, 3+\sqrt{2}}_6, 6 \}. $$

  Since $S_1(e)$ is regular, every vertex $x \in S_1(e)$ has the same
  out degree $d_s^+=10$ and we see that $G$ is $S_1$-out regular with
  $av_1^+(e)=10$. There are two types of vertices in $S_2(e)$, the
  ones with only one ``parent'' in $S_1(e)$ and other ones with
  precisely two ``parents'' in $S_1(e)$. Recall that every edge of
  $S_1(e)$ bounds three triangles of $G$, one triangle formed by its
  end-points and the center $e$, and two other triangles formed by its
  end-points and a third vertex in $S_2(e)$ (these third vertices are
  precisely the ones with two parents). Since $S_1(e)$ has $21$ edges,
  there are precisely $42$ vertices in $S_2(e)$ with two parents in
  $S_1(e)$, and a straightforward calculation shows that there must be
  $56$ vertices in $S_2(e)$ with one parent. As a consequence, the
  graph $S_1'(e)$ is, again, the Heawood graph, and the associated
  weighted Laplacian $\Delta_{S_1'(e)}$ has all weights equal to
  $2 \cdot (1/2) =1$. Therefore, we have
  $\Delta_{S_1''(e)} = \Delta_{S_1'(e)} + \Delta_{S_1(e)} = 2
  \Delta_{S_1(e)}$ and
  $$ \lambda_1 = \lambda_1(\Delta_{S_1''(e)}) = 3-\sqrt{2}. $$
  Since $d=14$, we have $\lambda_1 < 7 = d/2$ and, applying Theorem
  \ref{thm:curvs1out}(b),
  $$ \K_{G}(\N) = \begin{cases} \frac{7}{2} - \frac{28}{\N}, &
    \text{if $\N \le \N_0$,} \\
    - \frac{9}{2} - 2 \sqrt{2}, & \text{if $\N >
      \N_0$,} \end{cases} $$
  with $\N_0 = 14/(4+\sqrt{2})$.

  It is interesting that the smallest positive eigenvalue of the
  Laplacian of the links of vertices of 2-dimensional simplicial
  complexes are also important to deduce Kazdhan property (T) for
  groups acting cocompactly in these complexes: In our case the
  requirement is connectedness of $S_1(e)$ and
  $\lambda_1(\Delta_{S_1(e)}) > 3/2$ (see \cite[Theorem 1]{BaSw}),
  which is just satisfied. It is also known that $\Gamma$ has
  infinitely many increasing finite quotients and, as a consequence,
  allows expander constructions (see, e.g., \cite{PV}). If we had
  $\K_G(\infty) \ge 0$, this example were a counterexample to our
  Conjecture \ref{conj:exp}. In this case, this example were also a
  counterexample to our Conjecture \ref{conj:Bish}, since our group
  $\Gamma$ is isomorphic to the group $G = \langle x \mid r_3 \rangle$
  in \cite[Example 3.3]{EH} and, therefore, satisfies the conditions
  $C(3)$ and $T(6)$ and must contain a free subgroup of rank two. This implies
  that our group $\Gamma$ has exponential volume growth.

\end{example}

\section{Strongly regular graphs}\label{section:srg}

A $d$-regular graph $G=(V,E)$ with $|V|=N$ is said to be \emph{strongly regular} if there are integers $\alpha$ and $\beta$ such that
every two adjacent vertices have $\alpha$ common neighbours, and
every two non-adjacent vertices have $\beta$ common neighbours.
We say a graph $G$ of this kind is \emph{a strongly regular graph} with \emph{parameters} $(N, d, \alpha, \beta)$.
Note when $\beta=0$, $G$ is isomorphic to (copies of) the complete graph $K_{d+1}$. (See, e.g., \cite[Lemma 10.1.1]{GR01}).
We assume henceforth $\beta\geq 1$.

\subsection{Curvature functions of strongly regular graphs}

We will show that the curvature function $\K_{G,x}$ of a vertex $x$ in a strongly regular graph $G$ can be determined by the graph's parameters $d, \alpha, \beta$ and the spectrum of the induced subgraph $S_1(x)$.
Note the induced subgraph $S_1(x)$ is $\alpha$-regular and of size $d$.

\begin{theorem}\label{thm:srgCur}
Let $G=(V,E)$ be a strongly regular graph with parameters $(N,d,\alpha,\beta)$. Let $x\in V$, and $A_{S_1(x)}$ be the adjacency matrix of the induced subgraph $S_1(x)$. Then, for any $\N\in (0,\infty]$, we have
\begin{equation}
\K_{G,x}(\N)=2+\frac{\alpha}{2}-\frac{2d}{\N}+\left(\frac{2d(\beta-2)-\alpha^2}{2\beta}+\frac{2d}{\N}+\frac{2}{\beta}\min_{\lambda\in \sigma\left(\at{A_{S_1(x)}}{\mathbf{1}^\perp}\right)}\left(\lambda-\frac{\alpha}{2}\right)^2\right)_-,
\end{equation}
where the minimum is taken over all eigenvalues of $A_{S_1(x)}$ corresponding to eigenvectors in $\mathbf{1}^\perp$, and $a_-:=\min\{0,a\}, \,\,\forall a\in \mathbb{R}$.
\end{theorem}

Since $\AS\mathbf{1}=\alpha \mathbf{1}$, we have $\sigma\left(\at{A_{S_1(x)}}{\mathbf{1}^\perp}\right)=\sigma\{\AS\}\setminus\{\alpha\}$ when the spectrum are considered as multisets.
The following estimate,  which is purely in terms of the parameters $d,\alpha,\beta$, is a straightforward consequence of Theorems \ref{thm:srgCur}.
\begin{corollary}\label{cor:srglb}
Let $G=(V,E)$ be a strongly regular graph with parameters $(N,d,\alpha,\beta)$. Let $x\in V$. Then, for any $\N\in (0,\infty]$, we have
\begin{equation}
2+\frac{\alpha}{2}-\frac{2d}{\N}+\left(\frac{2d(\beta-2)-\alpha^2}{2\beta}+\frac{2d}{\N}\right)_-\leq \K_\N(G,x)\leq 2+\frac{\alpha}{2}-\frac{2d}{\N}.
\end{equation}
The above lower bound estimate is sharp if and only if the adjacency matrix $\AS$ of the induced subgraph $S_1(x)$ has an eigenvalue $\alpha/2$.
\end{corollary}

\begin{remark}
The set of $\alpha$-regular graphs whose adjacency spectrum include $\alpha/2$ is an interesting class of graphs to be explored further. Note that cycles $C_{6k}$, hypercubes $Q^{4k}$, $k=1,2,\ldots$ are examples. The Cartesian product of any two graphs from this set is still in this set.
\end{remark}

The following example illustrate the necessity of the spectrum of $A_{S_1(x)}$ in Theorem \ref{thm:srgCur}. That is, the curvature functions of a strongly regular graph are not uniquely determined by its parameters.

\begin{example}[Shrikhande graph and $4\times 4$ rook's graph]\label{example:shrikhande_rook} The Shrikhand graph can be constructed as a Cayley graph of $\mathbb{Z}_4\oplus \mathbb{Z}_4$ with the generator set $\{\pm (0,1),\pm (1,0),\pm (1,1)\}$, see, e.g., \cite[Section 9.2]{BH12}. The $4\times 4$ rook's graph is isomorphic to the line graph $L(K_{4,4})$ of $K_{4,4}$, and also to $K_4\times K_4$. They are strongly regular graphs with the same parameters $(16,6,2,2)$. But the subgraph $S_1(x)$ of a vertex $x$ in the Shrikhande graph of the $4\times 4$ rook's graph are not isomorphic, as illustrated in Figures \ref{figure:Shrikhande} and \ref{figure:rook}. The one for the Shrikehand graph is $C_6$, and the one for the $4\times 4$ rook's graph is two copies of $C_3$, denoted by $2C_3$. Recall that
$$\sigma(A_{2C_3})=\{-1,-1,-1,-1,2,2\}\,\,\,\text{and }\,\,\,\sigma(A_{C_6})=\{-2,-1,-1,1,1,2\}.$$
Hence, by Theorem \ref{thm:srgCur}, the curvature function of a vertex $x$ in the Shrikhande graph is given by
$$\K_{Shrikhande,x}(\N)=\left\{
              \begin{aligned}
                &3-\frac{12}{\N}, &&\hbox{if $0<\N\leq 12$;} \\
                &2, &&\hbox{if $\N>12$,}
              \end{aligned}
            \right.$$
and the curvature function for a vertex $x$ in the $4\times 4$ rook's graph is given by
$$
\K_{L(K_{4,4}),x}(\N)=3-\frac{12}{\N}\,\,\,\forall\,\,\N\in (0,\infty].$$
\begin{figure}[h]
\begin{minipage}[t]{0.45\linewidth}
\centering
\includegraphics[width=\textwidth]{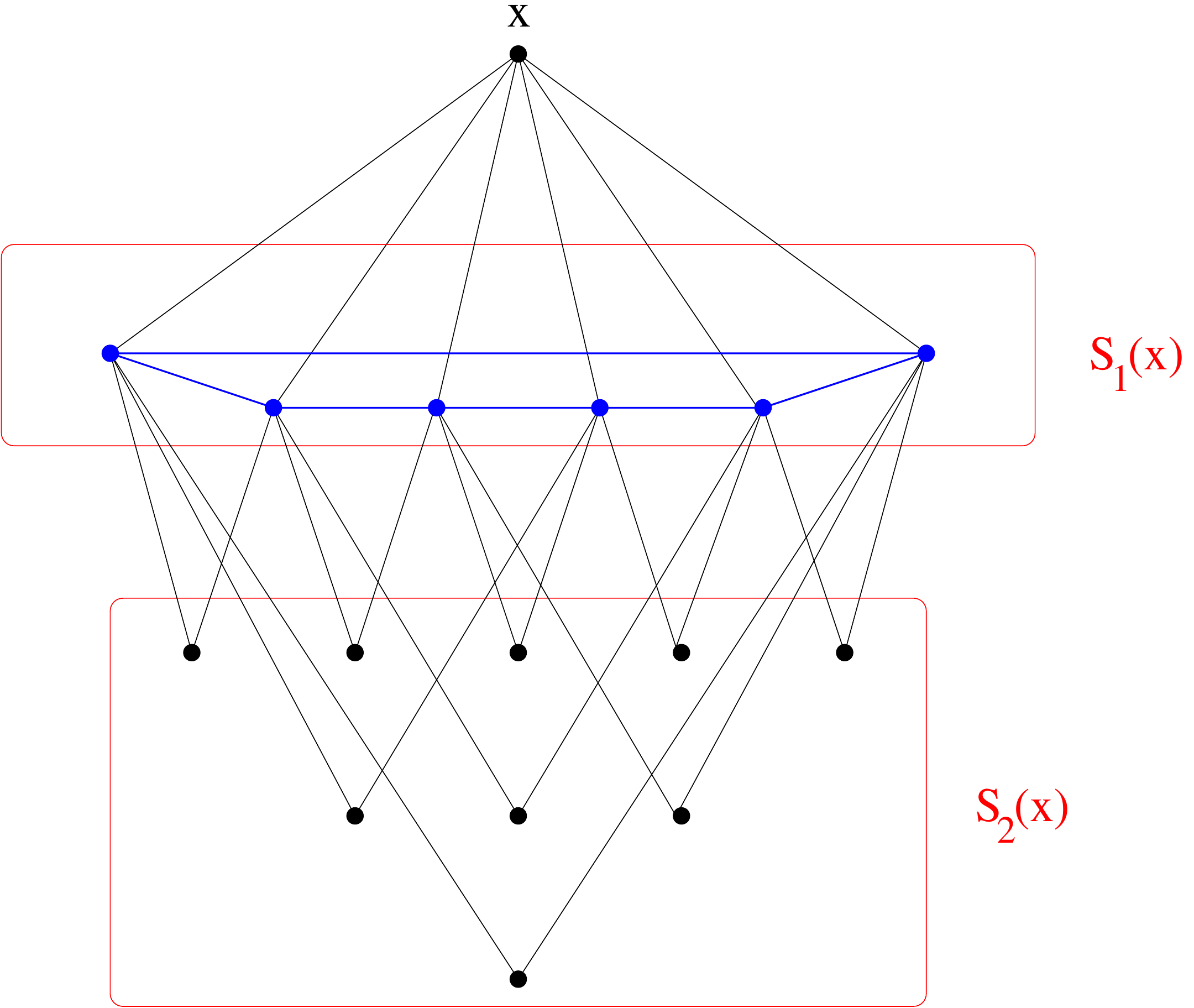}
\caption{Illustration of the subgraph $B_2(x)$ of Shrikhande graph\label{figure:Shrikhande}}
\end{minipage}
\hfill
\begin{minipage}[t]{0.45\linewidth}
\centering
\includegraphics[width=\textwidth]{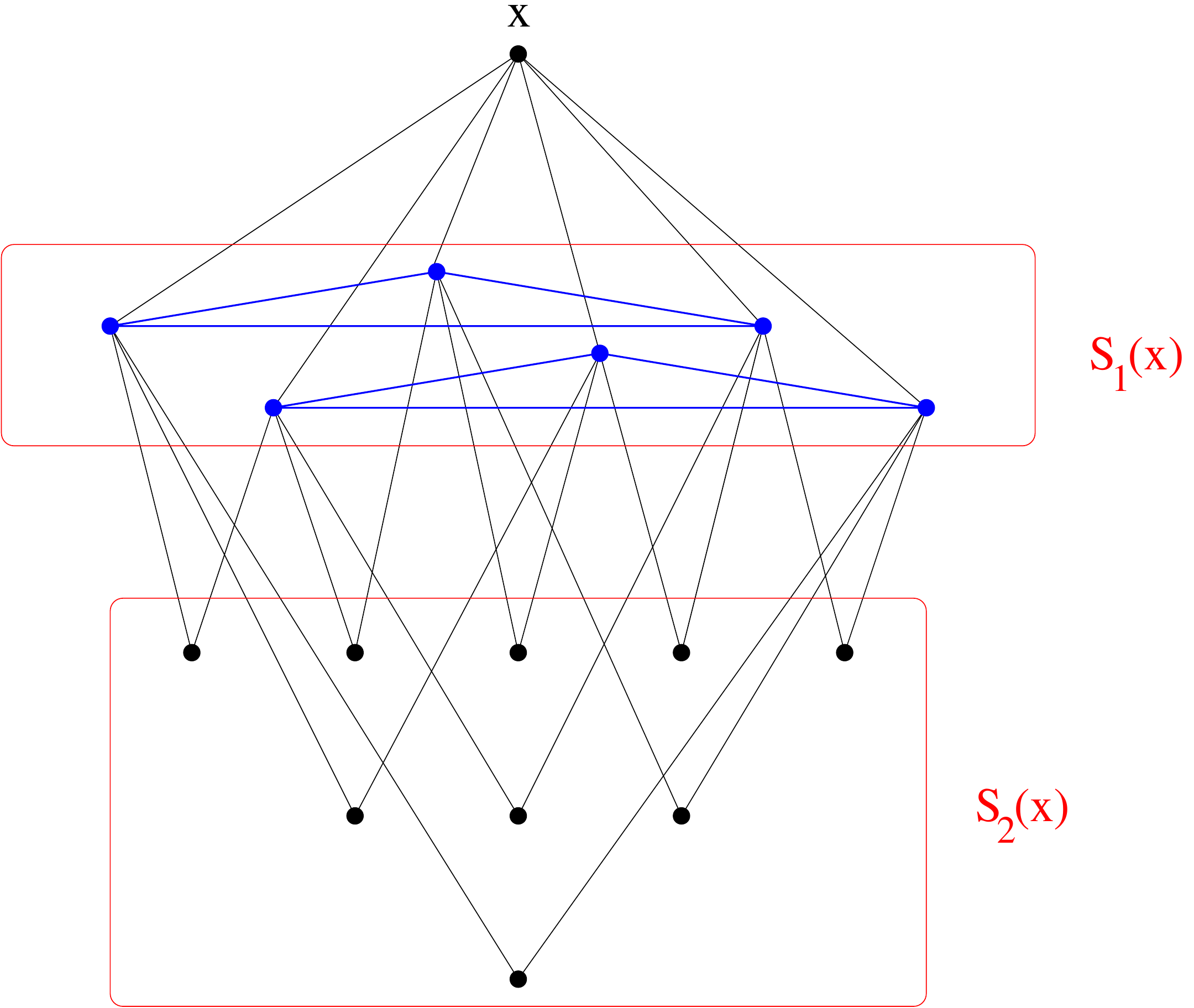}
\caption{Illustration of the subgraph $B_2(x)$ of $4\times 4$ rook's graph\label{figure:rook}}
\end{minipage}
\end{figure}
\end{example}

In the following, we prove Theorem \ref{thm:srgCur}. To that end, we first derive an expression of the matrix $\P_\infty(x)$ in terms of $A_{S_1(x)}$.

\begin{proposition}\label{prop:srgPmatrix}
Let $G=(V,E)$ be a strongly regular graph with parameters $(N,d,\alpha,\beta)$. Let $x\in V$, and $\AS$ be the adjacency matrix of the induced subgraph $S_1(x)$. Then we have
\begin{equation}\label{eq:srgPmatrix}
\frac{1}{2}\P_\infty(x)=\frac{\beta-2}{\beta}(dI_d-J_d)-\frac{2\alpha}{\beta}\AS+\frac{2}{\beta}\AS^2.
\end{equation}
\end{proposition}
\begin{proof}
We observe from the strongly regularity that for any $y\in S_1(x)$, $d_y^0=\alpha$, $d_y^+=d-\alpha-1$. For any $z\in S_2(x)$, $d_z^-=\beta$. Recall the off-diagonal entry corresponding to the pair of vertices $y_i,y_j\in S_1(x)$ is given by
\begin{equation}\label{eq:srgPoffdiag}
(\P_\infty(x))_{ij}=2-4w_{y_iy_j}-\frac{4}{\beta}\sum_{z\in S_2(x)}w_{y_iz}w_{zy_j},
\end{equation}
where $w_{y_iy_j}=(\AS)_{ij}$, and
\begin{equation}
\sum_{z\in S_2(x)}w_{y_iz}w_{zy_j}=\begin{cases}
\alpha-1-(\AS^2)_{ij}, & \text{ if }\,\,y_i\sim y_j;\\
\beta-1-(\AS^2)_{ij},  & \text{ otherwise }.
\end{cases}
\end{equation}
Note that $(\AS^2)_{ij}$ provides the number of $2$-paths connecting $y_i$ and $y_j$ via another vertex in $S_1(x)$.

Therefore we can rewrite (\ref{eq:srgPoffdiag}) as below: (For convenience, we drop the dependence on $x$ in our notations and write $A:=\AS$.):
\begin{align*}
\frac{1}{2}(\P_\infty)_{ij}&=1-2A_{ij}-\frac{2}{\beta}\left(\alpha A_{ij}+\beta(1-A_{ij})-1-(A^2)_{ij}\right)\\
&=-\frac{\beta-2}{\beta}-\frac{2\alpha}{\beta}A_{ij}+\frac{2}{\beta}(A^2)_{ij}.
\end{align*}
For the diagonal entries, we calculate by (\ref{eq:Pii}),
\begin{align*}
\frac{1}{2}(\P_\infty)_{ii}=&-(d-1)+2(d_{y_i}^++d_{y_i}^0)-\frac{2}{\beta}d_{y_i}^+\\
=&\frac{\beta-2}{\beta}(d-1)+\frac{2\alpha}{\beta}.
\end{align*}
In matrix form, we obtain
\begin{align*}
\frac{1}{2}\P_\infty=\left(\frac{\beta-2}{\beta}(d-1)+\frac{2\alpha}{\beta}\right)\cdot I_d+\frac{\beta-2}{\beta}(I_d-J_d)-\frac{2\alpha}{\beta}A+\frac{2}{\beta}(A^2-\alpha I_d)
\end{align*}
Rearranging terms, we prove (\ref{eq:srgPmatrix}).
\end{proof}

\begin{proof}[Proof of Theorem \ref{thm:srgCur}]
Since $G$ is $S_1$-out regular at $x$, we can apply Theorem \ref{thm:OutRegularFormula}. First, Proposition \ref{prop:srgPmatrix} implies
\begin{equation}
\frac{1}{2}\at{\P_\infty(x)}{\mathbf{1}^\perp}=\at{\left(\left(\frac{\beta-2}{\beta}\right)dI_d-\frac{2\alpha}{\beta}\AS+\frac{2}{\beta}\AS^2\right)}{\mathbf{1}^\perp}.
\end{equation}
This implies the spectrum is given by
\begin{equation}
\frac{1}{2}\lambda_{\min}(\at{\P_\infty(x)}{\mathbf{1}^\perp})=\min\left\{\left.\frac{d(\beta-2)}{\beta}-\frac{2\alpha}{\beta}\lambda+\frac{2}{\beta}\lambda^2\right|\lambda\in \sigma\left(\at{A_{S_1(x)}}{\mathbf{1}^\perp}\right)\right\}.
\end{equation}
Applying Theorem \ref{thm:OutRegularFormula}, we obtain
\begin{align*}
\K_{G,x}(\N)=&\K_\infty^0(x)+\min\left\{0, \frac{1}{2}\lambda_{\min}(\at{\P_\infty(x)}{\mathbf{1}^\perp})+\frac{2d}{\N}\right\}\\
=&2+\frac{\alpha}{2}-\frac{2d}{\N}+\min\left\{0, \frac{2d(\beta-2)-\alpha^2}{2\beta}-\frac{2d}{\N}+\frac{2}{\beta}\min_{\lambda\in \sigma\left(\at{A_{S_1(x)}}{\mathbf{1}^\perp}\right)}\left(\lambda-\frac{\alpha}{2}\right)^2\right\}.
\end{align*}
This completes the proof.
\end{proof}

It is well known that the girth of a strongly regular graph is determined by its parameters (see, e.g., \cite[Section 4]{CvL75}).

\begin{proposition}\label{prop:srgGirth}
Let $G=(V, E)$ be a strongly regular graph with parameters $(N, d, \alpha, \beta)$. Then the girth is $3$ if and only if $\alpha>0$, $4$ if and only if $\alpha=0$ and $\beta\geq 2$, and $5$ if and only if $\alpha=0$ and $\beta=1$.
\end{proposition}

\begin{corollary}\label{cor:srg_girth_curvature} Let $G=(V,E)$ be a strongly regular graph with vertex degree $d$. Then we have the following:
\begin{itemize}
\item[(i)] If $G$ has girth $4$, then for any $x\in V$, $$\K_{G,x}(\N)=2-\frac{2d}{\N}\,\,\,\,\forall\,\,\N\in(0,\infty].$$
\item[(ii)]
 If $G$ has girth $5$, then for any $x\in V$,
 $$\K_{G,x}(\N)=\left\{
              \begin{aligned}
                &2-\frac{2d}{\N}, &&\hbox{if $0<\N\leq 2$;} \\
                &2-d, &&\hbox{if $\N>2$.}
              \end{aligned}
            \right.$$
\end{itemize}
\end{corollary}
\begin{proof}
Suppose the parameters of $G$ are $(N,d,\alpha,\beta)$. By Proposition \ref{prop:srgGirth}, if its girth is $4$, then $\alpha=0,\beta\geq 2$. Therefore, we have
$$\frac{2d\beta-4d-\alpha^2}{2\beta}=\frac{2(\beta-2)d}{2\beta}\geq 0.$$
Applying Theorem \ref{thm:srgCur}, we obtain (i).

If the girth of $G$ is $5$, then we have $\alpha=0,\beta=1$. Therefore, we have
$$\frac{2d\beta-4d-\alpha^2}{2\beta}=-d,$$
and the induced subgraph $S_1(x)$ has empty edge set. Hence, Theorem \ref{thm:srgCur} tells (ii). Noticing that, in this case, the the local  subgraph $B_2(x)$ is isomorphic to that of a vertex in a $d$-regular tree, (ii) also follows from Example \ref{example:dtree}.
\end{proof}

In particular, we see, from Corollary \ref{cor:srg_girth_curvature}, any strongly regular graph with girth $4$ is $\infty$-curvature sharp. Such graphs include particularly Clebsch graph, Gewirtz graph, $M_{22}$ graph and Higman-Sims graph, which are strongly regular graphs with parameters $(16,5,0,2)$, $(56,10,0,2)$, $(77,16,0,4)$ and $(100,22,0,6)$, respectively.

When the girth is $3$, we believe in the following result.

\begin{conjecture}\label{conj:srg}
Any strongly regular graph with girth $3$ satisfies $CD(2,\infty)$.
\end{conjecture}
Recall Shrikhande graph is a strongly regular graph with girth $3$, whose curvature function is valued $2$ at $\infty$. So if the conjecture is true, the estimate is sharp.
In order to show Conjectre \ref{conj:srg}, it is enough, by Theorem \ref{thm:srgCur}, to show for any strongly regular graphs with parameters $(N,d,\alpha,\beta)$, $\alpha>0$,
\begin{equation}\label{eq:srg_CD2_infty}
\min_{\lambda\in \sigma\left(\at{A_{S_1(x)}}{\mathbf{1}^\perp}\right)}\left(\lambda-\frac{\alpha}{2}\right)^2\geq \frac{\alpha^2}{4}+d-\frac{\beta d}{2} -\frac{\alpha\beta}{4}.
\end{equation}
When the RHS of (\ref{eq:srg_CD2_infty}) is nonpositive, we can conclude that the graph satisfies $CD(2,\infty)$ without knowing the spectrum of $\AS$. But there are examples with the RHS of (\ref{eq:srg_CD2_infty}) positive, e.g., the strongly regular graphs with parameters $(n^2,2(n-1),n-2,2)$ with $n>4$ (this is actually the graph $L(K_{n,n})$, see Theorem \ref{thm:Chang_Shrikhande} below).

\subsection{Strongly regular graphs which are $\infty$-curvature sharp}
From Theorem \ref{thm:srgCur}, we obtain the following result immediately.
\begin{corollary}\label{cor:srg_infty_cur_sharp}
Let $G=(V,E)$ be a strongly regular graph with parameters $(N,d,\alpha,\beta)$. Then $G$ is $\infty$-curvature sharp if and only if
\begin{equation}\label{eq:srg_infty_cur_sharp}
\min_{\lambda\in \sigma\left(\at{A_{S_1(x)}}{\mathbf{1}^\perp}\right)}\left(\lambda-\frac{\alpha}{2}\right)^2\geq \frac{\alpha^2}{4}+d-\frac{\beta d}{2}.
\end{equation}
\end{corollary}

In this subsection, we discuss examples of strongly regular graphs with girth $3$ which are $\infty$-curvature sharp.
Let us start with the following two interesting families.  The following result can be found in, e.g., \cite[Theorems 8.6, 8.7]{Harary69}).
\begin{theorem}[\cite{Harary69}]\label{thm:Chang_Shrikhande} We have the following characterization of line graphs of $K_n$ and $K_{n,n}$:
\begin{itemize}
\item [(i)] The line graph $L(K_n)$, $n\geq 3$, is the unique strongly regular graph with parameters $(\binom{n}{2},2(n-2),n-2,4)$, except for $n=8$. When $n=8$, there are three more graphs besides $L(K_8)$, which are called Chang graphs, named after Li-Chien Chang.
\item [(ii)] The line graph $L(K_{n,n})$, $n\geq 2$, is the unique strongly regular graph with parameters $(n^2,2(n-1),n-2,2)$, except for $n=4$. When $n=4$, there is one more graph besides $L(K_{4,4})$, which is called Shrikhande graph, named after Sharadchandra Shankar Shrikhande (recall Example \ref{example:shrikhande_rook}).
\end{itemize}
\end{theorem}


\begin{example}\label{example:line_graph_sharp}
All strongly regular graphs with parameters $(n^2,2(n-1),n-2,2)$, $n\geq 2$ except for the Shrikhande graph and all strongly regular graphs with parameters $(\binom{n}{2},2(n-2),n-2,4)$, $n\geq 3$, i.e., $L(K_{n,n})$, $L(K_n)$ and the three Chang graphs, are $\infty$-curvature sharp.
\end{example}
\begin{proof}
The $\infty$-curvature sharpness of $L(K_{n,n})$ and $L(K_n)=J(n,2)$ follows from Examples \ref{example:LKmn} and \ref{example:Johnson} where the curvature functions of $L(K_{m,n})$ and $J(n,k)$ are derived. We like to mention those properties can also be shown by calculating the spectrum of $\AS$ and applying Corollary \ref{cor:srg_infty_cur_sharp}. Recall from Example \ref{example:shrikhande_rook}, the Shrikhande graph is only $12$-curvature sharp.

For the strongly regular graph with parameters $(\binom{n}{2},2(n-2),n-2,4)$, $n\geq 3$, we check that the RHS of (\ref{eq:srg_infty_cur_sharp}) satisfies
$$\frac{\alpha^2}{4}+d-\frac{\beta d}{2}=\frac{1}{4}(n-2)(n-10),$$
which is nonpositive when $3\leq n\leq 10$.
This implies, in particular, the three Chang graphs are $\infty$-curvature sharp.
%
\end{proof}

\begin{example}[Complete $k$-partite graphs]\label{example:complete_multipartite}
Complete $k$-partite graphs $K_{n,\ldots,n}$, $k\geq 2, n\geq 1$ are $\infty$-curvature sharp.
\end{example}
\begin{proof}
When $n=1$, this follows from Example \ref{example:Kn}. When $k\geq 2, n\geq 2$, this follows from Corollary \ref{cor:srg_infty_cur_sharp} since $K_{n,\ldots,n}$ is strongly regular with parameters $(nk,(k-1)n,(k-2)n,(k-1)n)$ and
$$\frac{\alpha^2}{4}+d-\frac{\beta d}{2}=\frac{n}{4}2k(2-k)\leq 0.$$
\end{proof}

\begin{example}[Paley graphs]\label{example:Paley_graph}
Let $q\equiv 1 (mod 4)$ be a prime power. The Paley graph $\mathrm{Paley}(q)$ is the graph with the finite field $\mathbb{F}_q$ as the vertex set, and the pairs of vertices that differ by a (nonzero) square as the edge set, see, e.g., \cite[Section 9.1]{BH12}. $\mathrm{Paley}(q)$ is a strongly regular graph with parameters $(q, (q-1)/2, (q-5)/4, (q-1)/4)$.
All Paley graphs $\mathrm{Paley}(q)$ with $q\geq 9$ are $\infty$-curvature sharp.
\end{example}
\begin{proof}
Note $\mathrm{Paley}(5)=C_5$, whose curvature function are given in Example \ref{example:Cn}, and $\mathrm{Paley}(9)=K_3\times K_3=L(K_{3,3})$, which is $\infty$-curvature sharp by Example \ref{example:line_graph_sharp}. The $\infty$-curvature sharpness of $\mathrm{Paley}(q)$, $q=13,17,\ldots$ follows from Corollary \ref{cor:srg_infty_cur_sharp} since
$$\frac{\alpha^2}{4}+d-\frac{\beta d}{2}=-\frac{1}{64}(3q^2-30q+11)$$
is nonpositive when $q>10$.
\end{proof}

Paley graphs have interesting properties. For example, they are Ramanujan graphs, and they are isomorphic to their complements.

The following example can be checked via Corollary \ref{cor:srg_infty_cur_sharp} without knowing the spectrum of $\AS$.

\begin{example}[Clebsch graph, Schl\"afli graph and their complements]
The Clebsch graph, i.e., the unique strongly regular graph with parameters $(16,5,0,2)$, and its complement, i.e., the unique strongly regular graphs with parameters $(16,10,6,6)$, are both $\infty$-curvature sharp. The Schl\"afli graph, i.e., the unique strongly regular graph with parameters $(27,16,10,8)$, and its complement, i.e., the unique strongly regular graph with parameters $(27,10,1,5)$, are also both $\infty$-curvature sharp.
\end{example}

The following is a very interesting and challenging problem.
\begin{problem}
Classify strongly regular graphs which are $\infty$-curvature sharp.
\end{problem}

By Corollary \ref{cor:srg_infty_cur_sharp}, it is equivalent to classify the strongly regular graphs for which the spectrum of the adjacency matrices of the local subgraphs $S_1(x)$ satisfy (\ref{eq:srg_infty_cur_sharp}).
It is interesting to mention Seidel's \cite{Seidel68} classification of strongly regular graphs with least eigenvalue $-2$ (of the adjacency matrices of the whole graphs): $L(K_n)$, the three Chang graphs, $L(K_{n,n})$, the Shrikhande graph, the complete $k$-partite graph $K_{2,\ldots, 2}$ (i.e., the cocktail party graphs), the Petersen graph, the complement of Clebsch graph, and the Schl\"afli graph.





\section{Curvature functions of graphs with general measures}\label{genmeassec}
Many of the computing methods for the Bakry-\'Emery curvature functions discussed above are extendible to the following general setting. Let $G=(V,E)$ be a locally finite simple graph. We can assign a symmetric nonnegative edge weights $w:E\to [0,\infty)$ and a positive vertex measure $\mu: V\to (0,\infty)$. For $\{x,y\}\in E$, we write $w_{xy}=w_{yx}$. (Recall we used a zero/one valued edge weights in (\ref{eq:0_1_edge_weight}).) We consider the curvature functions corresponding to the following Laplacian:
\begin{equation}
\D_{\mu,w} f(x):=\frac{1}{\mu(x)}\sum_{y,y\sim x}w_{xy}(f(y)-f(x)).
\end{equation}
Denote $d_{x}:= \sum_{y,y\sim x}w_{xy}$, and let
$S_1(x) = \{y_1,\dots,y_k\}$. By definition it is straightforward to
find the following matrices. We have
$$ \Delta_{\mu,w}(x)= \frac{1}{\mu(x)}\begin{pmatrix}
-d_x & w_{xy_1} & \cdots & w_{xy_k}
\end{pmatrix},
$$
and
$$
2\Gamma_{\mu,w}(x)=\frac{1}{\mu(x)}\begin{pmatrix}
d_x & -w_{xy_1} & \cdots & -w_{xy_k}\\
-w_{xy_1} & w_{xy_1} & \cdots & 0\\
\vdots & \vdots &\ddots & \vdots\\
-w_{xy_k} & 0 & \cdots & w_{xy_k}
\end{pmatrix}.
$$
The matrix $4\Gamma_2^{full}(x)$ is given entry-wise as follows:
$$(4\Gamma_2^{full}(x))_{x,x}=\frac{d_{x}^{2}}{\mu(x)^{2}}+\frac{3}{\mu(x)}\sum_{y\in S_1(x)}\frac{w_{xy}^{2}}{\mu(y)},$$
for any $y\in S_1(x)$,
\begin{align*}
(4\Gamma_2^{full}(x))_{x,y}= & -\frac{3w_{xy}^{2}}{\mu(x)\mu(y)}+\frac{w_{xy}}{\mu(x)\mu(y)}\left(\sum_{z\in S_{2}(x), z\sim y}w_{yz}\right)+\frac{d_{x}w_{xy}}{\mu(x)^{2}}
\\
& -\frac{1}{\mu(x)}\sum_{y'\in S_{1}(x), y'\sim y}\left(\frac{w_{xy}w_{yy'}}{\mu(y)}-\frac{w_{xy'}w_{yy'}}{\mu(y')}\right),
\end{align*}
\begin{align*}
(4\Gamma_2^{full}(x))_{y,y}= & \frac{3w_{xy}^{2}}{\mu(x)\mu(y)}+\frac{2w_{xy}^{2}}{\mu(x)^{2}} -\frac{d_{x}w_{xy}}{\mu(x)}+\frac{3w_{xy}}{\mu(x)\mu(y)}\sum_{z\in S_{2}(x), z\sim y}w_{yz}
\\
& +\sum_{y'\in S_1(x), y'\sim y}\left(\frac{w_{xy'}w_{yy'}}{\mu(y')}+3\frac{w_{xy}w_{yy'}}{\mu(y)}\right),
\end{align*}
for any $y_i,y_j\in S_1(x)$, $y_i\neq y_j$,
$$(4\Gamma_2^{full}(x))_{y_i,y_j}=2 \frac{w_{xy_{i}}w_{xy_{j}}}{\mu(x)^{2}}-\frac{1}{\mu(x)}\left(\frac{2w_{xy_{i}}w_{y_{i}y_{j}}}{\mu(y_{i})}+\frac{2w_{xy_{j}}w_{y_{i}y_{j}}}{\mu(y_{j})}\right),$$
and, for any $z\in S_2(x)$,
$$(4\Gamma_2^{full}(x))_{x,z}=(4\Gamma_2^{full}(x))_{z,z}=\frac{1}{\mu(x)}\sum_{y\in S_1(x), y\sim z}\frac{w_{xy}w_{yz}}{\mu(y)}, \,\,(4\Gamma_2^{full}(x))_{y,z}=-\frac{2w_{xy}w_{yz}}{\mu(x)\mu(y)},$$
for any $z_1,z_2\in S_2(x)$, $z_1\neq z_2$, $(4\Gamma_2^{full}(x))_{z_1,z_2}=0$.

When $w\equiv 1$ and $\mu\equiv 1$, it reduces to the non-normalized Laplacian (\ref{eq:nonnormalised_Laplacian}).

Now we discuss briefly the curvature functions $\K_{G,x}^{nor}(\cdot)$ corresponding to the normalized Laplacian, i.e., the case $w\equiv 1$ and $\mu(x)=d_x\,\,\forall x\in V$. This provides another interesting special case. Finally, we 
discuss some further analogous fundamental curvature results in the normalized case.
 
We have
$$\D^{nor}(x)=\frac{1}{d_x}\begin{pmatrix}
-d_x & 1 & \cdots & 1
\end{pmatrix},
$$
and
$$
2\Gamma^{nor}(x)=\frac{1}{d_x}\begin{pmatrix}
d_x & -1 & \cdots & -1\\
-1 & 1 & \cdots & 0\\
\vdots & \vdots &\ddots & \vdots\\
-1 & 0 & \cdots & 1
\end{pmatrix}.
$$
By using the full formula above the matrix $4\Gamma_2^{nor}(x)$ is given entry-wise as follows:
$$(4\Gamma_2^{nor}(x))_{x,x}=1+\frac{3}{d_x}\sum_{y\in S_1(x)}\frac{1}{d_y},$$
for any $y\in S_1(x)$,
$$(4\Gamma_2^{nor}(x))_{x,y}=-\frac{3+d_y+d_y^+}{d_xd_y}-\frac{1}{d_x}\sum_{y'\in S_1(x),y'\sim y}\left(\frac{1}{d_y}-\frac{1}{d_{y'}}\right),$$
$$(4\Gamma_2^{nor}(x))_{y,y}=\frac{2}{d_x^2}+\frac{3-d_y+3d_y^+}{d_xd_y}+\frac{1}{d_x}\sum_{y'\in S_1(x), y'\sim y}\left(\frac{1}{d_{y'}}+\frac{3}{d_y}\right),$$
for any $y_i,y_j\in S_1(x)$, $y_i\neq y_j$,
$$(4\Gamma_2^{nor}(x))_{y_i,y_j}=\frac{2}{d_x^2}-\frac{w_{y_iy_j}}{d_x}\left(\frac{2}{d_{y_1}}+\frac{2}{d_{y_2}}\right),$$
and, for any $z\in S_2(x)$,
$$(4\Gamma_2^{nor}(x))_{x,z}=(4\Gamma_2^{nor}(x))_{z,z}=\frac{1}{d_x}\sum_{y\in S_1(x), y\sim z}\frac{1}{d_y}, \,\,(4\Gamma_2^{nor}(x))_{y,z}=-\frac{2w_{yz}}{d_xd_y},$$
for any $z_1,z_2\in S_2(x)$, $z_1\neq z_2$, $(4\Gamma_2^{nor}(x))_{z_1,z_2}=0$.

We have the following upper bounds.
\begin{theorem}\label{nub}
Let $G=(V,E)$ be a locally finite simple graph and let $x\in V(G).$ For $\N\in(0,\infty]$, we have
$$\K^{nor}_{G,x}(\infty)\leq \frac{1}{2d_{x}} \sum_{y\in S_{1}(x)} \frac{4+\#_{\Delta}(x,y)}{d_{y}}-\frac{2}{\N}=\frac{1}{2d_x}\sum_{y\in S_1(x)}\frac{3+d_y-d_y^+}{d_y}-\frac{2}{\N}.$$
\end{theorem}
\begin{proof}
This can be shown by applying Sylvester's criterion to the submatrix of $$\Gamma_2^{nor}(x)-\frac{1}{\N}(\D^{nor}(x))^\top\D^{nor}(x)-\K_{G,x}^{nor}(\N)\Gamma(x)$$ corresponding to the vertices $\{x\}\sqcup S_2(x)$.
\end{proof}

Corresponding to Theorem \ref{thm:Sharpness}, we have the following result for normalized Laplacian.
\begin{theorem}\label{thm:Sharpness_nor}
Let $G=(V,E)$ be a locally finite simple graph and let $x\in V$. Then for any $\N\in~(0,\infty]$, $\K^{nor}_{G,x}(\N)$ is the solution of the following semidefintie programming,
\begin{align*}
 &\text{maximize}\,\,\, \frac{1}{2d_x}\sum_{y\in S_1(x)}\frac{3+d_y-d_y^+}{d_y}-\frac{2}{\N}-\frac{\lambda}{2}\\
&\text{subject to}\,\,\,\widehat{\P}^{nor}_\N(x)\geq -\lambda\cdot 2\Gamma^{nor}(x),
\end{align*}
where
\begin{equation} \label{eq:PN_nor}
\widehat{\P}^{nor}_{\N}(x):=\widehat{\P}^{nor}_{\infty}(x)+\frac{4}{d_x^2\N}\left(\begin{array}{cccc}
0 & 0 & \cdots & 0\\
0 & \multicolumn{3}{c}{\multirow{3}{*}{\raisebox{1mm}{\scalebox{1}{$d_xI_{d_x}-J_{d_x}$}}}}  \\
\vdots & \\
0 &
\end{array}
\right),
\end{equation}
and (recall (\ref{eq:QM}))
$$\widehat{\P}^{nor}_\infty(x):=4\Q\left(\Gamma_2^{nor}(x)-\frac{1}{2d_x}\sum_{y\in S_1(x)}\frac{3+d_y-d_y^+}{d_y}\Gamma^{nor}(x)\right)$$
Moreover, the following are equivalent:
\begin{itemize}
\item[(i)]
$\K_{G,x}^{nor}(\N)=\frac{1}{2d_x}\sum_{y\in S_1(x)}\frac{3+d_y-d_y^+}{d_y}-\frac{2}{\N}-\frac{\lambda}{2}$;
\item[(ii)]
The matrix $\widehat{\P}^{nor}_\N(x)+\lambda\cdot 2\Gamma^{nor}(x)$ is positive semidefinite and has zero eigenvalue of multiplicity at least $2$.
\end{itemize}
\end{theorem}

Theorem \ref{thm:Sharpness_nor} can be proved in a similar way as the proof of Theorem \ref{thm:Sharpness}, based on Proposition \ref{prop:MatrixVersion1}. Below we give an explicit description of $\widehat{\P}^{nor}_\infty(x)$. It is a $|B_1(x)|$ by $|B_1(x)|$ matrix satisfying $\widehat{\P}^{nor}_\infty(x)\mathbf{1}=0$. We have
\begin{equation}
\widehat{\P}^{nor}_\infty(x)=\left(\begin{array}{cccc}
0 & p_1(x) & \cdots & p_{d_x}(x)\\
p_1(x) & \multicolumn{3}{c}{\multirow{3}{*}{\raisebox{-1mm}{\scalebox{1}{$\P^{nor}_\infty(x)$}}}}  \\
\vdots & \\
p_{d_x}(x) &
\end{array}
\right),
\end{equation}
where
where for $i\in [d_x]$
\begin{align*}
p_i(x):=\frac{1}{d_x}\left(\frac{d_{y_i}^+-3}{d_{y_i}}-\frac{1}{d_x}\sum_{y\in S_1(x)}\frac{d_y^+-3}{d_y}\right)-\frac{1}{d_x}\sum_{y'\in S_1(x),y'\sim y_i}\left(\frac{1}{d_{y_i}}-\frac{1}{d_{y'}}\right).
\end{align*}
Mover, for or any $y_i,y_j\in S_1(x)$, $y_i\neq y_j$,
$$(\P^{nor}_\infty(x))_{y_i,y_j}=\frac{2}{d_x^2}-\frac{w_{y_iy_j}}{d_x}\left(\frac{2}{d_{y_i}}+\frac{2}{d_{y_j}}\right)-\frac{4}{d_xd_{y_i}d_{y_j}}\sum_{z\in S_2(x)}\frac{w_{y_iz}w_{zy_j}}{\sum_{y\in S_1(x),y\sim z}\frac{1}{d_y}}.$$
Similarly as in Theorems \ref{thm:lb} and \ref{thm:OutRegularFormula}, we have the following results.
\begin{theorem}\label{thm:lb_nor}
Let $G=(V,E)$ be a locally finite simple graph and let $x\in V$, Then for any $\N\in (0,\infty]$, we have
\begin{equation}\label{eq:lb_nor}
\K^{nor}_{G,x}(\N)\geq \frac{1}{2d_x}\sum_{y\in S_1(x)}\frac{3+d_y-d_y^+}{d_y}-\frac{2}{\N}+\frac{1}{2}\lambda_{\min}(d_x\widehat{\P}^{nor}_\N(x)).
\end{equation}
\end{theorem}

\begin{theorem}\label{thm:OutRegularFormulaNor}
Let $G=(V,E)$ be a locally finite simple graph and let $x\in V$.
Assume that $G$ satisfies the following regularity
\begin{equation}\label{eq:outRegularityNor}
p_i(x)=0, \,\,\text{for any}\,\, y_i\in S_1(x).
\end{equation} Then, we have for any $\N\in (0,\infty]$
\begin{equation}\label{eq:OutRegularFormulaNor}
\K^{nor}_{G,x}(\N)=\frac{1}{2d_x}\sum_{y\in S_1(x)}\frac{3+d_y-d_y^+}{d_y}-\frac{2}{\N}+\frac{1}{2}\lambda_{\min}\left(d_x\P^{nor}_\infty(x)+\frac{4}{d_x\N}(d_xI_{d_x}-J_{d_x})\right).
\end{equation}
\end{theorem}
In particular, if $d_{y_i}=av_1(x), \,\,d_{y_i}^+(x)=av_1^+(x)$ for any $y_i\in S_1(x)$, we have (\ref{eq:outRegularityNor}). If, furthermore, the graph $G$ is $d$-regular, all the above results are direct consequences of their counterparts for non-normalized Laplacians, since $\K_{G,x}(\N)=d\K^{nor}_{G,x}(\N)$.

\begin{example}[Complete bipartite graphs revisited]\label{example:Kmn_nor}
Let $K_{m,n}=(V,E),\,m,n\geq 1$ be a complete bipartite graph. Let $x\in V$ be a vertex with degree $d_x=n$. Then we have for any $\N\in(0,\infty]$
\begin{equation}\label{eq:Kmn_nor}
\K^{nor}_{K_{m,n},x}(\N)=\frac{2}{m}-\frac{2}{\N}.
\end{equation}
\end{example}
\begin{proof}
By definition, we have
$$\P^{nor}_\infty(x)=\frac{n-1}{n^2}\left(4\frac{m-1}{m}-2\right)(nI_n-J_n).$$
Therefore, when $n\geq 1, m\geq 2$, we have $\lambda_{\min}(\P^{nor}_\infty(x))\geq 0$. By Theorem \ref{thm:OutRegularFormulaNor}, we conclude (\ref{eq:Kmn_nor}) in this case.
Otherwise, when $m=n=1$, (\ref{eq:Kmn_nor}) follows directly from Example \ref{example:Kmn}.
\end{proof}
\begin{remark}
From Example \ref{example:Kmn_nor}, we see every vertex in $K_{m,n}$ is $\infty$-curvature sharp with respect to the normalized Laplacian. Recall from Example \ref{example:Kmn}, this is not always the case with respect to the non-normalized Laplacian.

Note that other curvature notions based on normalized Laplacian and optimal transportation have also been calculated on complete bipartite graphs. The Ollivier-Ricci curvature on each edge of $K_{m,n}$ is $0$ (\cite[Corollary 3.2]{BM15}); Lin-Lu-Yau's modified Ollivier-Ricci curvature \cite{LLY11} on each edge of $K_{m,n}$ is $2/\max\{m,n\}$ \cite[Section 5]{Smith14}. The latter one coincides with $\K^{nor}_{K_{m,n},\cdot}(\infty)$ when $m=n$. For the comparison of Ollivier-Ricci curvature and Bakry-\'Emery curvature functions on a complete graph, we refer to \cite[Section 4]{JL14}.
\end{remark}

Finally,
we present an example for which $\K^{nor}_{G,x}(\infty)$ and $\K_{G,x}(\infty)$ has different sign at every vertex $x$. The Figures \ref{figure:cur_nor} and \ref{figure:cur_non} are taken from the web-application for calculation of curvature on graphs by David Cushing and  George W. Stagg. The number of each vertex is the corresponding curvature values.
\begin{figure}[h]
\begin{minipage}[t]{0.48\linewidth}
\centering
\includegraphics[width=\textwidth]{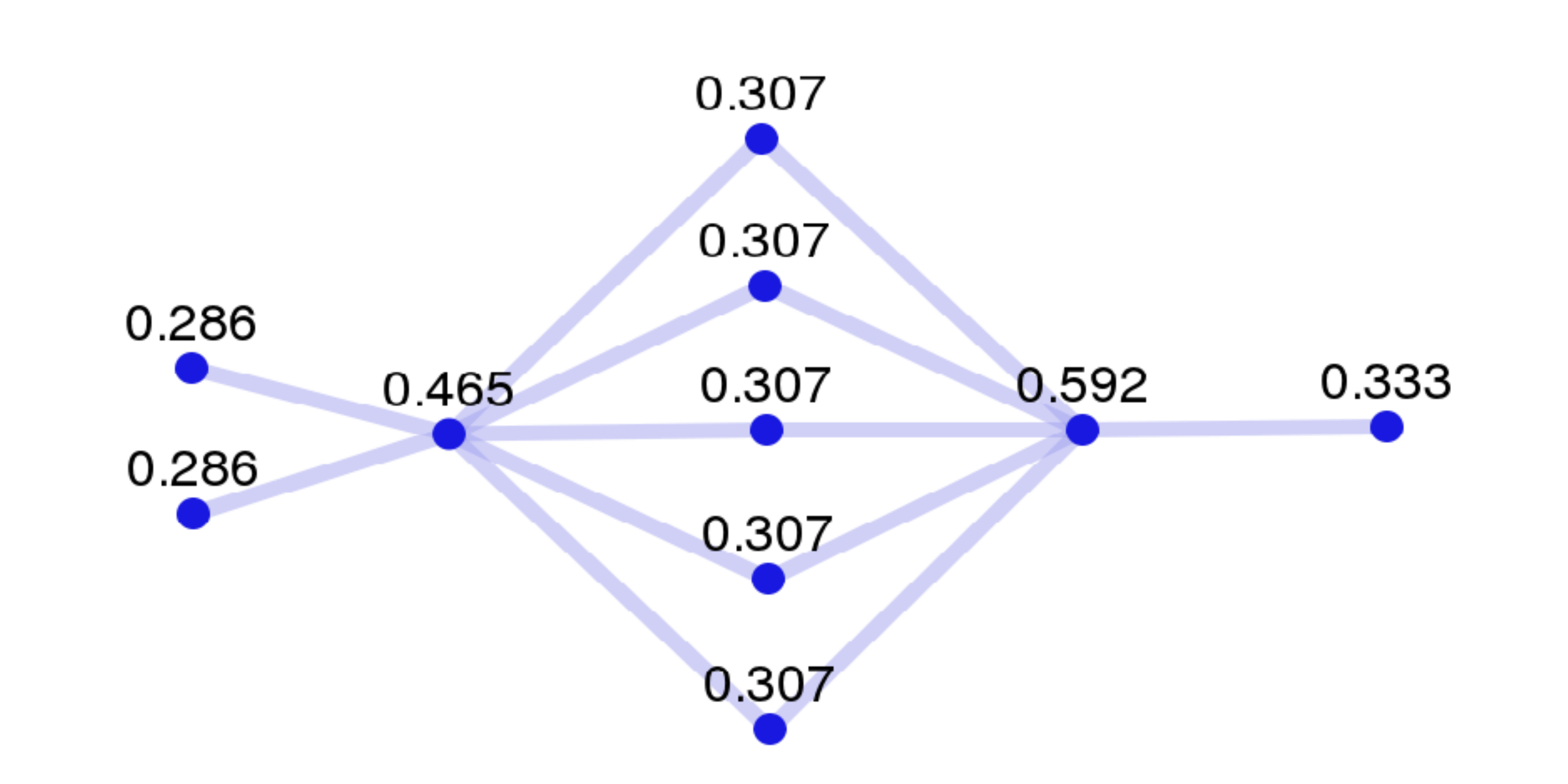}
\caption{Normalized curvature $\K^{nor}_{G,\cdot}(\infty)$\label{figure:cur_nor}}
\end{minipage}
\hfill
\begin{minipage}[t]{0.48\linewidth}
\centering
\includegraphics[width=.85\textwidth]{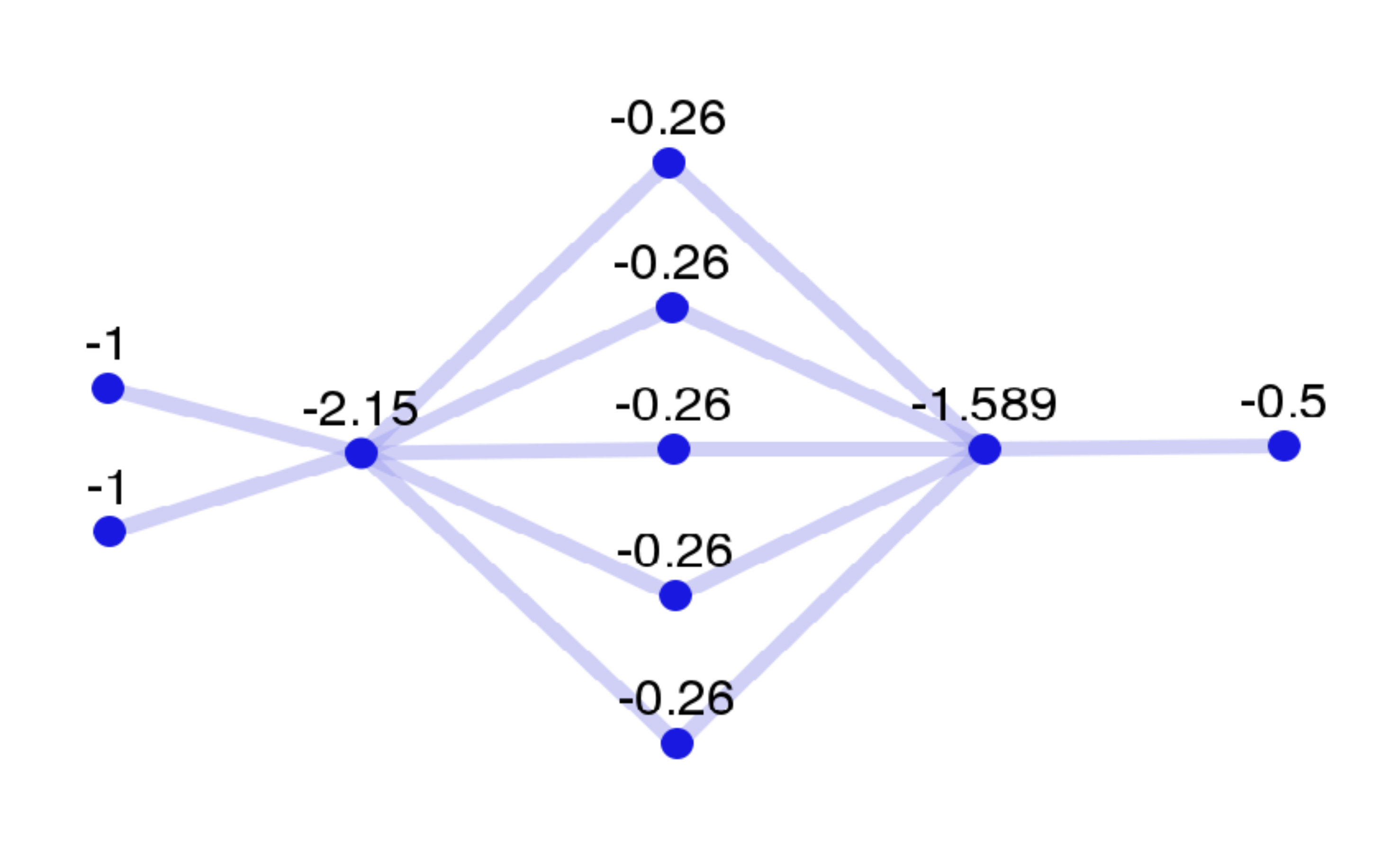}
\caption{Non-normalized curvature $\K_{G,\cdot}(\infty)$\label{figure:cur_non}}
\end{minipage}
\end{figure}

\section*{Acknowledgements}
This work was supported by the EPSRC Grant EP/K016687/1 "Topology, Geometry and Laplacians of Simplicial Complexes". We are very grateful to George W. Stagg for collaborating with DC on Python program and web-application
\begin{center}
http://teggers.eu/graph/
\end{center}
for calculation of Bakry-\'Emery curvature $CD(K,\infty)$ on
graphs. We also like to thank Anna Felikson for helpful discussions
about Coxeter groups, and Boba Hua and 
Jim Portegies for raising the
question whether curvature functions are concave.


%
%
%
%

\end{document}